\documentclass[11pt,reqno]{amsart}
\usepackage{amsmath,amssymb,graphicx,mathrsfs,amsopn,stmaryrd,amsthm,color,bm}
\usepackage[dvipsnames]{xcolor}
\usepackage[colorlinks=true,citecolor=blue,linkcolor=blue]{hyperref}
\usepackage[final]{showlabels}
\usepackage[english]{babel}
\usepackage[mathscr]{euscript}
\usepackage{geometry}
\usepackage[alphabetic,nobysame]{amsrefs}
\renewcommand{\MR}[1]{} \renewcommand{\PrintDOI}[1]{}
\newcommand{\arxiv}[1]{\href{http://arxiv.org/abs/#1}{\sf arXiv:\nolinkurl{#1}}}
\usepackage[
            cal=boondoxo]   
           {mathalpha}


\usepackage{dsfont}

\allowdisplaybreaks

\numberwithin{equation}{section}
\newtheorem{thm}{Theorem}[section]
\newtheorem{prop}[thm]{Proposition}
\newtheorem{lem}[thm]{Lemma}
\newtheorem{cor}[thm]{Corollary}
\newtheorem{conj}[thm]{Conjecture}
\theoremstyle{definition}
\newtheorem{eg}[thm]{Example}
\newtheorem{dfn}[thm]{Definition}
\theoremstyle{remark}

\newcommand{\beq}{\begin{equation}}
\newcommand{\eeq}{\end{equation}}
\newcommand{\be}{\begin{equation*}}
\newcommand{\ee}{\end{equation*}}
\newcommand{\bs}{{\bm\fks}}

\newcommand{\C}{\mathbb{C}}

\newcommand{\Z}{\mathbb{Z}}

\newcommand{\mc}{\mathcal}

\newcommand{\sfb}{\mathsf{b}}
\newcommand{\sfd}{\mathsf{D}}
\newcommand{\sfe}{\mathsf{E}}
\newcommand{\sff}{\mathsf{F}}

\newcommand{\sfy}{\mathsf{y}}

\newcommand{\gl}{\mathfrak{gl}}
\newcommand{\h}{\mathfrak{h}}

\newcommand{\fkS}{\mathfrak{S}}

\newcommand{\fks}{\mathfrak{s}}

\newcommand{\End}{\mathrm{End}}

\newcommand{\scrB}{\mathscr{B}}
\newcommand{\scrW}{\mathscr{W}}
\newcommand{\bC}{\mathbb{C}}
\newcommand{\bR}{\mathbb{R}}
\newcommand{\bZ}{\mathbb{Z}}

\newcommand{\tl}{\tilde}
\newcommand{\gge}{\geqslant}
\newcommand{\lle}{\leqslant}
\newcommand{\la}{\lambda}
\newcommand{\La}{\Lambda}

\newcommand{\bla}{\bm\lambda}

\newcommand{\glMN}{\mathfrak{gl}_{m|n}^{\bm s}}
\newcommand{\UglMN}{\mathrm{U}(\mathfrak{gl}_{m|n}^\s)}

\newcommand{\YglMN}{\mathscr{Y}_{\bm s}}

\newcommand{\YMN}{{\mathscr{Y}_{\bm s}}}
\newcommand{\YMNtl}{{\mathscr{Y}_{\tl{\bm s}}}}
\newcommand{\YMNp}{{\mathscr{Y}_{\bm s}^+}}
\newcommand{\BMN}{{\mathscr{B}_{\bm s,\bm \ve}}}
\newcommand{\BMNtl}{{\mathscr{B}_{\tl{\bm s},\tl{\bm \ve}}}}

\newcommand{\wt}{\widehat}
\newcommand{\ka}{\varkappa}
\newcommand{\ve}{\varepsilon}
\newcommand{\qedd}{\tag*{$\square$}}

\newcommand{\s}{{\bm s}}

\newcommand{\ovs}{{\overline{\bm s}}}
\newcommand{\ove}{{\overline{\bm \varepsilon}}}
\newcommand{\uns}{{\underline{\bm s}}}
\newcommand{\une}{{\underline{\bm \varepsilon}}}

\begin{document}
\pagestyle{myheadings}
\setcounter{page}{1}

\title[Twisted super Yangians of type AIII]{Twisted super Yangians of type AIII and their representations}

\author{Kang Lu}
\address{Department of Mathematics, University of Virginia, Charlottesville, VA 22903, USA}
\email{kang.lu@virginia.edu}

\begin{abstract}
We study the super analogue of the Molev-Ragoucy reflection algebras, which we call twisted super Yangians of type AIII, and classify their finite-dimensional irreducible representations. These superalgebras are coideal subalgebras of the super Yangian $\mathscr{Y}(\gl_{m|n})$ and are associated with symmetric pairs of type AIII in Cartan's classification.
We establish the Schur-Weyl type duality between degenerate affine Hecke algebras of type BC and twisted super Yangians.

		\medskip


\end{abstract}

\maketitle
\setcounter{tocdepth}{1}
\tableofcontents

\thispagestyle{empty}


\section{Introduction}
Reflection algebras, introduced by Sklyanin in his seminal paper \cite{Sklyanin1988Boundary}, are pivotal in constructing the commutative Bethe subalgebra and ensuring integrability of quantum integrable systems with boundary conditions. These algebras, inspired by Cherednik's scattering theory \cite{Cherednik1984Factorizing} for factorized particles on the half-line, form the foundation for various studies.

In \cite{Molev2002reflection}, Molev and Ragoucy studied a family of reflection algebras $\mathscr B_{\bm\ve}$, whose relations are described in terms of reflection equation and a certain unitary condition, and classified their finite-dimensional irreducible representations. These reflection algebras can also be called twisted Yangians of type AIII as they are coideal subalgebras of the Yangian $\mathscr Y(\gl_n)$ and deformations of the fixed point subalgebra of $\mathrm U(\gl_{n}[x])$ associated to symmetric pair of type AIII, see \S\ref{sec:properties}. The twisted Yangians depend on a sequence $\bm\ve=(\ve_1,\ve_2,\cdots,\ve_n)$, where $\ve_i=\pm 1$, and for different $\bm\ve$ the $\mathscr B_{\bm\ve}$ might not be isomorphic.

These twisted Yangians were further investigated by Chen, Guay and Ma in \cite{Chen2014twisted}. They related the twisted Yangians (in R-matrix presentation) with another family of twisted Yangians introduced by MacKay \cite{MacKay2002rational} (in Drinfeld's original presentation). A Drinfeld functor from the category of modules over degenerate affine Hecke algebras of type BC (dAHA) to the category of modules over twisted Yangians were constructed. It turns out the Drinfeld functor is an equivalence of categories under certain conditions, similar to the usual Schur-Weyl duality. Moreover, the Drinfeld functor sends a finite-dimensional irreducible module over dAHA to either zero space or a finite-dimensional irreducible module over twisted Yangians.

In the present article, we shall study the supersymmetric generalization of $\mathscr B_{\bm\ve}$, that are twisted super Yangians of type AIII. The twisted super Yangians $\mathscr B_{\s,\bm\ve}$ are coideal subalgebras of the super Yangian $\mathscr Y(\gl_{m|n}^\s)$ that depends on sequences of parity sequence $\s=(s_1,s_2,\cdots,s_{m+n})$ and $\bm\ve=(\ve_1,\ve_2,\cdots,\ve_{m+n})$, where $s_i,\ve_i=\pm 1$ for $1\lle i\lle m+n$ and $1$ appears in $\bm s$ exactly $m$ times. The new sequence $\s$ corresponds to the Dynkin diagram we choose for the associated general linear Lie superalgebra $\gl_{m|n}$. When $\s$ satisfies $s_i=1$ for $1\lle i\lle m$ and $s_i=-1$ otherwise, we call $\s$ the standard parity sequence. The standard parity sequence corresponds to the standard Borel subalgebra of $\gl_{m|n}$.

The twisted super Yangians appear previously (under the name \emph{reflection superalgebras}) in the study of analytical and nested algebraic Bethe ansatz \cite{Ragoucy2007analytical,Belliard2009nested} for quantum integrable models (open spin chains) with symmetry described by twisted super Yangians. For the case of the standard parity sequence $\s$ and a specific $\bm\ve$, they computed the highest weight of twisted super Yangian for a highest weight vector of super Yangians. These superalgebras were also recently studied in \cite{kettle2023orthosymplectic}, where some partial results of this paper were obtained, and in \cite{bagnoli2023double}, where a double version of twisted super Yangian is introduced and studied. Note that in \cite{kettle2023orthosymplectic,bagnoli2023double}, the author deals with twisted super Yangians associated with the standard parity sequence $\s$ and a specific $\bm\ve$ while ours are arbitrary\footnote{For the classical limit, the treatment is the same for twisted super Yangians associated to different $\s$ and $\bm\ve$, but the representation theory does rely on $\s$ and $\bm\ve$.}.

Our primary objective is to initiate the study of highest weight representation theory for twisted super Yangians, classify their finite-dimensional irreducible representations, and establish a Schur-Weyl type duality between the twisted super Yangians and degenerate affine Hecke algebras of type B/C, analogous to \cite{Molev2002reflection,Chen2014twisted}. We deal with twisted super Yangians associated to arbitrary parity sequence $\s$ and $\bm\ve$. Our main methods are similar strategy to \cite{Molev1998finite,Molev2002reflection,Chen2014twisted}, cf. also \cite{Guay2017representationsI}. Since $\s$ and $\bm\ve$ are both arbitrary, the calculations become more complicated than that in \cite{Molev2002reflection,Chen2014twisted}. We need to put extra effort to correctly insert the necessary sign factors $\s$ and $\bm\ve$.

Finite-dimensional irreducible representations of super Yangians were classified by Zhang \cite{Zhang1995reps,Zhang1996super} for the standard parity sequence. A complete and concrete description of criteria for an irreducible $\mathscr Y(\gl_{m|n}^\s)$-module (for arbitrary $\s$) being finite-dimensional is not available, though such a criteria can be obtained recursively using the odd reflections \cite{Molev2022odd,Lu2022note}. Consequently, we only have classification of finite-dimensional irreducible $\mathscr B_{\s,\bm\ve}$-modules for the cases (1) arbitrary $\bm\ve$ when $n=0,1$ and (2) the standard parity sequence $\s$ when the occurrence of $i$ such that $\ve_i\ne \ve_{i+1}$  is at most 1.

There are also twisted Yangians of types AI and AII introduced by Olshanski \cite{Olshanski1992twisted} and of types BCD introduced by Guay and Regelskis \cite{Guay2016twisted} via R-matrix presentation, whose representation theory are studied in \cite{Molev1998finite,Guay2017representationsI}. Another family of twisted Yangians associated to general symmetric pairs were introduced by MacKay \cite{MacKay2002rational} in terms of Drinfeld's J-symbols. More recently, together with Wang and Zhang, we introduced another family of twisted Yangians for symmetric pairs of split and quasi-split types in Drinfeld's new presentation, \cite{Lu2023drinfeld,Lu2024drinfeld,Lu2024affine}. The Drinfeld's new presentation of twisted Yangians allows one to define the shifted twisted Yangians and their truncations. Such shifted twisted Yangians are closely related to Slodowy slices and finite W-algebras of classical types, and fixed point loci of affine Grassmannian slices, see \cite{Brown2009twisted,Tappeiner2024shifted,Lu2025shifted-W,Lu2025shifted-G} and cf. \cite{BK06,KWWY14,BFN19}. The isomorphism between the twisted Yangians beyond type A in R-matrix and Drinfeld presentations remains unproven, see \cite{Lu2023drinfeld,Lu2024isomorphism,Lu2024drinfeld}, offering an interesting avenue for future research. In the sequel work \cite{Lu2024twisted}, we extend the joint work \cite{Lu2024drinfeld} with Zhang, introducing a family of twisted super Yangians in Drinfeld presentations that are isomorphic to (quasi-split type) twisted super Yangians here for arbitrary ``symmetric" parity sequences. Finally, we remark that a family of twisted super Yangians combining types AI and AII were introduced in \cite{Briot2003twisted} where some partial results about their representation theory were obtained. These twisted super Yangians and their quantum Berezinians were recently investigated in \cite{Lin2024from,Lin2025Ber}. 

This article is organized in the following fashion. Section \ref{sec:super yangian} revisits basic properties of the super Yangian $\mathscr Y(\gl_{m|n}^\s)$. Section \ref{sec:twisted-super-yangians} delves into twisted super Yangians and their properties. Section \ref{sec:reps} explores highest weight representation theory and tensor product structures for twisted super Yangians. Section \ref{sec:rank1} classifies finite-dimensional irreducible representations for rank 1, while Section \ref{sec:classification} extends this classification to higher ranks in key cases. Finally, Section \ref{sec:schur-weyl} establishes a Schur-Weyl type duality between degenerate affine Hecke algebras of type BC and twisted super Yangians.

\medskip

{\bf Acknowledgments.} 
The author is partially supported by NSF grants DMS-2001351 and DMS--2401351, both awarded to Weiqiang Wang. 

\section{Super Yangian}\label{sec:super yangian}
Throughout the paper, we work over $\bC$.
\subsection{General linear Lie superalgebras}\label{sec glmn}
In this subsection, we recall the basics of the general linear Lie superalgebra $\glMN$, see e.g. \cite{Cheng2009dualities} for more detail.

A \textit{vector superspace} $W = W_{\bar 0}\oplus W_{\bar 1}$ is a $\bZ_2$-graded vector space. We call elements of $W_{\bar 0}$ \textit{even} and elements of
$W_{\bar 1}$ \textit{odd}. We write $\vert w\vert \in\{\bar 0,\bar 1\}$ for the parity of a homogeneous element $w\in W$. Set $(-1)^{\bar 0}=1$ and $(-1)^{\bar 1}=-1$.

Fix $m,n\in \bZ_{\gge 0}$ and set $\ka=m+n$. Denote by $S_{m\vert n}$ the set of all sequences $\bm s=(s_{1},s_2,\dots,s_{\ka})$ where $s_i\in\{\pm1\}$ and $1$ occurs exactly $m$ times. Elements of $S_{m\vert n}$ are called \textit{parity sequences}. The parity sequence of the form $\bm{s_0}=(1,\dots,1,-1,\dots,-1)$ is the \textit{standard parity sequence}.

Fix a parity sequence $\s\in S_{m\vert n}$ and define $\vert i\vert \in \bZ_2$ for $1\lle i \lle \ka$ by $s_i=(-1)^{\vert i\vert }$.

The Lie superalgebra $\glMN$ is generated by elements $e_{ij}^\s$, $1\lle i,j\lle \ka$, with the supercommutator relations
\[
[e^\s_{ij},e^\s_{kl}]=\delta_{jk}e^\s_{il}-(-1)^{(\vert i\vert +\vert j\vert )(\vert k\vert +\vert l\vert )}\delta_{il}e^\s_{kj},
\]
where the parity of $e_{ij}^\s$ is $\vert i\vert +\vert j\vert $. In the following, we shall drop the superscript $\s$ when there is no confusion.

Denote by $\UglMN$ the universal enveloping superalgebra of $\glMN$. The superalgebra $\UglMN$ is a Hopf superalgebra with the coproduct given by $\Delta(x)=1\otimes x+x\otimes 1$ for all $x\in \glMN$.

The \textit{Cartan subalgebra $\h$} of $\glMN$ is spanned by $e_{ii}$, $1\lle i \lle \ka$. Let $\epsilon_i$, $1\lle i \lle \ka$, be a basis of $\h^*$ (the dual space of $\h$) such that $\epsilon_i(e_{jj})=\delta_{ij}$. There is a bilinear form $(\ ,\ )$ on $\h^*$ given by $(\epsilon_i,\epsilon_j)=s_i\delta_{ij}$. The \textit{root system $\bf{\Phi}$} is a subset of $\h^*$ given by
\[
{\bf \Phi}:=\{\epsilon_i-\epsilon_j~\vert ~1\lle i,j\lle \ka \text{ and }i\ne j\}.
\]
We call a root $\epsilon_i-\epsilon_j$ \textit{even} (resp. \textit{odd}) if $\vert i\vert =\vert j\vert $ (resp. $\vert i\vert \ne \vert j\vert $).

Set $\alpha_i:=\epsilon_i-\epsilon_{i+1}$ for $1\lle i \lle \ka$. Denote by
$$
{\bf P}:=\bigoplus_{1\lle i \lle \ka}\bZ \epsilon_i,\quad  {\bf Q}:=\bigoplus_{1\lle i < \ka}\bZ \alpha_i,\quad {\bf Q}_{\gge 0}:=\bigoplus_{1\lle i <\ka}\bZ_{\gge 0} \alpha_i$$ the \textit{weight lattice}, the \textit{root lattice}, and the \textit{cone of positive roots}, respectively. Define a partial ordering $\gge$ on $\h^*$: $\mu\gge \nu$ if $\mu-\nu\in {\bf Q}_{\gge 0}$. 

For a weight $\mu\in\h^*$, it is convenient to write $\mu$ as the $\ka$-tuple $(\mu_i)_{1\lle i\lle \ka}$ such that $\mu_i=\mu(e_{ii})$.

A module $M$ over a superalgebra $\mathcal A$ is a vector superspace $M$ with a homomorphism of superalgebras $\mathcal A\to \End(M)$. A $\glMN$-module is a module over $\mathrm{U}(\glMN)$. However, we shall not distinguish modules which only differ by a parity.

For a $\glMN$-module $M$, define the \textit{weight subspace of weight} $\mu$ by
\begin{equation}\label{eq:uweight-space}
(M)_{\mu}:=\{v\in M\ \vert \ e_{ii}v=\mu(e_{ii})v,\ 1\lle i \lle \ka\}.
\end{equation}
For a $\glMN$-module $M$ such that $(M)_{\mu}=0$ unless $\mu\in\bf Q$, we say that $M$ is \textit{$\bf Q$-graded}.

For a $\glMN$-module $M$, we call a vector $v\in M$ \textit{singular} if $e_{ij}v=0$ for $1\lle i<j\lle \ka$. We call a nonzero vector $v\in M$ a \textit{singular vector of weight} $\mu$ if $v$ satisfies
\[
e_{ii}v=\mu(e_{ii})v,\quad e_{jk}v=0,
\]
for $1\lle i\lle \ka$ and $1\lle j< k\lle \ka$. A nonzero vector $v\in (M)_{\mu}$ is {\it a highest (resp. lowest) weight vector} of $M$ if $(M)_{\nu}= 0$ unless $\mu-\nu\in\bf Q_{\gge 0}$ (resp. $\nu-\mu\in\bf Q_{\gge 0}$). Clearly, a highest weight vector is singular while a lowest weight vector $v$ satisfies $e_{ji}v=0$ for $1\lle i< j\lle \ka$.

Denote by $L(\mu)$ the irreducible $\glMN$-module generated by a
singular vector of weight $\mu$.

Let $V:=\bC^{m\vert n}$ be the vector superspace with a basis $v_i$, $1\lle i \lle \ka$, such that $\vert v_i\vert =\vert i\vert $. Let $E_{ij}\in\End(V)$ be the linear operators such that $E_{ij}v_k=\delta_{jk}v_i$. The map $\rho_V:\glMN\to \End(V),\ e_{ij}\mapsto E_{ij}$ defines a $\glMN$-module structure on $V$. As a $\glMN$-module, $V$ is isomorphic to $L({\epsilon_1})$. The vector $v_i$ has weight $\epsilon_i$. The highest weight vector is $v_1$ and the lowest weight vector is $v_{\ka}$. We call it the \textit{vector representation} of $\glMN$.

\subsection{Super Yangians}\label{sec rtt}
Fix a parity sequence $\bm s\in S_{m|n}$ and recall the definition of super Yangian $\YglMN:=\mathscr Y(\mathfrak{gl}_{m|n}^\s)$ from \cite{Nazarov1991berezinian}.

\begin{dfn}The super Yangian $\YglMN$ is the $\Z_2$-graded unital associative algebra over $\C$ with generators $\{t_{ij}^{(r)}\ |\ 1\lle i,j \lle\ka, \, r\gge 1\}$ and the defining relations are given by
\beq\label{eq:comm-series}
[t_{ij}(u),t_{kl}(v)]=\frac{(-1)^{|i||j|+|i||k|+|j||k|}}{u-v}(t_{kj}(u)t_{il}(v)-t_{kj}(v)t_{il}(u)).
\eeq
where
\[
t_{ij}(u)=\delta_{ij}+\sum_{r\gge 1}  t_{ij}^{(r)}u^{-r},
\]
and the generators $t_{ij}^{(r)}$ have parities $|i|+|j|$.
\end{dfn}
The super Yangian $\YglMN$ has the RTT presentation as follows. Define the rational R-matrix $R(u)\in \End(V \otimes  V)$ by $R(u)=1-\mathcal P/u$, where $\mc P\in \End(V \otimes V)$ is the super flip operator defined by
\[
\mathcal P=\sum_{i,j=1}^\ka s_jE_{ij}\otimes E_{ji}.
\]
The rational R-matrix satisfies the quantum Yang-Baxter equation
\beq\label{eq yang-baxter}
R_{12}(u-v)R_{13}(u)R_{23}(v)=R_{23}(v)R_{13}(u)R_{12}(u-v).
\eeq
Define the operator $T(u)\in \YglMN[[u^{-1}]]\otimes \End(V) $,
$$
T(u)=\sum_{i,j=1}^\ka (-1)^{|i||j|+|j|} t_{ij}(u)\otimes E_{ij}.
$$
Then defining relations \eqref{eq:comm-series} can be written as
\beq\label{eq:RTT}
R(u-v)T_1(u)T_2(v)=T_2(v)T_1(u)R(u-v).
\eeq

The super Yangian $\YglMN$ is a Hopf superalgebra with the coproduct
\beq\label{eq Hopf}
\Delta: t_{ij}(u)\mapsto \sum_{k=1}^\ka t_{ik}(u)\otimes t_{kj}(u),
\eeq
and the antipode $S:T(u)\to T^{-1}(u)$.

Define the series
$$
t_{ij}'(u)=\delta_{ij}+\sum_{r\gge 1} {t}_{ij}^{\prime( r)}u^{-r}
$$
by
\beq\label{eq:inverseT}
T^{-1}(u)=\sum_{i,j=1}^\ka (-1)^{|i||j|+|j|}  t_{ij}'(u)\otimes E_{ij}.
\eeq
Then
\beq\label{eq:T'-expression}
t_{ij}'(u)=\delta_{ij}+\sum_{k=1}^\infty (-1)^k\sum_{a_1,\cdots,a_{k-1}=1}^\ka t_{ia_1}^\circ(u)t_{a_1a_2}^\circ(u)\cdots t_{a_{k-1}j}^\circ(u),
\eeq
where $t_{ij}^\circ(u)=t_{ij}(u)-\delta_{ij}$. In particular, by taking the coefficient of $u^{-r}$, for $r\gge 1$, one obtains
\beq\label{eq:T-expression-comp}
t_{ij}^{\prime(r)}=\sum_{k=1}^r (-1)^k\sum_{a_1,\cdots,a_{k-1}=1}^\ka\sum_{r_1+\cdots+r_k=r}t_{ia_1}^{(r_1)}t_{a_1a_2}^{(r_2)}\cdots t_{a_{k-1}j}^{(r_k)},
\eeq
where $r_i$ for $1\lle i\lle k$ are positive integers.

By \eqref{eq:RTT}, one has
\beq\label{eq:T'RT}
\begin{split}
T_1^{-1}(-u)R(u+v)T_2(v)=T_2(v)R(u+v)T_1^{-1}(-u),\\
T_1(u)R(u+v)T_2^{-1}(-v)=T_2^{-1}(-v)R(u+v)T_1(u),
\end{split}
\eeq
and
\beq\label{eq:tt'}
(u-v)[t_{ij}(u),t_{kl}'(v)]=(-1)^{|i||j|+|i||k|+|j||k|}\Big(\delta_{kj}\sum_{s=1}^\ka t_{is}(u)t_{sl}'(v)-\delta_{il}\sum_{s=1}^{\ka}t_{ks}'(v)t_{sj}(u)\Big).
\eeq

For $z\in\C$ there exists an isomorphism of Hopf superalgebras,
\begin{align}
\tau_z:\YglMN\to\YglMN, \qquad t_{ij}(u)\mapsto t_{ij}(u-z).\label{eq tau z}
\end{align}

The universal enveloping superalgebra $\mathrm U(\glMN)$ is a Hopf subalgebra of $\YglMN$ via the embedding $e_{ij}\mapsto s_it_{ij}^{(1)}$. The left inverse of this embedding is the \emph{evaluation homomorphism} $\pi_{m|n}^\s: \YglMN\to \UglMN$ given by
\beq\label{eq:evaluation-map}
\pi_{m|n}^\s: t_{ij}(u)\mapsto \delta_{ij}+s_ie_{ij}u^{-1}.
\eeq

The evaluation homomorphism is a superalgebra homomorphism but not a Hopf superalgebra homomorphism.
For any $\glMN$-module $M$, it is naturally a $\YglMN$-module obtained by pulling back $M$ through the evaluation homomorphism $\pi_{m|n}$. We denote the corresponding $\YglMN$-module by the same letter $M$ and call it an \emph{evaluation module}.

The following standard PBW-type theorem for super Yangian $\YMN$ is known.
\begin{thm}[{\cite{Gow2007gauss,Peng2016parabolic}}]\label{thm:PBW}
Given any total ordering on the elements $t_{ij}^{(p)}$ for $1\lle i,j\lle \ka$ and $p\in\mathbb Z_{>0}$, the ordered monomials in these elements, containing no second or higher
order powers of the odd generators, form a basis of the super Yangian $\YMN$.
\end{thm}

Besides the antipode $S$, we also have the following anti-automorphisms of $\YMN$ defined by
\begin{align*}
&\mc t:\YMN\to\YMN,\quad t_{ij}(u)\mapsto (-1)^{|i||j|+|j|}t_{ji}(u),\\
&\mc n:\YMN\to\YMN,\quad t_{ij}(u)\mapsto t_{ij}(-u).
\end{align*}
Then the anti-automorphisms $S$, $\mc t$, and $\mc n$ of $\YMN$ pairwise commute, see e.g. \cite[Proposition 1.5]{Nazarov2020yangian}. Let $\Omega$ be the anti-automorphism of $\YMN$ given by
\beq\label{Omega}
\Omega=S\circ \mc t\circ \mc n,\qquad \Omega(t_{ij}(u))=(-1)^{|i||j|+|j|}t_{ji}'(-u).
\eeq
\subsection{Highest weight representations}
We first recall the results about the highest weight representations for $\YMN$ from \cite{Zhang1996super}.
\begin{dfn}
A representation $L$ of $\YMN$ is called highest $\ell_\s$-weight if there exists a nonzero vector $\xi\in L$ such that $L$ is generated by $\xi$ and $\xi$ satisfies
\beq\label{eq:highest-Y}
\begin{split}
   & t_{ij}(u)\xi=0,  \qquad \quad\quad \quad 1\lle i<j\lle \ka,\\
   &    t_{ii}(u)\xi=\la_i(u)\xi,\qquad\quad  1\lle i\lle \ka,
\end{split}
\eeq
where $\la_i(u)\in 1+u^{-1}\C[[u^{-1}]]$. The vector $\xi$ is called a {\it highest $\ell_\s$-weight vector} of $L$ and the tuple $\bla(u)=(\la_i(u))_{1\lle i\lle \ka}$ is the {\it highest $\ell_\s$-weight} of $L$.
\end{dfn}

Let $\bla(u)=(\la_i(u))_{1\lle i\lle \ka}$ be a $\ka$-tuple as above. Then there exists a unique, up to isomorphism, irreducible highest weight representation $L(\bla(u))$ with the highest weight $\bla(u)$. Any finite-dimensional irreducible representation of $\YMN$ is isomorphic to $L(\bla(u))$ for some $\bla(u)$. The criterion for $L(\bla(u))$ being finite-dimensional was classified in \cite{Zhang1996super} when $\s$ is the standard parity sequence.

\begin{thm}[\cite{Zhang1996super}]\label{thm:zhang}
If $\s$ is the standard parity sequence, then the irreducible $\YMN$-module $L(\bla(u))$ is finite-dimensional if and only if there exist monic polynomials $P_i(u)$, $1\lle i\lle \ka$, such that
\[
\frac{\la_i(u)}{\la_{i+1}(u)}=\frac{P_i(u+s_i)}{P_i(u)},\quad \frac{\la_m(u)}{\la_{m+1}(u)}=\frac{P_m(u)}{P_{\ka}(u)}, \quad 1\lle i\lle \ka \text{ and } i\ne m,
\]
and $\deg P_m=\deg P_{\ka}$.
\end{thm}

A criterion for $L(\bla(u))$ to be finite-dimensional with an arbitrary parity sequence $\s$ can be recursively deduced from Theorem \ref{thm:zhang} via the odd reflections of super Yangian, see \cite{Molev2022odd,Lu2022note}. However, a compact description of such a criterion for an arbitrary parity sequence $\s$ is not available.

Regard $\UglMN$ as a subalgebra of $\YglMN$, then we have $t_{ii}^{(1)}=s_ie_{ii}$. In particular, one assigns a $\glMN$-weight to an $\ell_\s$-weight via the map
\beq\label{varpi1}
\varpi: \big(1+u^{-1}\C[[u^{-1}]]\big)^{\ka}\to \h^*,\ \bm\la(u)\mapsto \varpi(\bm\la(u))\quad \text{ such that } \quad \varpi(\bm\la(u))(e_{ii})=s_i\la_{i,1},
\eeq
where $\la_{i,1}$ is the coefficients of $u^{-1}$ in $\la_i(u)$.

Given a $\YglMN$-module $L$, consider it as a $\glMN$-module and its $\glMN$-weight subspaces $(M)_{\mu}$, see \eqref{eq:uweight-space}.
\begin{lem}\label{lem:wt-change}
We have
\[
t_{ij}^{(r)}(M)_{\mu}\subset (M)_{\mu+\epsilon_i-\epsilon_j},\qquad t_{ij}^{\prime (r)}(M)_{\mu}\subset (M)_{\mu+\epsilon_i-\epsilon_j},
\]
for $1\lle i, j\lle \ka$, and $r\in\bZ_{>0}$.
\end{lem}
\begin{proof}
By \eqref{eq:comm-series} and \eqref{eq:tt'}, we have
\beq\label{eq:1st-node-tt}
[t_{ij}^{(1)},t_{kl}(u)]=(-1)^{|i||j|+|i||k|+|j||k|}\big(\delta_{kj}t_{il}(u)-\delta_{il}t_{kj}(u)\big),
\eeq
\beq\label{eq:1st-node-tt'}
[t_{ij}^{(1)},t_{kl}'(u)]=(-1)^{|i||j|+|i||k|+|j||k|}\big(\delta_{kj}t_{il}'(u)-\delta_{il}t_{kj}'(u)\big).
\eeq
Note that $t_{ij}^{(1)}$ is identified with $s_ie_{ij}$, then the lemma follows from the above equations by a direct computation.
\end{proof}

Thus, we have the following corollary of Theorem \ref{thm:PBW} and Lemma \ref{lem:wt-change}.
\begin{cor}\label{cor:hwrt}
If $L$ is a $\YMN$-module of highest $\ell_\s$-weight $\bla(u)$, then $L$ has a $\glMN$-weight subspace decomposition. Moreover, its highest $\ell_\s$-weight vector has weight $\varpi(\bla(u))$ and the other weight vectors have weights that are strictly smaller (with respect to $\gge$ defined in \S \ref{sec glmn}) than $\varpi(\bla(u))$.\qed
\end{cor}

Let $\YMNp$ be the left ideal of $\YMN$ generated by all the coefficients of $t_{ij}(u)$ with $1\lle i<j\lle \ka$. We write $X\doteq X'$ if $X-X'\in \YMNp$. Clearly, if $\xi$ is a highest $\ell_\s$-weight vector of $\YMN$ and $X\doteq X'$, then $X\xi=X'\xi$.

\begin{prop}[\cite{Ragoucy2007analytical,Belliard2009nested}]\label{prop:t'-l-weight}
If $\xi$ is a highest $\ell_\s$-weight vector of highest $\ell_\s$-weight $\bla(u)$ in a representation $L$ of $\YMN$, then
\beq\label{eq:t'-annih}
\begin{split}
   & t'_{ij}(u)\xi=0,  \qquad \quad\quad \quad 1\lle i<j\lle \ka,\\
   & t'_{ii}(u)\xi=\la_i'(u)\xi,\qquad\quad  1\lle i\lle \ka,
\end{split}
\eeq
for certain $\la'_i(u)\in 1+u^{-1}\C[[u^{-1}]]$. (The formal series $\la'_i(u)$ will be determined later.)
\end{prop}
\begin{proof}
Let $1\lle i<j\lle \ka$. By \eqref{eq:tt'}, for any $1\lle k\lle \ka$, we have
\[
(-1)^{|i||j|+|i||k|+|j||k|}[t_{kj}(u),t_{ik}'(v)]\doteq -\frac{1}{u-v}\sum_{s=j}^{\ka}t_{is}'(v)t_{sj}(u).
\]
Expanding $(u-v)^{-1}$ as $\sum_{r=0}^\infty v^ru^{-r-1}$ and take the coefficients of $u^{-1}v^{-p}$ and $u^{-2}v^{-p}$, we have
\beq\label{prop:kill-1}
(-1)^{|i||j|+|i||k|+|j||k|}[t_{kj}^{(1)},t_{ik}^{\prime (p)}]\doteq -t_{ij}^{\prime (p)},
\eeq
\beq\label{prop:kill-2}
(-1)^{|i||j|+|i||k|+|j||k|}[t_{kj}^{(2)},t_{ik}^{\prime (p)}]\doteq -t_{ij}^{\prime (p+1)}-t_{ij}^{\prime (p)}t_{jj}^{(1)}-\sum_{s=j+1}^\ka t_{is}^{\prime (p)}t_{sj}^{(1)}.
\eeq
We prove $t_{ij}^{\prime(p)}\xi=0$ for all $1\lle i<j\lle \ka$ by induction on $p$.

The base case is clear because it is immediate from \eqref{eq:T-expression-comp} that $t_{ij}^{\prime(1)}=-t_{ij}^{(1)}$. Suppose now that $t_{ij}^{\prime(p)}\xi=0$. It follows from \eqref{prop:kill-1} and the induction hypothesis that
\beq\label{prop:kill-3}
t_{is}^{\prime (p)}t_{sj}^{(1)}\xi=0,\qquad j<s\lle \ka.
\eeq
Note that $\xi$ is an eigenvector of $t_{jj}^{(1)}$ and $t_{jj}^{(2)}$, we have
\beq\label{prop:kill-4}
[t_{jj}^{(2)},t_{ij}^{\prime(p)}]\xi=0,\qquad  t_{ij}^{\prime (p)}t_{jj}^{(1)}\xi=0.
\eeq
Setting $k=j$ in \eqref{prop:kill-2} and applying \eqref{prop:kill-2} to $\xi$, we immediately obtain $t_{ij}^{\prime(p+1)}\xi=0$ from \eqref{prop:kill-3} and \eqref{prop:kill-4}. Thus by induction, we have $t_{ij}^{(r)}\xi=0$ for all $1\lle i<j\lle \ka$ and $r\in \Z_{>0}$.

Since $\glMN$ can be regarded as a subalgebra of $\YMN$, the $\YMN$-module $L$ is hence a $\glMN$-module and has the weight decomposition. The vector $\xi$ has the weight $\varpi(\bla)$. By Corollary \ref{cor:hwrt}, $(L)_{\varpi(\bla)}$ is of dimension $1$ and all other weights appearing in $L$ are smaller than $\varpi(\bla)$. It follows from Lemma \ref{lem:wt-change} that $t_{jj}'(u)$ preserves $(L)_{\varpi(\bla)}$ and hence preserves $\xi$. Therefore, $t_{jj}'(u)\xi=\la_i'(u)\xi$ for some $\la_i'(u)\in 1+u^{-1}\C[[u^{-1}]]$.
\end{proof}

By the same strategy, we have the following lemma.

\begin{lem}\label{lem:tia-kill}
Let $\xi$ be a highest $\ell_\s$-weight vector. If $1\lle i <j\lle \ka$ and $1\lle c\lle a\lle \ka$,  then we have $t_{ia}(u)t_{cj}'(v)\xi= 0$. Similarly, if $1\lle i\lle \ka$ and $1\lle c <a \lle \ka$, then $t_{ia}(u)t_{ci}'(v)\xi =0$.
\end{lem}
\begin{proof}
First we consider the case when $a>c$. Then by \eqref{eq:tt'}, we have
\beq\label{eq:notsure-001}
[t_{ia}(u),t_{cj}'(v)]\xi=0.
\eeq
If $c < j$, it is clear from Proposition \ref{prop:t'-l-weight} that $t_{ia}(u)t_{cj}'(v)\xi= 0$. If $c\gge j$, then $a>c \gge j >i$, $t_{cj}'(v)t_{ia}(u)\xi=0$. It follows from \eqref{eq:notsure-001} that $t_{ia}(u)t_{cj}'(v)\xi= 0$.

Then we consider the case when $a=c$. If $a <j$, then $t_{ia}(u)t_{aj}'(v)\xi= 0$ by Proposition \ref{prop:t'-l-weight}. If $a\gge j$, by \eqref{eq:tt'}, we have
$$
s_a(u-v)[t_{ia}(u),t_{aj}'(v)]\xi=\sum_{c=1}^\ka t_{ic}(u)t_{cj}'(v)\xi.
$$
Note that the right hand side is independent of $a$. By setting $a=j$ and using Proposition \ref{prop:t'-l-weight}, we find that $$\sum_{c=1}^\ka t_{ic}(u)t_{cj}'(v)\xi=0.$$Hence we always have $[t_{ia}(u),t_{aj}'(v)]\xi=0$. Then again by Proposition \ref{prop:t'-l-weight},
\[
t_{ia}(u)t_{aj}'(v)\xi =[t_{ia}(u),t_{aj}'(v)]\xi + (-1)^{(|i|+|a|)(|a|+|j|)}t_{aj}'(v)t_{ia}(u)\xi =0,
\]
as $i<j\lle a$. 

Then we prove the second statement. If $c<i$, then the statement follows from Proposition \ref{prop:t'-l-weight}. Now suppose that $c\gge i$. By \eqref{eq:tt'} and the first statement, we have
\[
(-1)^{|i||a|+|i||c|+|a||c|}(u-v)[t_{ia}(u),t_{ci}'(v)]\xi=-\sum_{k=1}^\ka t_{ck}'(v)t_{ka}(u)\xi=0.
\]
Since $a>c\gge i$, we have $t'_{ci}(v)t_{ia}(u)\xi=0$. It follows from the above equation that $t_{ia}(u)t_{ci}'(v)\xi=0$, completing the proof.
\end{proof}

The following proposition was proved in \cite{Ragoucy2007analytical,Belliard2009nested} for the standard parity sequence. The strategy in \cite{Ragoucy2007analytical,Belliard2009nested} does not work in general for arbitrary parity sequences.

Let $\rho_k=\sum_{a=k}^\ka s_a$ for $1\lle k\lle \ka$. By convention, $\rho_{\ka+1}=0$.
\begin{prop}\label{prop:highest-weight-inver}
Let $\xi$ be a highest $\ell_\s$-weight vector of highest $\ell_\s$-weight $\bla(u)$. Suppose $\la_i'(u)$ is defined as in \eqref{eq:t'-annih}, then
\beq\label{eq:la-prime}
\la_i'(u)=\frac{1}{\la_{i}(u+\rho_{i+1})}\prod_{k=i+1}^\ka \frac{\la_k(u+\rho_k)}{\la_{k}(u+\rho_{k+1})}.
\eeq
\end{prop}
\begin{proof}
    For a given parity sequence $\s=(s_1,s_2,\cdots,s_\ka)\in S_{m|n}$, set $\bm{\fks}=(s_{\ka},s_{\ka-1},\cdots,s_2,s_1)$. To distinguish generating series for super Yangians of different parity sequences, we shall write $t_{ij}^{\s}(u)$, $t_{ij}^{\bm \fks}(u)$, etc. It is also convenient to identify an operator $\sum_{i,j=1}^\ka (-1)^{|i||j|+|j|}a_{ij}\otimes E_{ij}$ in $\YMN[[u^{-1}]]\otimes \End(V)$ with the matrix $(a_{ij})_{i,j=1}^\ka$. Then the extra sign ensures that the product of two matrices can still be calculated in the usual way.

Recall the Gauss decomposition of super Yangian $\YMN$, see \cite{Gow2007gauss,Peng2016parabolic}. Let $\sfe_{ij}^\s(u)$, $\sff_{ji}^\s(u)$, $\sfd_{k}^\s(u)$, where $1\lle i<j\lle \ka$ and $1\lle k\lle \ka$, be defined by the Gauss decomposition,
\begin{align*}
t_{ii}^\s(u)&=\sfd_i^\s(u)+\sum_{k<i}\sff_{ik}^\s(u)\sfd_k^\s(u)\sfe_{ki}^\s(u),\\
t_{ij}^\s(u)&=\sfd_i^\s(u)\sfe_{ij}^\s(u)+\sum_{k<i}\sff_{ik}^\s(u)\sfd_k^\s(u)\sfe_{kj}^\s(u),\\
t_{ji}^\s(u)&=\sff_{ji}^\s(u)\sfd_i^\s(u)+\sum_{k<i}\sff_{jk}^\s(u)\sfd_k^\s(u)\sfe_{ki}^\s(u).
\end{align*}
Similarly, one can define $\sfe_{ij}^\bs(u)$, $\sff_{ji}^\bs(u)$, $\sfd_{k}^\bs(u)$.

Let $t_{ij}^{\prime\s}(u)$ correspond to $t_{ij}'(u)$ in $\YMN$. Define similarly $\sfe_{ij}^{\prime\s}(u)$, $\sff_{ji}^{\prime\s}(u)$, $\sfd_{k}^{\prime\s}(u):=\sfd_{k}^\s(u)^{-1}$, for $1\lle i<j\lle \ka$ and $1\lle k\lle \ka$, by
\begin{align*}
&t_{ii}^{\prime\s}(u)=\sfd_i^\s(u)^{-1}+\sum_{k>i}\sfe_{ik}^{\prime\s}(u)\sfd_k^\s(u)^{-1}\sff_{ki}^{\prime\s}(u),\\
&t_{ij}^{\prime\s}(u)=\sfe_{ij}^{\prime\s}(u)\sfd_j^\s(u)^{-1}+\sum_{k>j}\sfe_{ik}^{\prime\s}(u)\sfd_k^\s(u)^{-1}\sff_{kj}^{\prime\s}(u),\\
&t_{ji}^{\prime\s}(u)=\sfd_j^{\s}(u)^{-1}\sff_{ji}^{\prime\s}(u)+\sum_{k>j}\sfe_{jk}^{\prime\s}(u)\sfd_k^\s(u)^{-1}\sff_{ki}^{\prime\s}(u).
\end{align*}
Then
\beq\label{eq:app:e-f-inv}
\begin{split}
&\sfe_{ij}^{\prime\s}(u)=\sum_{i=i_0<i_1<\cdots<i_r=j}(-1)^r \sfe_{i_0i_1}^\s(u)\sfe_{i_1i_2}^{\s}(u)\cdots  \sfe_{i_{r-1}i_r}^{\s}(u),\\
&\sff_{ji}^{\prime\s}(u)=\sum_{i=i_0<i_1<\cdots<i_r=j}(-1)^r \sff_{i_{r}i_{r-1}}^\s(u)\sff_{i_{r-1}i_{r-2}}^{\s}(u)\cdots  \sff_{i_{1}i_0}^{\s}(u).
\end{split}
\eeq

There exists an isomorphism between $\mathscr{Y}_{\bs}$  and $\YMN$ given by the map
\beq\label{eq:identifi}
t_{\ka+1-j,\ka+1-i}^\bs(u)\to (-1)^{|i||j|+|j|}t_{ij}^{\prime\s}(u),\quad 1\lle i,j\lle \ka,
\eeq
where the parities $|i|$ and $|j|$ are determined by the parity sequence $\s$.

We shall identify $t_{ij}^\bs(u)$ with $t_{\ka+1-j,\ka+1-i}^{\prime\s}(u)$ with certain signs as in \eqref{eq:identifi}.
With this identification, when $\ka=3$, one has
\[
T^{\bs}(u)=\begin{pmatrix}
(\sfd_3^{\s})^{-1} & \sfe_{23}^{\prime\s}(\sfd_3^{\s})^{-1} & \sfe_{13}^{\prime\s}(\sfd_3^{\s})^{-1}\\
(\sfd_3^{\s})^{-1}\sff_{32}^{\prime\s} & (\sfd_2^{\s})^{-1}+\sfe_{23}^{\prime\s}(\sfd_3^{\s})^{-1}\sff_{32}^{\prime\s} & \sfe_{12}^{\prime\s}(\sfd_2^{\s})^{-1}+\sfe_{13}^{\prime\s}(\sfd_3^{\s})^{-1}\sff_{32}^{\prime\s}\\
(\sfd_3^{\s})^{-1}\sff_{31}^{\prime\s} & (\sfd_2^{\s})^{-1}\sff_{21}^{\prime\s}+\sfe_{23}^{\prime\s}(\sfd_3^{\s})^{-1}\sff_{31}^{\prime\s} & (\sfd_1^{\s})^{-1}+\sfe_{12}^{\prime\s}(\sfd_2^{\s})^{-1}\sff_{21}^{\prime\s}+\sfe_{13}^{\prime\s}(\sfd_3^{\s})^{-1}\sff_{31}^{\prime\s}
\end{pmatrix},
\]
cf. \cite[equaltion (B.4)]{Liashyk2019new}. Here we drop the spectral parameter $u$ and the signs for brevity.

Under the identification above, one proves similarly to  \cite[Theorem 4.2]{Liashyk2019new} for the super Yangian $\YMN$ that
\begin{align}
&\sfe_{\ka+1-j,\ka+1-i}^\bs(u)=(-1)^{|i||j|+|j|}\sfe_{ij}^{\prime\s}(u+\rho_{j}), & 1\lle i<j\lle \ka,\nonumber\\
&\sff_{\ka+1-i,\ka+1-j}^\bs(u)\, =\, (-1)^{|i||j|+|i|}\sff_{ji}^{\prime\s}(u+\rho_{j}),& 1\lle i<j\lle \ka,\nonumber\\
&\sfd_{\ka+1-k}^{\bs}(u) = \frac{1}{\sfd_{k}^{\s}(u+\rho_{k+1})}\prod_{a=k+1}^\ka\frac{\sfd_{a}^{\s}(u+\rho_{a})}{\sfd_{a}^{\s}(u+\rho_{a+1})}, & 1\lle k\lle \ka.\label{eq:cartan-current-new}
\end{align}
Now we are ready to prove Proposition \ref{prop:highest-weight-inver}.

It is well known that for a highest $\ell_\s$-weight vector $v$ of highest weight $\bla(u)$, we have
\beq\label{eq:app:1}
\sfe_{ij}^\s(u)v=0,\qquad t_{ii}^\s(u)v=\sfd_i^\s(u)v=\la_i(u)v,
\eeq
see e.g. \cite[Section 2.5]{Lu2022note}. By Gauss decomposition,
$$
t_{ii}^\bs(u)=\sfd_i^\bs(u)+\sum_{k<i}\sff_{ik}^\bs(u)\sfd_k^\bs(u)\sfe_{ki}^\bs(u),
$$
it follows from \eqref{eq:app:e-f-inv} and \eqref{eq:app:1} that $t_{ii}^\bs(u)v=\sfd_i^\bs(u)v$. Therefore, by \eqref{eq:cartan-current-new} and \eqref{eq:app:1} that
\[
t_{ii}^{\prime\s}(u)v=t_{\ka+1-i,\ka+1-i}^\bs(u)v=\sfd_{\ka+1-i}^\bs(u)v=\frac{1}{\la_{i}(u+\rho_{i+1})}\prod_{k=i+1}^\ka \frac{\la_k(u+\rho_k)}{\la_{k}(u+\rho_{k+1})}v,
\]
completing the proof of Proposition \ref{prop:highest-weight-inver}.
\end{proof}

\section{Twisted super Yangian of type AIII}\label{sec:twisted-super-yangians}
\subsection{Definition}
Fix a sequence of integers $\bm\ve=(\ve_1,\ve_2,\cdots,\ve_\ka)$, where $\ve_i \in\{\pm 1\}$.
 Denote by $G^{\bm\ve}$ the diagonal $\ka\times \ka$ (super)matrix
\beq\label{eq:G}
G^{\bm\ve}=\mathrm{diag}(\ve_1,\ve_2,\cdots,\ve_\ka).
\eeq
The matrix $G^{\bm\ve}$ satisfies the reflection equation
\beq\label{eq:reflectG}
R(u-v)G^{\bm\ve}_1R(u+v)G^{\bm\ve}_2=G^{\bm\ve}_2R(u+v)G^{\bm\ve}_1R(u-v).
\eeq

\begin{dfn}[{\cite{Molev2002reflection,Ragoucy2007analytical,Belliard2009nested}}]
The twisted super Yangian  of type AIII, denoted by $\mathscr B_{\bm s,\bm\ve}$, is a $\Z_2$-graded unital associative algebra over $\C$ with generators $\{b_{ij}^{(r)}\ |\ 1\lle i,j\lle \ka, \, r\gge 1\}$ and defining relations given by
\beq\label{eq:comm-series b}
\begin{split}
[b_{ij}(u),b_{kl}(v)]= & \ \frac{(-1)^{|i||j|+|i||k|+|j||k|}}{u-v}(b_{kj}(u)b_{il}(v)-b_{kj}(v)b_{il}(u))\\
&\ + \frac{(-1)^{|i||j|+|i||k|+|j||k|}}{u+v}\Big(\delta_{kj}\sum_{a=1}^\ka b_{ia}(u)b_{al}(v)-\delta_{il}\sum_{a=1}^\ka b_{ka}(v)b_{aj}(u)\Big)\\
&\ - \frac{1}{u^2-v^2}\delta_{ij}\Big(\sum_{a=1}^\ka b_{ka}(u)b_{al}(v)-\sum_{a=1}^\ka b_{ka}(v)b_{al}(u)\Big)
\end{split}
\eeq
and the unitary condition
\beq\label{eq:unitary-series}
\sum_{a=1}^\ka b_{ia}(u)b_{aj}(-u)=\delta_{ij},
\eeq
where
\[
b_{ij}(u)=\delta_{ij}\ve_i+\sum_{r\gge 1} b_{ij}^{(r)}u^{-r},
\]
and the generators $b_{ij}^{(r)}$ have the parity $|i|+|j|$.
\end{dfn}
Define the operator $B(u)\in \mathcal \mathscr B_{\bm s,\bm\ve}[[u^{-1}]] \otimes\End(V)$,
\beq\label{eq matrix notation b}
B(u)=\sum_{i,j=1}^\ka (-1)^{|i||j|+|j|} b_{ij}(u)\otimes E_{ij}.
\eeq
Then the defining relations of $\mathcal \mathscr B_{\bm s,\bm\ve}$ can also be written as the reflection equation
\beq\label{eq:comm-generators b}
R(u-v)B_1(u)R(u+v)B_2(v)=B_2(v)R(u+v)B_1(u)R(u-v)
\eeq
and
\beq\label{eq:unitary}
B(u)B(-u)=1,
\eeq
where $1$ is the identity matrix.

We shall also use the algebra $\widetilde{\mathscr B}_{\bm s,\bm\ve}$ defined in the same way as $\mathscr B_{\bm s,\bm\ve}$ but with the unitary condition \eqref{eq:unitary-series} omitted. Since there are no other types in this paper, we shall simply call $\mathscr B_{\bm s,\bm\ve}$ and $\widetilde{\mathscr B}_{\bm s,\bm\ve}$ {\it twisted super Yangian} and {\it extended twisted super Yangian}, respectively.

The extended twisted super Yangian (reflection superalgebra) previously appeared in \cite{Ragoucy2007analytical,Belliard2009nested} on the study of Bethe ansatz for open spin chains with diagonal boundary conditions. Certain properties on $\widetilde{\mathscr B}_{\bm s,\bm\ve}$ has been obtained in \cite{Ragoucy2007analytical,Belliard2009nested}. We shall reproduce some of them.

\begin{prop}
In the extended twisted super Yangian $\widetilde{\mathscr B}_{\bm s,\bm\ve}$, the product $B(u)B(-u)$ is a scalar matrix
\beq\label{eq:bb-central}
B(u)B(-u)=f(u)1,
\eeq
where $f(u)$ is a series in $u^{-2}$ whose coefficients are central in $\widetilde{\mathscr B}_{\bm s,\bm\ve}$.
\end{prop}
\begin{proof}
The proof is parallel to that of \cite[Proposition 2.1]{Molev2002reflection}. Multiplying both sides of \eqref{eq:comm-series b} by $u^2-v^2$ and set $v=-u$, one has
\beq
\begin{split}
(-1)^{|i||j|+|i||k|+|j||k|} 2u\Big(\delta_{kj} &\ \sum_{a=1}^\ka b_{ia}(u)b_{al}(v)-\delta_{il}\sum_{a=1}^\ka b_{ka}(v)b_{aj}(u)\Big)\\
&\ = \delta_{ij}\Big(\sum_{a=1}^\ka b_{ka}(u)b_{al}(v)-\sum_{a=1}^\ka b_{ka}(v)b_{al}(u)\Big).
\end{split}
\eeq
By taking suitable indices $i,j,k,l$, one obtains that
$$
B(u)B(-u)=B(-u)B(u)
$$
and the matrix is indeed a scalar matrix. Therefore, \eqref{eq:bb-central} holds and in particular $f(u)$ is a series in $u^{-2}$ as $B(u)B(-u)$ is even.

Multiplying both side of \eqref{eq:comm-generators b} by $B_2(-v)$ from the right, we have
\begin{align*}
R(u-v)B_1(u)R(u+v)f(v)\  \ =\ \ &\ B_2(v)R(u+v)B_1(u)R(u-v)B_2(-v)\\
\stackrel{\eqref{eq:comm-generators b}}{=}\ & \  B_2(v)B_2(-v)R(u-v)B_1(u)R(u+v)\\
=\ \  &\ f(v)R(u-v)B_1(u)R(u+v).
\end{align*}
Therefore, the coefficients of $f(v)$ are central in $\widetilde{\mathscr B}_{\bm s,\bm\ve}$.
\end{proof}

Let $\mc h(u)\in1+u^{-1}\C[[u^{-1}]]$ be such that $\mc h(u)\mc h(-u)=1$. There is an automorphism $\mc M_{\mc h(u)}$ defined by
\beq\label{mhu}
\mc M_{\mc h(u)}:\BMN\to \BMN,\quad B(u)\mapsto \mc h(u)B(u).
\eeq

\subsection{Basic properties of twisted super Yangian}\label{sec:properties}
In this subsection, we collect some basic properties of the twisted super Yangian $\mathscr B_{\bm s,\bm\ve}$.

For $\s=(s_i)_{1\lle i\lle \ka}$ and $\bm\ve=(\ve_i)_{1\lle i\lle \ka}$, define
\[
-\s=(-s_1,-s_2,\cdots,-s_{\ka}),\qquad -\bm\ve=(-\ve_1,-\ve_2,\cdots,-\ve_\ka).
\]
Then we have the following isomorphisms between twisted super Yangians,
\begin{align}
\BMN\to \mathscr B_{-\s,\bm\ve},\qquad b_{ij}(u)\mapsto b_{ij}(-u),\label{neghomo1}\\
\BMN\to \mathscr B_{\s,-\bm\ve},\qquad b_{ij}(u)\mapsto -b_{ij}(u).\label{neghomo2}
\end{align}
Moreover, the super-transposition defines an anti-homomorphism
\[
\BMN\to \BMN,\qquad b_{ij}(u)\mapsto (-1)^{|i||j|+|j|}b_{ji}(u).
\]

\begin{prop}\label{thm:embedding}
The mapping
\beq\label{eq:emd-b}
\varphi: B(u)\to T(u) G^{\bm\ve}T^{-1}(-u)
\eeq
defines an embedding which identify the twisted super Yangian $\mathscr B_{\bm s,\bm\ve}$ as a subalgebra of the super Yangian $\YMN$.
\end{prop}
\begin{proof}
The proof is essentially the same as that of \cite[Theorem 3.1]{Molev2002reflection}. We first check that $\varphi$ induces a superalgebra homomorphism which we again denote by $\varphi$ and then prove that this superalgebra homomorphism $\varphi$ is injective.

For brevity, we simply write $G$ for $G^{\bm\ve}$.

Set $S(u)= T(u) GT^{-1}(-u)$, then we immediately have
$$
S(u)S(-u)=T(u) GT^{-1}(-u)T(-u) GT^{-1}(u)=1
$$
which verifies the unitary condition \eqref{eq:unitary}.

On the other hand, we also have
\begin{align*}
R(u-v)S_1(u)R(u+v)S_2(v)  \ = \ &\ R(u-v)T_1(u)G_1T_1^{-1}(-u)R(u+v)T_2(v)G_2T_2^{-1}(-v)\\
\stackrel{\eqref{eq:T'RT}}{=}  & \
R(u-v)T_1(u)G_1T_2(v)R(u+v)T_1^{-1}(-u)G_2T_2^{-1}(-v)\\
\ =\ & \
R(u-v)T_1(u)T_2(v)G_1R(u+v)G_2T_1^{-1}(-u)T_2^{-1}(-v)\\
\stackrel{\eqref{eq:RTT}}{=}  & \
T_2(v)T_1(u)R(u-v)G_1R(u+v)G_2T_1^{-1}(-u)T_2^{-1}(-v)\\
\stackrel{\eqref{eq:reflectG}}{=}  & \
T_2(v)T_1(u)G_2R(u+v)G_1R(u-v)T_1^{-1}(-u)T_2^{-1}(-v)\\
\stackrel{\eqref{eq:RTT}}{=}  & \
T_2(v)T_1(u)G_2R(u+v)G_1T_2^{-1}(-v)T_1^{-1}(-u)R(u-v)\\
\ =\ & \
T_2(v)G_2T_1(u)R(u+v)T_2^{-1}(-v)G_1T_1^{-1}(-u)R(u-v)\\
\stackrel{\eqref{eq:T'RT}}{=}  & \
T_2(v)G_2T_2^{-1}(-v)R(u+v)T_1(u)G_1T_1^{-1}(-u)R(u-v)\\
\ =\ & \ S_2(v)R(u+v)S_1(u)R(u-v).
\end{align*}
Therefore, $S(u)$ also satisfies the reflection equation \eqref{eq:comm-generators b}.

Then we show that $\varphi$ is injective. Introduce the filtration on $\YMN$ defined by $\deg_1 t_{ij}^{(r)}=r$, see \cite{Gow2007gauss}, and a similar filtration on $\BMN$ by setting $\deg_1 b_{ij}^{(r)}=r$. Note that for the matrix elements of $S(u)$, we have
\beq\label{eq:embed-expl}
\mathtt s_{ij}(u)=\ve_i\delta_{ij}+\sum_{r>0} \mathtt s_{ij}^{(r)}u^{-r}=\sum_{a=1}^\ka \ve_at_{ia}(u)t_{aj}'(-u).
\eeq
It follows from \eqref{eq:T-expression-comp} that the degree of $\mathtt s_{ij}^{(r)}$ is at most $r$. Therefore $\varphi$ preserves the filtration and hence induces a homomorphism of the associated graded superalgebras
$$
\overline{\varphi}:\mathrm{gr}_1\BMN\to \mathrm{gr}_1\YMN.
$$
Denote by $\bar t_{ij}^{(r)}$ the image of $t_{ij}^{(r)}$ in the $r$-th component of $\mathrm{gr}_1\YMN$.

It is clear from \eqref{eq:comm-series} that $\mathrm{gr}_1\YMN$ is supercommutative, and moreover it follows from \cite[Theorem 1]{Gow2007gauss} that these elements $\bar t_{ij}^{(r)}$ are algebraically independent generators (in the super sense). It is also clear that $\mathrm{gr}_1\BMN$ is supercommutative. Denote by $\bar b_{ij}^{(r)}$ the image of $b_{ij}^{(r)}$ in the $r$-th component of $\mathrm{gr}_1\BMN$. Due to the unitary condition \eqref{eq:unitary-series}, the elements
\beq\label{eq:ind-generator-b}
\begin{split}
&\bar b_{ij}^{(2p-1)}, \qquad  \text{if }\ve_i=\ve_j,\\
&\bar b_{ij}^{(2p)},  \qquad\quad \text{if }\ve_i\ne\ve_j,
\end{split}
\eeq
for $1\lle i,j\lle \ka$ and $p\in\mathbb Z_{>0}$, generate the superalgebra $\mathrm{gr}_1 \BMN$. By \eqref{eq:embed-expl} and \eqref{eq:T-expression-comp}, we find that
\beq\label{eq:graded-generators}
\overline{\varphi}:\bar b_{ij}^{(r)}\mapsto ((-1)^{r-1}\ve_i+\ve_j)\bar t_{ij}^{(r)}+\cdots.
\eeq
Here $\cdots$ stands for a linear combination of monomials in $\bar t_{ab}^{(p)}$ with $p<r$ for various $1\lle a,b\lle \ka$. therefore, the elements in \eqref{eq:ind-generator-b} are algebraically independent, completing the proof.
\end{proof}

We immediately have the following PBW-type theorem for the twisted super Yangian $\BMN$.
\begin{cor}\label{cor:PBW}
Given any total ordering on the elements
\beq
\begin{split}
& b_{ij}^{(2p-1)}, \qquad  \text{if }\ve_i=\ve_j,\\
& b_{ij}^{(2p)},  \qquad\quad \text{if }\ve_i\ne\ve_j,
\end{split}
\eeq
for $1\lle i,j\lle \ka$ and $p\in\mathbb Z_{>0}$, the ordered monomials in these elements, containing no second or higher
order powers of the odd generators, form a basis of the twisted super Yangian $\BMN$.
\end{cor}

Thanks to Proposition \ref{thm:embedding}, the twisted super Yangian is identified with a subalgebra of $\YMN$ by identifying $b_{ij}(u)$ with $\mathtt s_{ij}(u)$. As the twisted super Yangians of types AI and AII, $\BMN$ is also a coideal subalgebra of $\YMN$.
\begin{prop}\label{prop:coproduct}
The subalgebra $\BMN$ is a left coideal subalgebra in $\YMN$,
\beq\label{eq:b-copro-in-t}
\Delta(b_{ij}(u))=\sum_{a,c=1}^\ka t_{ia}(u)t_{cj}'(-u)\otimes b_{ac}(u)(-1)^{(|c|+|j|)(|a|+|c|)}.
\eeq
\end{prop}
\begin{proof}
Note that $\Delta$ is a superalgebra homomorphism. One finds
$$
\Delta(t_{ij}'(u))=\sum_{a=1}^\ka t_{aj}'(u)\otimes t_{ia}'(u)(-1)^{(|a|+|j|)(|i|+|a|)}.
$$
Then the statement follows from a straightforward computation.
\end{proof}

Let $\theta$ be the involution of $\glMN$ sending $e_{ij}$ to $\ve_i\ve_j e_{ij}$ and $(\glMN)^\theta$ the fixed point Lie subalgebra of $\glMN$ under $\theta$. Note that $\theta$ depends on the diagonal matrix $G^{\bm\ve}$ which we shall not write explicitly. Then $(\glMN,(\glMN)^\theta)$ is a (super)symmetric pair of type AIII, cf. \cite{Shen2025quantum}. Write $\glMN=\mathfrak k\oplus\mathfrak p$ as the $(\pm 1)$-eigenspace decomposition with respect to $\theta$. In particular,
$$
\mathfrak k=(\glMN)^\theta\cong \gl_{m_1|n_1}\oplus \gl_{m_2|n_2},
$$
where
\begin{align*}
   &m_1=\#\{i\,|\, s_i=\ve_i=1,\, 1\lle i\lle \ka\},&n_1=\#\{i\,|\, -s_i=\ve_i=1,\, 1\lle i\lle \ka\},\\
   &m_2=\#\{i\,|\, s_i=-\ve_i=1,\, 1\lle i\lle \ka\},& n_2=\#\{i\,|\, s_i=\ve_i=-1,\, 1\lle i\lle \ka\}.
\end{align*}
Clearly, a basis of $\mathfrak k$ is given by all $e_{ij}$ for $1\lle i,j\lle\ka$ and $\ve_i=\ve_j$, while a basis of $\mathfrak p$ is given by all $e_{ij}$ for $1\lle i,j\lle\ka$ and $\ve_i\ne\ve_j$.

Extend the involution $\theta$ on $\glMN$ to $\hat\theta$ on $\glMN[x]$ by sending
$$
\hat\theta(gx^k)=\theta(g)(-x)^k,
$$
for $g\in \glMN$ and $k\in\Z_{\gge 0}$. Let $\glMN[x]^{\hat\theta}$ be the fixed point subalgebra of $\glMN[x]$ under $\hat\theta$. Then we have
\[
\glMN[x]^{\hat\theta}=(\mathfrak k\otimes \C[x^2]) \bigoplus (\mathfrak p\otimes x\C[x^2]).
\]

There is also another filtration of $\YMN$ defined by setting $\deg_2 t_{ij}^{(r)}=r-1$. It is well known \cite{Gow2007gauss} that the associated graded superalgebra $\mathrm{gr}_2\YMN$ is the universal enveloping superalgebra $\mathrm{U}(\gl_{m|n}[x])$ and the correspondence is given by
\beq\label{eq:griso}
\mathrm{U}(\gl_{m|n}[x])\to \mathrm{gr}_2\,\YMN, \qquad e_{ij} x^r\mapsto s_i\bar{t}_{ij}^{(r+1)}.
\eeq
Regard $\BMN$ as a subalgebra of $\YMN$ via Proposition \ref{thm:embedding} and denote by $\mathrm{gr}_2 \BMN$ the image of $\BMN$ in $\mathrm{gr}_2\YMN$. Alternatively, define a filtration on $\BMN$ given by $\deg_2 b_{ij}^{(r)}=r-1$, then $\mathrm{gr}_2\BMN$ is the associated graded superalgebra.
\begin{prop}\label{prop:limit}
The twisted super Yangian $\BMN$ is a deformation of $\mathrm{U}(\glMN[x]^{\hat\theta})$,
\[
\mathrm{gr}_2 \BMN\cong \mathrm{U}(\glMN[x]^{\hat\theta}).
\]
\end{prop}
\begin{proof}
For $r\in\Z_{\gge 0}$, let $\bar b_{ij}^{(r+1)}$ be the image of $b_{ij}^{(r+1)}$ in the $r$-th component of $\mathrm{gr}_2\,\BMN$. It follows from the proof of Proposition \ref{thm:embedding} that under the isomorphism \eqref{eq:griso},
\[
s_i((-1)^{r-1}\ve_i+\ve_j)e_{ij}x^{r-1}\mapsto ((-1)^{r-1}\ve_i+\ve_j)\bar{t}_{ij}^{(r)}=\bar b_{ij}^{(r)}.
\]
Note that $\mathrm{U}(\glMN[x]^{\hat\theta})$ (resp. $\mathrm{gr}_2\,\BMN$) is generated by $((-1)^{r-1}\ve_i+\ve_j)e_{ij}x^{r-1}$ (resp. $\bar b_{ij}^{(r)}$), for $1\lle i,j\lle \ka$ and $r\in\bZ_{>0}$, the proposition follows from the PBW theorem of Lie superalgebras and Corollary \ref{cor:PBW}.
\end{proof}

\section{Highest weight representations}\label{sec:reps}
In this section, we discuss the highest weight representations of the twisted super Yangian $\BMN$.
\subsection{Highest weight representations}
Similar to \cite{Molev2002reflection}, we define the highest weight representation of $\BMN$ as follows.
\begin{dfn}
A representation $V$ of $\BMN$ is called highest $\ell_{\s,\bm\ve}$-weight if there exists a nonzero vector $\eta\in V$ such that $V$ is generated by $\eta$ and $\eta$ satisfies
\beq\label{eq:highest-B}
\begin{split}
   & b_{ij}(u)\eta=0,  \qquad \quad\quad \quad 1\lle i<j\lle \ka,\\
   &    b_{ii}(u)\eta=\mu_i(u)\eta,\qquad\quad  1\lle i\lle \ka,
\end{split}
\eeq
where $\mu_i(u)\in \ve_i+u^{-1}\C[[u^{-1}]]$. The vector $\eta$ is called a {\it highest $\ell_{\s,\bm\ve}$-weight vector} of $V$ and the tuple $\bm \mu(u)=(\mu_i(u))_{1\lle i\lle \ka}$ is the {\it highest $\ell_{\s,\bm\ve}$-weight} of $V$.
\end{dfn}

\begin{eg}\label{eg:1-dim}
For any $\gamma\in\bC$, there exists a one-dimensional module $\bC_{\gamma}:=\bC \eta_{\gamma}$ generated by a highest $\ell_{\s,\bm\ve}$-weight vector $\eta_{\gamma}$  such that
\[
b_{ij}(u)\eta_\gamma = \delta_{ij}\frac{\ve_iu+\gamma}{u-\gamma}\eta_\gamma.\qedd
\]
\end{eg}

We have the following standard statements. By the relations \eqref{eq:comm-series b}, we have
\[
[b_{ij}^{(1)},b_{kl}(u)]=(-1)^{|i||j|+|i||k|+|j||k|}(\ve_i+\ve_j)(\delta_{kj}b_{il}(u)-\delta_{il}b_{kj}(u)).
\]
In particular, we have
\beq\label{eq:bii-weightb}
\big[s_i\ve_ib_{ii}^{(1)}/2,b_{kl}(u)\big]=\delta_{ki}b_{il}(u)-\delta_{il}b_{ki}(u).
\eeq
Therefore, the operators $s_i\ve_ib_{ii}^{(1)}/2$ are pairwise commuting. This also follows from the fact that under the identification \eqref{eq:embed-expl}, we have
\[
b_{ii}^{(1)}=2\ve_i t_{ii}^{(1)}=2s_i\ve_ie_{ii}.
\]

We say that a $\BMN$-module $V$ has a $\glMN$-weight subspace decomposition if it possesses a common eigenspace decomposition for the commuting operators $b_{ii}^{(1)}$, $1\lle i\lle \ka$. We say that a vector $v\in V$ has weight $\mc w=(\mc w_1,\dots,\mc w_\ka)$ if
\[
\frac{1}{2}s_i\ve_i b_{ii}^{(1)}v=\mc w_i v,\quad 1\lle i\lle \ka.
\]
Denote by $(V)_{\mc w}$ the weight subspace of $V$ with weight $\mc w$. The definition of weight subspace is compatible with \eqref{varpi1} if we consider a $\YMN$-module as a $\BMN$-module by restriction. For a highest $\ell_{\s,\bm \ve}$-weight $\bm\mu(u)$, we define a $\glMN$-weight $\varpi(\bm\mu(u))$, similar to \eqref{varpi1}, associated to  it by the rule
\beq\label{varpi2}
\varpi(\bm\mu(u))(e_{ii})=\frac{1}{2}s_i\ve_i\mu_{i,1},
\eeq
where $\mu_{i,1}$ is the coefficient of $u^{-1}$ in the series $\mu_i(u)$. Then a highest $\ell_{\s,\bm\ve}$-weight vector of $\glMN$-weight $\varpi(\bm\mu(u))$.

\begin{lem}\label{lem:wt-changeb}
We have
\[
b_{ij}^{(r)}(V)_{\mc w}\subset (V)_{\mc w+\epsilon_i-\epsilon_j},
\]
for $1\lle i, j\lle \ka$, and $r\in\bZ_{>0}$.
\end{lem}
\begin{proof}
The lemma follows from \eqref{eq:bii-weightb} by a direct computation.
\end{proof}

Thus, we have the following corollary of Corollary \ref{cor:PBW} and Lemma \ref{lem:wt-changeb}.
\begin{cor}\label{cor:wt-changeb}
If $V$ is a $\BMN$-module of highest $\ell_{\s,\bm\ve}$-weight $\bm\mu(u)$, then $V$ has a $\glMN$-weight subspace decomposition. Moreover, its highest $\ell_{\s,\bm\ve}$-weight vector has weight $\varpi(\bm\mu(u))$ and the other weight vectors have weights that are strictly smaller (with respect to $\gge$ defined in \S \ref{sec glmn}) than $\varpi(\bm\mu(u))$.\qed
\end{cor}

Let $V$ be a representation of $\BMN$. Set
$$
V^\circ=\{\eta\in V~|~b_{ij}(u)\eta=0,\ 1\lle i<j\lle \ka\}.
$$
\begin{lem}\label{lem:nontrivial}
If $V$ is a finite-dimensional representation of $\BMN$, then $V^\circ$ is nontrivial.
\end{lem}
\begin{proof}
Since the operators $s_i\ve_ib_{ii}^{(1)}/2$ are pairwise commuting and hence have at least a common eigenvector $\tl\eta\ne 0$ in $V$. Suppose $V^\circ=0$, then there exists an infinite sequence of nonzero vectors in $V$,
\[
\tl\eta,\quad b_{i_1j_1}^{(r_1)}\tl\eta,\quad b_{i_2j_2}^{(r_2)}b_{i_1j_1}^{(r_1)}\tl\eta,\quad \cdots,
\]
where $i_k<j_k$ and $r_k>0$ for all $k\in \Z_{>0}$. It follows from Lemma \ref{lem:wt-changeb} and Corollary \ref{cor:wt-changeb} that the above vectors have different $\glMN$-weights. Therefore, they must be linearly independent and hence we obtain a contradiction as $V$ is finite-dimensional, completing the proof.
\end{proof}

Throughout the paper, for $X,X'\in \BMN$, we shall write $X\equiv X'$ if $X-X'$ belongs to the left ideal of $\BMN$ generated by the coefficients of $b_{ij}(u)$ for $1\lle i<j\lle \ka$.

\begin{lem}\label{lem:invariant}
The space $V^\circ$ is invariant under the operators $b_{rr}(u)$, for $1\lle r\lle \ka$.
\end{lem}
\begin{proof}
We prove $b_{ij}(u)b_{rr}(v)\equiv 0$ for $1\lle i<j\lle \ka$ and $1\lle r\lle \ka$ by a reverse induction on $r$.

For the base case $r=\ka$, it is immediate from \eqref{eq:comm-series b} that $b_{ij}(u)b_{\ka \ka}(v)\equiv 0$ for $i<j<\ka$. Similarly, for $i<\ka$, we obtain
$$
b_{i\ka}(u)b_{\ka \ka}(v)\equiv \frac{(-1)^{|i||j|+|i||k|+|j||k|}}{u+v}b_{i\ka}(u)b_{\ka \ka}(v),
$$
which implies $b_{i\ka}(u)b_{\ka \ka}(u)\equiv 0$. Therefore, the base case is established.

Now let $r<\ka$. For $i<j$ and $i<k$, it follows from \eqref{eq:comm-series b} that
\beq\label{lem:commun-1}
s_j b_{ij}(u)b_{jk}(v)\equiv \frac{1}{u+v}\sum_{a=k}^{\ka} b_{ia}(u)b_{ak}(v).
\eeq
Note that the right hand side of \eqref{lem:commun-1} is independent of $j$. Therefore, for $j$ and $j'$ such that $i<j,j'$, we have
\beq\label{lem:commun-2}
s_j b_{ij}(u)b_{jk}(v)\equiv s_{j'} b_{ij'}(u)b_{j'k}(v).
\eeq
In the following, we always assume that $i<j$. We have four cases.

\begin{enumerate}
\item The case $i<r$ and $j\ne r$. It is straightforward from \eqref{eq:comm-series b} that $b_{ij}(u)b_{rr}(v)\equiv 0$.

\item The case $i<r$ and $j=r$. By \eqref{lem:commun-1} and \eqref{lem:commun-2}, we also have
\[
s_r b_{ir}(u)b_{rr}(v)\equiv \frac{s_r}{u+v}b_{ir}(u)b_{rr}(v)\sum_{a=r}^{\ka} s_a,
\]
which gives $b_{ir}(u)b_{rr}(v)\equiv 0$.

\item The case $r<i<j$. Using \eqref{eq:comm-series b} for $[b_{rr}(v),b_{ij}(u)]$, we have
\begin{align*}
b_{ij}(u)b_{rr}(v) \ \equiv \ & \ \frac{1}{u^2-v^2}\Big(\sum_{a=j}^\ka b_{ia}(u)b_{aj}(v)-\sum_{a=j}^\ka b_{ia}(v)b_{aj}(u)\Big)\\
\stackrel{\eqref{lem:commun-2}}{\equiv}  & \ \frac{s_j}{u^2-v^2}(b_{ij}(u)b_{jj}(v)-b_{ij}(v)b_{jj}(u))\sum_{a=j}^{\ka} s_a.
\end{align*}
Thus, $b_{ij}(u)b_{rr}(v) \equiv 0$ as $b_{ij}(u)b_{jj}(v)\equiv b_{ij}(v)b_{jj}(u)\equiv 0$ by induction hypothesis.

\item The case $r=i<j$. By \eqref{eq:comm-series b}, we have
\begin{align*}
s_rb_{rj}(u)b_{rr}(v)\equiv \frac{1}{u-v}(b_{rj}(u)b_{rr}(v)-b_{rj}(v)b_{rr}(u))-\frac{1}{u+v}\sum_{a=j}^\ka b_{ra}(v)b_{aj}(u).
\end{align*}
Note that by \eqref{lem:commun-2}, $s_a b_{ra}(v)b_{aj}(u)\equiv s_j b_{rj}(v)b_{jj}(u)\equiv 0$ for $j\lle a\lle \ka$ by induction hypothesis. We obtain that
\beq\label{lem:commun-3}
\frac{u-v-s_r}{u-v} b_{rj}(u)b_{rr}(v)+\frac{s_r}{u-v}b_{rj}(v)b_{rr}(u)\equiv 0.
\eeq
Interchanging $u$ and $v$, we also have
\beq\label{lem:commun-4}
-\frac{s_r}{u-v} b_{rj}(u)b_{rr}(v)+\frac{u-v+s_r}{u-v}b_{rj}(v)b_{rr}(u)\equiv 0.
\eeq
The system of equations \eqref{lem:commun-3} and \eqref{lem:commun-4} has only zero solution, therefore we conclude that $$b_{rj}(u)b_{rr}(v)\equiv 0.$$
\end{enumerate}
The proof now is complete.
\end{proof}

\begin{lem}\label{lem:commute}
All the operators $b_{rr}(u)$, $1\lle r\lle \ka$, on $V^\circ$ commute.
\end{lem}
\begin{proof}
For any $1\lle r\lle \ka$, it follows from \eqref{eq:comm-series b} that
\beq\label{lem:commun-5}
\Big(1-\frac{s_r}{u+v}\Big)[b_{rr}(u),b_{rr}(v)]\equiv \frac{s_r}{u+v}\iota_r(u,v),
\eeq
where
\beq\label{lem:commun-6}
\iota_r(u,v)=\sum_{a=r+1}^\ka(b_{ra}(u)b_{ar}(v)-b_{ra}(v)b_{ar}(u)).
\eeq
Again by \eqref{eq:comm-series b}, for $a > r$, we have
\[
b_{ra}(u)b_{ar}(v)\equiv \frac{s_a}{u-v}(b_{aa}(u)b_{rr}(v)-b_{aa}(v)b_{rr}(u))+\frac{s_a}{u+v}\Big(\sum_{c=r}^\ka b_{rc}(u)b_{cr}(v)-\sum_{c=a}^\ka b_{ac}(v)b_{ca}(u)\Big).
\]
Switching $u$ and $v$ and taking the difference, we obtain
\begin{align*}
 b_{ra}(u)b_{ar}(v)- &\  b_{ra}(v)b_{ar}(u ) \\ \equiv & \
 \frac{s_a}{u+v}\Big([b_{rr}(u),b_{rr}(v)]+[b_{aa}(u),b_{aa}(v)]+\iota_{r}(u,v)+\iota_a(u,v)\Big).
\end{align*}
Summing over $a$ from $r+1$ to $\ka$, we obtain
\beq\label{lem:commun-7}
(u+v-\rho_{r+1})\iota_{r}(u,v) \equiv \rho_{r+1}[b_{rr}(u),b_{rr}(v)]+\sum_{a=r+1}^\ka s_a\Big([b_{aa}(u),b_{aa}(v)]+\iota_a(u,v)\Big)
\eeq
Using \eqref{lem:commun-5} and \eqref{lem:commun-7}, one easily shows that $[b_{rr}(u),b_{rr}(v)]\equiv 0$ and $\iota_r(u,v)\equiv 0$ by a reverse induction on $r$. Therefore, if $i<r$, it follows from \eqref{eq:comm-series b} that
\[
[b_{ii}(u),b_{rr}(v)]\equiv -\frac{1}{u^2-v^2}\Big([b_{rr}(u),b_{rr}(v)]+\iota_r(u,v)\Big)\equiv 0.
\]
Hence we proved that all the operators $b_{rr}(u)$, $1\lle r\lle \ka$, on $V^\circ$ commute.
\end{proof}
Now we are ready to prove the main result of this subsection.
\begin{thm}
Every finite-dimensional irreducible representation $V$ of the twisted super Yangian $\BMN$ is a highest $\ell_{\s,\bm\ve}$-weight representation. Moreover, $V$ contains a unique (up to proportionality) highest $\ell_{\s,\bm\ve}$-weight vector.
\end{thm}
\begin{proof}
By Lemma \ref{lem:nontrivial}, $V^\circ$ is nontrivial. Hence it follows from Lemmas \ref{lem:invariant}, \ref{lem:commute} that $V^\circ$ contains a common eigenvector $\eta\ne 0$ for all operators $b_{rr}(u)$, $1\lle r\lle \ka$. Therefore, the vector $\eta$ satisfies \eqref{eq:highest-B} for some formal series $\mu_i(u)$.

Consider the submodule $\BMN\eta$ in $V$, as $V$ is irreducible, we conclude that $\BMN\eta$ coincides with $V$. The uniqueness of $\eta$ (up to proportionality) follows from Corollary \ref{cor:wt-changeb}.
\end{proof}

\subsection{Verma modules}
For any $\ka$-tuple $\bm\mu(u)=(\mu_i(u))_{1\lle i\lle \ka}$, where $\mu_i(u)\in \ve_i+u^{-1}\C[[u^{-1}]]$, denote by $M(\bm\mu(u))$ the quotient of $\BMN$ by the left ideal generated by all coefficients of the series $b_{ij}(u)$, for $1\lle i<j\lle \ka$, and $b_{ii}(u)-\mu_i(u)$, for $1\lle i\lle \ka$. We call $M(\bm\mu(u))$ the {\it Verma module} with highest $\ell_{\s,\bm\ve}$-weight $\bm\mu(u)$.

The Verma module $M(\bm\mu(u))$ may be trivial due to nontrivial relations. If $M(\bm\mu(u))$ is nontrivial, then denote by $V(\bm\mu(u))$ the unique irreducible quotient. Clearly, any irreducible highest $\ell_{\s,\bm\ve}$-weight module of $\BMN$ with highest $\ell_{\s,\bm\ve}$-weight $\bm\mu(u)$ is isomorphic to $V(\bm\mu(u))$.

In the rest of this subsection, we discuss the sufficient and necessary condition for $M(\bm\mu(u))$ being nontrivial.

Before stating and proving the theorem, we prepare a few lemmas that will be useful. For each $1\lle i\lle \ka$, set
\beq\label{eq:beta}
\beta_i(u,v)=\sum_{a=i}^\ka b_{ia}(u)b_{ai}(v).
\eeq
\begin{lem}
For $1\lle i<\ka$, if $u+v=\rho_{i+1}$, then we have
\begin{align*}
b_{ii}(u)b_{ii}(v)\ +&\ \frac{1}{u-v}\sum_{a=i+1}^\ka s_a(b_{aa}(u)b_{ii}(v)-b_{aa}(v)b_{ii}(u))\\
\equiv &\
b_{i+1,i+1}(u)b_{i+1,i+1}(v)+\frac{1}{u-v}\sum_{a=i+2}^\ka s_a(b_{aa}(u)b_{i+1,i+1}(v)-b_{aa}(v)b_{i+1,i+1}(u)).
\end{align*}
\end{lem}
\begin{proof}
For $a>i$, it follows from \eqref{eq:comm-series b} that
\[
b_{ia}(u)b_{ai}(v)\equiv \frac{s_a}{u-v}\big(b_{aa}(u)b_{ii}(v)-b_{aa}(v)b_{ii}(u)\big)+\frac{s_a}{u+v}\big(\beta_i(u,v)-\beta_a(v,u)\big).
\]
Summing over $a$ from $i+1$ to $\ka$, we have
\beq\label{eq:lem-nontrivial-1}
\begin{split}
    \frac{u+v-\rho_{i+1}}{u+v}\beta_i(u,v)\equiv &\ b_{ii}(u)b_{ii}(v)-\frac{1}{u+v}\sum_{a=i+1}^\ka s_a\beta_a(v,u)\\
    & \ +\frac{1}{u-v}\sum_{a=i+1}^\ka s_a\big(b_{aa}(u)b_{ii}(v)-b_{aa}(v)b_{ii}(u)\big).
\end{split}
\eeq
Interchanging $u$ and $v$ and taking the difference, we obtain
\beq\label{eq:lem-nontrivial-2}
\frac{u+v-\rho_{i+1}}{u+v}\big(\beta_i(u,v)-\beta_i(v,u)\big)\equiv \frac{1}{u+v}\sum_{a=i+1}^\ka s_a\big(\beta_a(u,v)-\beta_a(v,u)\big),
\eeq
where we also used that $b_{ii}(u)b_{ii}(v)\equiv b_{ii}(v)b_{ii}(u)$, see Lemma \ref{lem:commute}. Note that $\beta_\ka(u,v)\equiv \beta_\ka(v,u)$, one easily shows by a reverse induction on $i$ that $\beta_i(u,v)\equiv \beta_i(v,u)$ using \eqref{eq:lem-nontrivial-2}.

Applying \eqref{eq:lem-nontrivial-1} for $i$ and $i+1$ and using $\beta_i(u,v)\equiv \beta_i(v,u)$, one has
\beq\label{eq:lem-nontrivial-3}
\begin{split}
    \frac{u+v-\rho_{i+1}}{u+v}&\, \big(\beta_i(u,v)-\beta_{i+1}(u,v)\big)\\
    \equiv   b_{ii}(u)&\, b_{ii}(v) -b_{i+1,i+1}(u)b_{i+1,i+1}(v)+\frac{1}{u-v}\sum_{a=i+1}^\ka s_a\big(b_{aa}(u)b_{ii}(v)-b_{aa}(v)b_{ii}(u)\big)\\
    &\ \qquad\qquad \qquad \qquad  -\frac{1}{u-v}\sum_{a=i+2}^\ka s_a\big(b_{aa}(u)b_{i+1,i+1}(v)-b_{aa}(v)b_{i+1,i+1}(u)\big).
\end{split}
\eeq
Now the statement follows immediately if $u+v=\rho_{i+1}$.
\end{proof}
It is convenient to set, for $1\lle i\lle \ka$,
\beq\label{eq:def-tl-b}
\varpi_{i} :=\sum_{j=i}^{\ka}\ve_js_j,\qquad 
\tl b_{ii}(u) :=(2u-\rho_{i+1})b_{ii}(u)+\sum_{a=i+1}^\ka s_ab_{aa}(u).
\eeq
\begin{lem}\label{lem:b-in-t}
Regard $\BMN$ as a subalgebra of $\YMN$ as in Proposition \ref{thm:embedding}. Then we have
\beq
\tl b_{ii}(u)\approx \big(2\ve_i u-\ve_i\rho_{i+1}+\varpi_{i+1}\big)t_{ii}(u)t_{ii}'(-u),
\eeq
where $A \approx B$ if $A\xi =B\xi$ for any highest $\ell_\s$-weight vector $\xi$.
\end{lem}
\begin{proof}
For $1\lle i\lle \ka$, set
\[
\psi_i(u)=\sum_{a=i}^\ka t_{ia}(u)t_{ai}'(-u),\quad \wp_i(u)=\sum_{a=i}^\ka t_{ia}'(-u)t_{ai}(u).
\]
By \eqref{eq:tt'} and Proposition \ref{prop:t'-l-weight}, for $a>i$, we have
\beq\label{eq:inlemb-t}
t_{ia}(u)t_{ai}'(-u)\approx \frac{s_a}{2u}(\psi_i(u)-\wp_a(u)),\quad [t_{ii}(u),t_{ii}'(-u)]\approx\frac{s_i}{2u}(\psi_i(u)-\wp_i(u)).
\eeq
Therefore, we obtain
\beq\label{eq:psi}
\psi_i(u)\approx t_{ii}(u)t_{ii}'(-u)+\sum_{a=i+1}^\ka \frac{s_a}{2u}(\psi_i(u)-\wp_a(u)),\quad 1\lle i\lle \ka.
\eeq
Similarly, one has
\beq\label{eq:wp}
\wp_i(u)\approx t_{ii}'(-u)t_{ii}(u)+\sum_{a=i+1}^\ka\frac{s_a}{2u}(\wp_i(u)-\psi_a(u)),\quad 1\lle i\lle \ka.
\eeq
By Proposition \ref{prop:t'-l-weight}, we have $t_{ii}(u)t_{ii}'(-u)\approx t_{ii}'(-u)t_{ii}(u)$ for all $1\lle i\lle \ka$. Also note that it follows from \eqref{eq:inlemb-t} that $\psi_\ka(u)\approx \wp_\ka(u)$. By a reverse induction and using \eqref{eq:psi}, \eqref{eq:wp}, one proves that $\psi_i(u)\approx \wp_i(u)$ for all $1\lle i\lle \ka$. Hence, we have
\beq\label{eq:inlem-b-t-wp2}
(2u-\rho_{i+1})\wp_i(u)\approx 2ut_{ii}(u)t_{ii}'(-u)-\sum_{a=i+1}^\ka s_a \wp_a(u),\quad 1\lle i\lle \ka.
\eeq

On the other hand, it follows from \eqref{eq:embed-expl} and \eqref{eq:inlemb-t} that
\beq\label{eq:inlem-b-t-wp1}
b_{ii}(u)\approx \ve_i t_{ii}(u)t_{ii}'(-u) +\sum_{a=i+1}^\ka \frac{\ve_as_a}{2u}(\wp_i(u)-\wp_a(u)),\quad 1\lle i\lle \ka.
\eeq
Solving $b_{ii}(u)$ and $t_{ii}(u)t_{ii}'(-u)$ in terms of $\wp_{j}(u)$ from the system of equations \eqref{eq:inlem-b-t-wp2} and \eqref{eq:inlem-b-t-wp1}, it is not hard  to see by a brute force computation that
\[
(2u-\rho_{i+1})b_{ii}(u)+\sum_{a=i+1}^\ka s_ab_{aa}(u)\approx \big(2\ve_i u-\ve_i\rho_{i+1}+\varpi_{i+1}\big)t_{ii}(u)t_{ii}'(-u).\qedhere
\]
\end{proof}

Now we are ready to prove the main theorem of this subsection.
\begin{thm}\label{thmnontrivial}
The Verma module $M(\bm\mu(u))$ is nontrivial if and only if
\beq\label{thm:nontrivial-N}
\mu_\ka(u)\mu_{\ka}(-u)=1,
\eeq
and for $1\lle i < \ka$, the following conditions are satisfied
\beq\label{thm:nontrivial-i}
\tilde \mu_i(u)\tilde\mu_i(-u+\rho_{i+1})=\tilde \mu_{i+1}(u)\tilde\mu_{i+1}(-u+\rho_{i+1}),
\eeq
where
\beq\label{eq:mu-tilde}
\tilde \mu_i(u)=(2u-\rho_{i+1})\mu_i(u)+\sum_{a=i+1}^\ka s_a\mu_a(u).
\eeq
\end{thm}
\begin{proof}
We first show that conditions \eqref{thm:nontrivial-N} and \eqref{thm:nontrivial-i} are necessary. By the unitary condition \eqref{eq:unitary-series}, we have
\[
\sum_{a=1}^\ka b_{\ka a}(u)b_{a\ka}(-u)=1.
\]
Then \eqref{thm:nontrivial-N} follows from the above equation applied to the highest $\ell_{\s,\bm\ve}$-weight vector of $V(\bm\mu(u))$. Applying Lemma \ref{lem:nontrivial} to the highest weight vector of $V(\bm\mu(u))$, we get
\begin{align*}
\mu_i(u)\mu_i(v)\ +&\ \frac{1}{u-v}\sum_{a=i+1}^\ka s_a\big(\mu_a(u)\mu_i(v)-\mu_a(v)\mu_i(u)\big)\\
= &\
\mu_{i+1}(u)\mu_{i+1}(v)+\frac{1}{u-v}\sum_{a=i+2}^\ka s_a\big(\mu_{a}(u)\mu_{i+1}(v)-\mu_{a}(v)\mu_{i+1}(u)\big),
\end{align*}
where $u+v=\rho_{i+1}$. It is not hard to see that the above equation is equivalent to conditions \eqref{thm:nontrivial-i}.

Conversely, suppose the conditions \eqref{thm:nontrivial-N} and \eqref{thm:nontrivial-i} are satisfied. We shall show that there exists a highest $\ell_\s$-weight vector $\xi$ of highest $\ell_\s$-weight $\bla(u)$ such that $\xi$ is of highest $\ell_{\s,\bm\ve}$-weight $\bm\mu(u)$ by restricting a $\YMN$-module as a $\BMN$-module. This shall prove that the Verma module $M(\bm\mu(u))$ is non-trivial.

First, observe from Lemma \ref{lem:tia-kill} and \eqref{eq:embed-expl} that $b_{ij}(u)\xi =0$ for $1\lle i<j\lle \ka$.

We construct $\la_i(u)\in 1+u^{-1}\C[[u^{-1}]]$ inductively as follows. By \eqref{thm:nontrivial-N}, there exists $\la_\ka(u)\in 1+u^{-1}\C[[u^{-1}]]$ such that
$$
\mu_\ka(u)=\ve_\ka \la_\ka(u)/\la_\ka(-u).
$$
Suppose we already have $\la_j(u)$ for $i<j\lle \ka$. Define
\[
\La_i(u)=\frac{(2\ve_{i+1}u-\ve_{i+1}\rho_{i+2}+\varpi_{i+2})\tl \mu_i(u)\la_{i+1}(u)}{(2\ve_{i}u-\ve_{i}\rho_{i+1}+\varpi_{i+1})\tl \mu_{i+1}(u)\la_{i+1}(-u+\rho_{i+1})}.
\]
Note that if $\ve_i=\ve_{i+1}$, then we have
\beq\label{eq:good1}
2\ve_{i+1}u-\ve_{i+1}\rho_{i+2}+\varpi_{i+2}=2\ve_{i}u-\ve_{i}\rho_{i+1}+\varpi_{i+1};
\eeq
if $\ve_i=-\ve_{i+1}$, then we have
\beq\label{eq:good2}
2\ve_{i+1}u-\ve_{i+1}\rho_{i+2}+\varpi_{i+2}=2\ve_{i}(-u+\rho_{i+1})-\ve_{i}\rho_{i+1}+\varpi_{i+1}.
\eeq
Therefore, one easily checks that the condition \eqref{thm:nontrivial-i}
ensures that $\La_i(u)\La_{i}(-u+\rho_{i+1})=1$. Hence there exists $\la_i(u)\in 1+u^{-1}\C[[u^{-1}]]$ such that $\La_i(u)=\la_i(u)/\la_i(-u+\rho_{i+1})$.

With our choice of $\bla(u)$, we  have
\[
\frac{\tl \mu_i(u)}{\tl \mu_{i+1}(u)}=\frac{(2\ve_{i}u-\ve_{i}\rho_{i+1}+\varpi_{i+1})\la_i(u) \la_{i+1}(-u+\rho_{i+1})}{(2\ve_{i+1}u-\ve_{i+1}\rho_{i+2}+\varpi_{i+2})\la_{i+1}(u)\la_i(-u+\rho_{i+1})}
\]
and $\mu_\ka(u)=\ve_\ka\la_\ka(u)/\la_\ka(-u)$. By Propositions \ref{prop:t'-l-weight}, \ref{prop:highest-weight-inver} and Lemma \ref{lem:b-in-t}, one verifies that $\xi$ is  indeed of highest $\ell_{\s,\bm\ve}$-weight $\bm\mu(u)$.
\end{proof}
\subsection{Tensor product of representations}
Recall from Proposition \ref{prop:coproduct} that $\BMN$ is a left coideal subalgebra in $\YMN$. Given a $\YMN$-module $L$ and a $\BMN$-module $V$, then $L\otimes V$ is a $\BMN$-module given by the coproduct formula \eqref{eq:b-copro-in-t} in Proposition \ref{prop:coproduct}.

Let $L=L(\bla(u))$ be a highest $\ell_\s$-weight module over $\YMN$ with a highest $\ell_\s$-weight vector $\xi$. Let $V=V(\bm\mu(u))$ be a highest $\ell_{\s,\bm\ve}$-weight module over $\BMN$ with a highest $\ell_{\s,\bm\ve}$-weight vector $\eta$. We end this section by showing that $\xi\otimes \eta$ is a highest $\ell_{\s,\bm\ve}$-weight vector and calculating its highest $\ell_{\s,\bm\ve}$-weight.

Again we shall use the convenient notation $A\approx B$ if $A\xi =B\xi$.

\begin{lem}\label{lem:sum-T-}
For $1\lle i < a\lle \ka$, we have
\[
(2u-\rho_{i+1})t_{ia}(u)t_{ai}'(-u)+\sum_{c=i+1}^a s_ct_{ca}(u)t_{ac}'(-u)\approx s_at_{ii}(u)t_{ii}'(-u).
\]
\end{lem}
\begin{proof}
By \eqref{eq:inlemb-t}, one obtains
\begin{align*}
    (2u-\rho_{i+1})t_{ia}(u)t_{ai}'(-u)+& \ \sum_{c=i+1}^{a-1} s_ct_{ca}(u)t_{ac}'(-u)\\
\approx &\  \frac{s_a}{2u}\Big((2u-\rho_{i+1})(\wp_i(u)-\wp_a(u))+\sum_{c=i+1}^a s_c(\wp_c(u)-\wp_a(u))\Big)\\
\approx &\  \frac{s_a}{2u}\Big((2u-\rho_{i+1})\wp_i(u)-(2u-\rho_{a+1})\wp_a(u)+\sum_{c=i+1}^a s_c\wp_c(u)\Big)\\
\approx &\ \frac{s_a}{2u}\Big(2ut_{ii}(u)t_{ii}'(-u)-2ut_{aa}(u)t_{aa}'(-u)\Big),
\end{align*}
where the last equality follows from \eqref{eq:inlem-b-t-wp2}. Now the lemma follows.
\end{proof}

\begin{prop}\label{prop:tensor-product}
We have $b_{ij}(u)(\xi\otimes \eta)=0$, $1\lle i<j\lle \ka$, and
$$
\tl b_{ii}(u)(\xi\otimes \eta)=\la_i(u)\la_i'(-u)\tl \mu_i(u)(\xi\otimes \eta),\qquad 1\lle i\lle \ka,
$$
where $\la_i'(u)$ and $\tl\mu_i(u)$ are defined in \eqref{eq:la-prime} and \eqref{eq:mu-tilde}, respectively.
\end{prop}
\begin{proof}
It is easily seen from Lemma \ref{lem:tia-kill} and \eqref{eq:b-copro-in-t} that $b_{ij}(u)(\xi\otimes \eta)=0$.

Again we write $A\approx B$ if $A(\xi\otimes\eta)=B(\xi\otimes\eta)$. It follows from Lemma \ref{lem:tia-kill} that
\begin{align*}
\Delta(b_{ii}(u)) \approx \sum_{a=i}^\ka \big(t_{ia}(u)t_{ai}'(-u)\otimes b_{aa}(u)\big)(\xi\otimes\eta).
\end{align*}
Therefore,
\begin{align*}
\Delta(\tl b_{ii}(u)) \approx &\ (2u-\rho_{i+1}) \sum_{a=i}^\ka t_{ia}(u)t_{ai}'(-u)\otimes b_{aa}(u)+\sum_{a=i+1}^\ka s_a\sum_{c=a}^\ka t_{ac}(u)t_{ca}'(-u)\otimes b_{cc}(u)\\
= &\  \sum_{a=i}^\ka \Big((2u-\rho_{i+1})t_{ia}(u)t_{ai}'(-u)+\sum_{c=i+1}^a s_ct_{ca}(u)t_{ac}'(-u)\Big)\otimes b_{aa}(u)\\
\approx &\ (2u-\rho_{i+1})t_{ii}(u)t_{ii}'(-u)\otimes b_{ii}(u)+\sum_{a=i+1}^\ka s_a t_{ii}(u)t_{ii}'(-u)\otimes b_{aa}(u)\\
= &\ t_{ii}(u)t_{ii}'(-u)\otimes \Big((2u-\rho_{i+1})b_{ii}(u)+\sum_{a=i+1}^\ka s_a b_{aa}(u)\Big)=t_{ii}(u)t_{ii}'(-u)\otimes \tl b_{ii}(u).
\end{align*}
Here we applied Lemma \ref{lem:sum-T-} in the third equality. Now the statement follows.
\end{proof}
\begin{eg}\label{eg:tensor-1}
Recall the one dimensional $\BMN$-module $\bC_{\gamma}=\bC\eta_{\gamma}$ from Example \ref{eg:1-dim} for $\gamma\in\bC$. Let $L=L(\bla(u))$ be a highest $\ell_\s$-weight module over $\YMN$ with a highest $\ell_\s$-weight vector $\xi$. Then by Proposition \ref{prop:tensor-product} we have
\[
\tl{b}_{ii}(u)(\xi\otimes\eta_\gamma)=\frac{(2\ve_i u-\ve_i\rho_{i+1}+\varpi_{i+1}+2\gamma)u}{u-\gamma}\la_i(u)\la_{i}'(-u)(\xi\otimes\eta_\gamma),
\]
cf. Lemma \ref{lem:b-in-t}.\qed
\end{eg}

\section{Classifications in rank 1}\label{sec:rank1}
In this section, we investigate the finite-dimensional irreducible representations of twisted super Yangian of the small rank case $\ka=2$. We consider $\s=(s_1,s_2)$ and $\bm\ve=(\ve_1,\ve_2)$.
\subsection{Non-super case}
The case $\bm s=(1,1)$ has already been studied in \cite[Propositions 4.4, 4.5]{Molev2002reflection} via identifying $\BMN$ with (Olshanski's) twisted Yangians $\mathscr Y(\mathfrak{sp}_2)$ and $\mathscr Y(\mathfrak{so}_2)$ of types AI and AII.

\begin{prop}[{\cite{Molev2002reflection}}]\label{prop:iff-even}
Suppose $s_1=s_2$.
\begin{enumerate}
    \item If $\ve_1=\ve_2$, then the $\BMN$-module $V(\bm\mu(u))$ is finite-dimensional if and only if there exists a monic polynomial $P(u)$ such that $P(-u+2s_2)=P(u)$ and
\[
\frac{\tl \mu_1(u)}{\tl \mu_2(u)}=\frac{P(u+s_2)}{P(u)}.
\]
\item If $\ve_1\ne \ve_2$, then the $\BMN$-module $V(\bm\mu(u))$ is finite-dimensional if and only if there exist $\gamma\in \C$ and a monic polynomial $P(u)$ such that $P(-u+2s_2)=P(u)$, $P(\gamma)\ne 0$, and
\[
\frac{\tl \mu_1(u)}{\tl \mu_2(u)}=\frac{P(u+s_2)}{P(u)}\cdot \frac{\gamma -u}{\gamma+u-s_2}.
\]
In this case, the pair $(P(u),\gamma)$ is unique.
\end{enumerate}
\end{prop}
\begin{proof}
For the case $s_1=s_2=1$ and $\ve_1=1$, the statements are proved in \cite[Propositions 4.4]{Molev2002reflection}. To obtain other cases, one uses the isomorphisms \eqref{neghomo1}--\eqref{neghomo2}.
\end{proof}

Note that if $s_1=s_2$, then $\mathscr Y(\gl_\s)\cong \mathscr Y(\gl_2)$.

\begin{prop}[{\cite{Molev2002reflection}}]\label{prop:con-even}
Suppose $s_1=s_2$.
\begin{enumerate}
\item If $\ve_1=\ve_2$, then any finite-dimensional irreducible $\BMN$-module $V(\bm\mu(u))$ is isomorphic to the restriction of a $\mathscr Y(\gl_\s)$-module $L$, where $L$ is some finite-dimensional irreducible $\mathscr Y(\gl_\s)$-module.
\item If $\ve_1\ne \ve_2$, then any finite-dimensional irreducible $\BMN$-module $V(\bm\mu(u))$ is isomorphic to $L\otimes \bC_\gamma$, where $L$ is some finite-dimensional irreducible $\mathscr Y(\gl_\s)$-module and $\bC_\gamma$ is some one-dimensional $\BMN$-module defined in Example \ref{eg:1-dim} with $\gamma\in\bC$.
\end{enumerate}
\end{prop}

\subsection{Super case}
In the rest of this section, we establish super analogous results to the previous propositions when $s_1\ne s_2$. Our main results in this subsection are the following.
\begin{prop}\label{prop:rank1}
If $\bm s$ is such that $s_1\ne s_2$, then the $\BMN$-module $V(\bm\mu(u))$ is finite-dimensional if and only if there exists a monic polynomial $P(u)$ such that
\[
\frac{\tl \mu_1(u)}{\tl \mu_2(u)}=\ve_1\ve_2(-1)^{\deg P}\frac{P(u)}{P(-u+s_2)}.
\]
\end{prop}
\begin{proof}
The ``$\Longleftarrow$" part follows from Theorem \ref{thm:zhang} and Theorem \ref{thm:suff} below as $V(\bm\mu(u))$ can be obtained as a quotient of the restriction of a finite-dimensional irreducible $\YMN$-module.

To show the ``$\Longrightarrow$" part, note that due to the condition \eqref{thm:nontrivial-i}, such a polynomial $P(u)$ exists provided that ${\tl \mu_1(u)}/{\tl \mu_2(u)}$ or alternatively $\mu_1(u)/\mu_2(u)$ is an expansion of a rational function in $u$ at $u=\infty$.

We will work on
$$
x_{ij}(u)=\ve_i\delta_{ij}+\sum_{r>0}x_{ij}^{(r)}u^{-r}:=b_{ij}(u+\tfrac{s_2}{2})
$$
and the case $\bm s=(1,-1)$ since the other case $\bm s=(-1,1)$ is similar by inserting the signs at suitable positions or using the isomorphism \eqref{neghomo1}.

Using \eqref{eq:comm-series b}, we have
\[
[b_{21}(u),b_{22}(v)]=\frac{-1}{u-v}(b_{21}(u)b_{22}(v)-b_{21}(v)b_{22}(u))+\frac{1}{u+v}(b_{21}(v)b_{11}(u)+b_{22}(v)b_{21}(u)),
\]
which gives
\begin{align*}
\frac{u+v+1}{u+v}b_{22}(v)b_{21}(u)=b_{21}(u)b_{22}(v)-&\,\frac{1}{u+v}b_{21}(v)b_{11}(u)\\-&\,\frac{1}{u-v}(b_{21}(v)b_{22}(u)-b_{21}(u)b_{22}(v)).
\end{align*}
Substituting $u\to u-1/2$, $v\to v-1/2$ and dividing both sides by $(u+v)/(u+v-1)$, we obtain
\beq\label{eq:new-b}
\begin{split}
x_{22}(v)x_{21}(u)=&\, x_{21}(u)x_{22}(v)-\frac{1}{v+u}(x_{21}(v)x_{11}(u)+x_{21}(u)x_{22}(v))\\
+&\, \frac{1}{v-u}(x_{21}(v)x_{22}(u)-x_{21}(u)x_{22}(v))+\frac{1}{v^2-u^2}(x_{21}(u)x_{22}(v)-x_{21}(v)x_{22}(u)).
\end{split}
\eeq

Taking the coefficients of $u^{-k}v^{-2}$ and $v^{-2}$, we have
\begin{align}
x^{(2)}_{22}x_{21}^{(k)}&=x_{21}^{(k)}(x_{22}^{(2)}-2x_{22}^{(1)}+x_{22}^{(0)})-x_{21}^{(1)}(x_{11}^{(k)}-x_{22}^{(k)}),\label{eq:k2}\\
x^{(2)}_{22}x_{21}(u)&=x_{21}(u)(x_{22}^{(2)}-2x_{22}^{(1)}+x_{22}^{(0)})-x_{21}^{(1)}(x_{11}(u)-x_{22}(u)).\label{eq:just2}
\end{align}
Similarly, taking the coefficients of $u^{-k}v^{-3}$ and $v^{-3}$, we have
\begin{align}
x^{(3)}_{22}x_{21}^{(k)}&=-2\ve_2x_{21}^{(k+2)}+x_{21}^{(k)}(x_{22}^{(3)}-2x_{22}^{(2)}+x_{22}^{(1)})\nonumber \\ &\qquad\qquad\qquad\,  -x_{21}^{(2)}(x_{11}^{(k)}-x_{22}^{(k)})+x_{21}^{(1)}(x_{11}^{(k+1)}+x_{22}^{(k+1)}-x_{22}^{(k)}),\label{eq:k3}\\
x^{(3)}_{22}x_{21}(u)&=-2\ve_2u^2x_{21}(u)+x_{21}(u)(x_{22}^{(3)}-2x_{22}^{(2)}+x_{22}^{(1)})\nonumber\\&\quad \ \  -x_{21}^{(2)}(x_{11}(u)-x_{22}(u))+x_{21}^{(1)}(ux_{11}(u)+ux_{22}(u)-x_{22}(u)).\label{eq:just3}
\end{align}

Denote by $\eta$ the highest $\ell_{\s,\bm\ve}$-weight vector of $V(\bm\mu(u))$.  Define the series $\la_i(u)$ by
\[
\la_i(u)=\sum_{r\gge 0} \la_{ir}u^{-r},\qquad i=1,2,
\]
where $\la_{ir}\in\bC$ and $x_{ii}^{(r)}\eta=\la_{ir}\eta$. In particular, $\la_{i0}=\ve_i$.

(1) The case $\ve_1=\ve_2$. We first prove that for $r>0$, the vector $x_{21}^{(2r)}\eta$ is a linear combination of vectors $x_{21}^{(1)}\eta,x_{21}^{(3)}\eta,\cdots,x_{21}^{(2r-1)}\eta$.

We prove it by induction on $k$. Setting $k=0$ in \eqref{eq:k2} and noticing that $x_{21}^{(0)}=0$, we have
\beq\label{eq:ind-base}
(\ve_1+\ve_2)x_{21}^{(2)}=x_{21}^{(1)}(x_{11}^{(1)}+x_{22}^{(1)}-\ve_2).
\eeq
Therefore, $2\ve_2x_{21}^{(2)}\eta=(\la_{11}+\la_{21}-\ve_2)x_{21}^{(1)}\eta$. Now suppose that
\beq\label{eq:ind-hypo}
x_{21}^{(2r)}\eta=c_1x_{21}^{(1)}\eta+\cdots+c_{2r-1}x_{21}^{(2r-1)}\eta
\eeq
for some $c_1,\cdots,c_{2r-1}\in\bC$. Then applying \eqref{eq:k3} to $x_{22}^{(3)}(c_1x_{21}^{(1)}+\cdots+c_{2k-1}x_{21}^{(2r-1)})\eta$, we find that $x_{22}^{(3)}(c_1x_{21}^{(1)}+\cdots+c_{2r-1}x_{21}^{(2r-1)})\eta$ is a linear combination of vectors $x_{21}^{(1)}\eta$, $x_{21}^{(3)}\eta$, $\cdots$, $x_{21}^{(2r-1)}\eta$ by \eqref{eq:ind-base} and \eqref{eq:ind-hypo}. Similarly, by \eqref{eq:k3} and \eqref{eq:ind-base}, $x_{22}^{(3)}x_{21}^{(2r)}\eta$ is equal to $-2\ve_2x_{21}^{(2r+2)}\eta$ plus a linear combination of vectors $x_{21}^{(1)}\eta$, $x_{21}^{(2r)}\eta$. Thus, the claim is proved.

Let $\eta_r=x_{21}^{(2r-1)}\eta$ for $r\in\Z_{>0}$. Since $V(\bm\mu(u))$ is finite-dimensional, there exists a minimal non-negative integer $k$ such that $\eta_{k+1}$ is a linear combination of the vectors $\eta_1,\cdots,\eta_k$,
\beq\label{eq:pf-low-0}
\eta_{k+1}=c_1\eta_1+\cdots+c_k\eta_k.
\eeq
Then for any $r>k$, one proves similarly as above that
\[
\eta_r=a_{r1}\eta_1+\cdots +a_{rk}\eta_k
\]
for some $a_{ri}\in \bC$, where $1\lle i\lle k$. Therefore, there exist series $a_i(u)\in u^{1-2i}(1+\bC[[u^{-1}]])$, $1\lle i\lle k$, in $u^{-1}$ such that
\beq\label{eq:x21-eta}
x_{21}(u)\eta=a_1(u)\eta_1+a_2(u)\eta_2+\cdots+a_k(u)\eta_k.
\eeq

To simplify the notation, we use the following shorthand notation for these scalars,
\begin{align*}
	\La_0&=\la_{22}-2\la_{21}+\ve_2,\qquad \La_1=\la_{23}-2\la_{22}+\la_{21},\\
	\beta_r&=\la_{2r}-\la_{1r},\qquad\qquad\ \ \  \theta_r=\la_{1,r+1}+\la_{2,r+1}-\la_{2r}.
\end{align*}
By \eqref{eq:just2} and \eqref{eq:x21-eta}, we have
\beq\label{eq:pf-low-1}
\begin{split}
	x_{22}^{(2)}x_{21}(u)\eta=&\,\La_0x_{21}(u)\eta-(\la_1(u)-\la_2(u))\eta_1\\
	=&\,\La_0(a_1(u)\eta_1+a_2(u)\eta_2+\cdots+a_k(u)\eta_k)-(\la_1(u)-\la_2(u))\eta_1.
\end{split}
\eeq
On the other hand, by \eqref{eq:k2}, we have
\beq\label{eq:pf-low-2}
	x_{22}^{(2)}\sum_{r=1}^k a_r(u)\eta_k= \sum_{r=1}^k a_r(u)x_{22}^{(2)}x_{21}^{(2r-1)}\eta	=\sum_{r=1}^k a_r(u)(\La_0\eta_r+\beta_{2r-1}\eta_1).
\eeq
Comparing the coefficients of $\eta_1$, it follows that
\beq\label{eq:pf-low-3}
\la_2(u)-\la_1(u)=\sum_{k=1}^r \beta_{2k-1}a_k(u).
\eeq

Recall that \eqref{eq:ind-base} implies $x_{21}^{(2)}\eta=\frac{1}{2}\theta_0\ve_2\eta_1$. Similarly, by \eqref{eq:just3}, we have
\begin{align*}
	x_{22}^{(3)}&x_{21}(u)\eta\\
	=&\,(-2\ve_2u^2+\La_1)x_{21}(u)\eta+\Big(u\la_1(u)+u\la_2(u)-\la_2(u)+\frac{1}{2}\theta_0\ve_2(\la_2(u)-\la_1(u))\Big)\eta_1\\
	=&\,(-2\ve_2u^2+\La_1)\sum_{r=1}^k a_r(u)\eta_r+\Big(u\la_1(u)+u\la_2(u)-\la_2(u)+\frac{1}{2}\theta_0\ve_2(\la_2(u)-\la_1(u))\Big)\eta_1.
\end{align*}
On the other hand, by \eqref{eq:k3}, we have
\begin{align*}
	x_{22}^{(3)}\sum_{r=1}^k a_r(u)\eta_r=&\,\sum_{r=1}^k a_r(u)x_{22}^{(3)}x_{21}^{(2r-1)}\eta\\
	=&\, \sum_{r=1}^k a_r(u)\Big(-2\ve_2\eta_{r+1}+\La_1\eta_r+\frac{1}{2}\theta_0\ve_2\beta_{2r-1}\eta_1+\theta_{2r-1}\eta_1\Big).
\end{align*}
Applying \eqref{eq:pf-low-0} to the above equality and comparing the coefficients of $\eta_r$ for $1<r\lle k$, we obtain that
\[
(-2\ve_2u^2+\La_1)a_r(u)=(-2\ve_2a_{r-1}(u)+\La_1 a_r(u)-2\ve_2 c_ra_k(u)),
\]
which reduces to
\beq\label{eq:pf-low-4}
a_{r-1}(u)=u^2a_r(u)-c_ra_k(u).
\eeq
Hence, for any $1\lle r\lle k$, we have $a_r(u)=\mathfrak P_r(u)a_k(u)$, where $\mathfrak P_r(u)$ is a polynomial in $u$ of degree $2(k-r)$. Finally, taking the coefficients of $\eta_1$ and using \eqref{eq:pf-low-4}, we conclude that
\beq\label{eq:pf-low-5}
\Big(u-\frac{1}{2}\theta_0\ve_2\Big)\la_1(u)+\Big(u+\frac{1}{2}\theta_0\ve_2-1\Big)\la_2(u)=\mathfrak P(u)a_k(u),
\eeq
where $\mathfrak P(u)$ is a polynomial in $u$ of degree $2k$. Note that \eqref{eq:pf-low-3} and \eqref{eq:pf-low-4} imply
\beq\label{eq:pf-low-6}
\la_2(u)-\la_1(u)=\mathscr P(u) a_k(u),
\eeq
where $\mathscr P(u)$ is a polynomial in $u$ of degree at most $2k-2$. It follows from \eqref{eq:pf-low-5} and \eqref{eq:pf-low-6} that $\la_1(u)/\la_2(u)$ is an expansion of a rational function in $u$ at $u=\infty$, completing the proof for the case $\ve_1=\ve_2$.

(2) The case $\ve_1\ne\ve_2$. The proof for this case is very similar to the previous case. The difference is that one needs to use $x_{21}^{(2r)}\eta$ and $x_{22}^{(4)}x_{21}^{(2r)}$ instead of $x_{21}^{(2r-1)}\eta$ and $x_{22}^{(2)}x_{21}^{(2r-1)}$, respectively, cf. \cite[Proposition 6.1]{Molev1998finite} as we have $x_{21}^{(1)}=0$ in this case. Then \eqref{eq:pf-low-6} is replaced with
\[
(u^2+\cdots)\la_2(u)-(u^2+\cdots)\la_1(u)=\mathscr P(u)a_k(u),
\]
where $\cdots$ stand for two different linear polynomials  in $u$ and $\mathscr P(u)$ is a polynomial in $u$ of degree $2k+2$. Note that in this case $\mathfrak P(u)$ in \eqref{eq:pf-low-5} is of degree at most $2k$.
We omit the detail for this case.
\end{proof}

\subsection{Preparations}
In this subsection, we prepare ingredients to establish super analogue of Proposition \ref{prop:con-even} and always assume that $\bm s=(s_1,s_2)$ satisfies $s_1\ne s_2$.

We start with simple calculations for 2-dimensional evaluation modules. Let $\mc a$ and $\mc b$ be complex numbers such that $\mc a+\mc b\ne 0$. Then the evaluation $\YglMN$-module $L(\mc a,\mc b)$ is two dimensional. Let $v^+$ be a nonzero singular vector and set $v^-=e_{21}v^+$.

\begin{lem}\label{lemn2}
We have the following explicit action,
\begin{align*}
t_{11}(u)v^+&=\frac{u+s_1\mc a}{u}v^+,\qquad  t_{12}(u)v^+=0,\\
t_{22}(u)v^+&=\frac{u-s_1\mc b}{u}v^+,\qquad t_{21}(u)v^+=-\frac{s_1}{u}v^-,\\
t_{11}(u)v^-&=\frac{u-s_1+s_1\mc a}{u}v^-,\qquad t_{21}(u)v^-=0,\\
t_{22}(u)v^-&=\frac{u-s_1-s_1\mc b}{u}v^-,\qquad t_{12}(u)v^-=\frac{s_1(\mc a+\mc b)}{u}v^+,\\
t'_{11}(u)v^+&=\frac{u(u-s_1-s_1\mc b)}{(u-s_1+s_1\mc a)(u-s_1\mc b)}v^+,\qquad  t_{12}'(u)v^+=0,\\
t'_{22}(u)v^+&=\frac{u}{u-s_1\mc b}v^+,\qquad t_{21}'(u)v^+=\frac{s_1u}{(u-s_1+s_1\mc a)(u-s_1\mc b)}v^-,\\
t'_{11}(u)v^-&=\frac{u}{u-s_1+s_1\mc a}v^+,\qquad t_{12}'(u)v^-=\frac{-s_1(\mc a+\mc b)u}{(u-s_1+s_1\mc a)(u-s_1\mc b)}v^+,\\
t'_{22}(u)v^-&=\frac{u(u+s_1\mc a)}{(u-s_1+s_1\mc a)(u-s_1\mc b)}v^-, \qquad t_{21}'(u)v^-=0.
\end{align*}
In particular, $(u+s_1-s_1\mc a)(u+s_1\mc b)t_{ij}(u)t'_{kl}(-u)$ and $(u+s_1-s_1\mc a)(u+s_1\mc b)b_{ij}(u)$ act on $L(\mc a,\mc b)$ polynomially in $u$.
\end{lem}
\begin{proof}
The formulas follow from straightforward computation and the second statement follows from the formulas.
\end{proof}

We call two $\BMN$-modules $V_1,V_2$ are \textit{almost isomorphic} if $V_1$ is isomorphic to the module obtained by pulling back $V_2$ through an automorphism of the form $\mc M_{\mc h(u)}$, see \eqref{mhu}. In particular, the modules $V(\bm\mu(u))$ and $V(\bm\nu(u))$ are almost isomorphic if and only if
\[
\frac{\tl \mu_i(u)}{\tl \mu_{i+1}(u)}=\frac{\tl \nu_i(u)}{\tl \nu_{i+1}(u)},\qquad 1\lle i<\ka.
\]
Similarly, one can define \textit{almost isomorphic} $\YMN$-modules. Then the modules $L(\bm\la(u))$ and $L(\bm\La(u))$ are almost isomorphic if and only if
\[
\frac{\la_i(u)}{\la_{i+1}(u)}=\frac{\La_i(u)}{\La_{i+1}(u)},\qquad 1\lle i<\ka.
\]
If $V_1,V_2$ are \textit{almost isomorphic}, then we write $V_1\simeq V_2$.

To understand the module structure of finite-dimensional irreducible $\BMN$ modules, it suffices to investigate these modules up to almost isomorphism.

We shall also need the dual modules.
Let $L$ be a finite-dimensional $\YglMN$-module. The \textit{dual} $L^*$ of $L$ is the representation of $\YglMN$ on the dual vector space of $L$ defined as follows:
\[
(\mc y\cdot \omega)(v):=(-1)^{|\omega||\mc y|}\omega(\Omega(\mc y)\cdot v),\quad  \mc y\in \YglMN,\ \omega\in L^* ,\ v\in L,
\]
where $\Omega$ is defined in \eqref{Omega}. Let $w$ be another finite-dimensional $\YglMN$-module. Then we have $(L\otimes W)^*=L^*\otimes W^*$.                                                                            Let $L$ be a finite-dimensional $\YglMN$-module of highest $\ell_\s$-weight generated by a highest $\ell_\s$-weight vector $\zeta$. Let $\zeta^*\in L^*$ be the vector such that $\zeta^*(\zeta)=1$ and $\zeta^*(v)=0$ for all $v\in L$ with $\mathrm{wt}(v)\ne \mathrm{wt}(\zeta)$. 

\begin{cor}\label{cor dual weight}
Let $V$ be a finite-dimensional $\YglMN$-module of highest $\ell_\s$-weight generated by a highest $\ell_\s$-weight vector $v$ of $\ell_\s$-weight $\bm\zeta(u)=(\zeta_i(u))_{1\lle i\lle \ka}$. Then $v^*$ is of highest $\ell_\s$-weight $\wt{\bm{\zeta}}(u)=(\wt\zeta_i(u))_{1\lle i\lle \ka}$, where
\[
\wt\zeta_i(u)=\frac{1}{\la_{i}(-u+\rho_{i+1})}\prod_{k=i+1}^\ka \frac{\la_k(-u+\rho_k)}{\la_{k}(-u+\rho_{k+1})}.
\]
\end{cor}
\begin{proof}
It is clear that $v^*$ is of highest $\ell_\s$-weight and its highest $\ell_\s$-weight is immediate from Proposition \ref{prop:highest-weight-inver}.
\end{proof}

Note that
\begin{align*}
\Omega(b_{ij}(u))&=\sum_{a=1}^\ka\ve_a(-1)^{|i||a|+|a|+|a||j|+|j|+(|a|+|i|)(|a|+|j|)}t_{ja}(u)t'_{ai}(-u)\\
&= \sum_{a=1}^\ka\ve_a(-1)^{|i||j|+|j|}t_{ja}(u)t'_{ai}(-u)=(-1)^{|i||j|+|j|}b_{ji}(u).
\end{align*}
This means that the subalgebra $\BMN$ of $\YMN$ is stable under $\Omega$ and the restriction of $\Omega$ to $\BMN$ yields an anti-automorphism of $\BMN$.

Let $V$ be a finite-dimensional $\BMN$-module. The \textit{dual} $V^*$ of $V$ is the representation of $\BMN$ on the dual vector space of $V^*$ defined as follows:
\[
(\mc y\cdot \omega)(v):=(-1)^{|\omega||\mc y|}\omega(\Omega(\mc y)\cdot v),\quad  \mc y\in \BMN,\ \omega\in V^* ,\ v\in V.
\]
Clearly, the dual $\C_{\gamma}^*$ of the one-dimensional module $\C_{\gamma}$ is isomorphic to $\C_\gamma$.

Let $L$ be a finite-dimensional $\YMN$-module and $V$ a finite-dimensional $\BMN$-module, then it is straightforward to verify that
\[
(L\otimes V)^*\cong L^*\otimes V^*.
\]

Fix $k\in\bZ_{>0}$. Let $(\mc a_i,\mc b_i)$ be a pair of complex numbers for each $1\lle i\lle k$. Consider the following tensor product of evaluation $\YMN$-modules,
\[
L(\bm{\mc a,\mc b})=L(\mc a_1,\mc b_1)\otimes \cdots\otimes L(\mc a_k,\mc b_k).
\]
\begin{lem}\label{lem:dual}
We have
\[
L(\bm{\mc a,\mc b})^*\simeq L(\mc b_1+1,\mc a_1-1)\otimes \cdots\otimes L(\mc b_k+1,\mc a_k-1)
\]
as $\YMN$-modules. Moreover, we have
\[
\big(L(\bm{\mc a,\mc b})\otimes \C_{\gamma}\big)^*\simeq L(\mc b_1+1,\mc a_1-1)\otimes \cdots\otimes L(\mc b_k+1,\mc a_k-1)\otimes \C_{\gamma}
\]
as $\BMN$-modules.
\end{lem}
\begin{proof}
The lemma follows from Corollary \ref{cor dual weight} by a direct computation.
\end{proof}

Though we work with the case $\ka=2$, the dual modules can be generalized to arbitrary $\ka$ and $\s$.

\subsection{The case $\ve_1=\ve_2$}
Now we assume further that $\ve_1=\ve_2$. Suppose $V(\bm\mu(u))$ is finite-dimensional, then by Proposition \ref{prop:rank1} we have
\beq\label{P-def}
\frac{\tl \mu_1(u)}{\tl \mu_2(u)}=(-1)^{\deg P(u)}\frac{P(u)}{P(-u-s_1)}.
\eeq
We assume further that $P(u)$ and $P(-u-s_1)$ are relatively prime. Otherwise, we may cancel common factors and obtain a polynomial of smaller degree. Then $P(-s_1/2)\ne 0$. Suppose
\[
P(u)=(u+s_1\mc a_1)(u+s_1\mc a_2)\cdots (u+s_1\mc a_{l})
\]
where $l=\deg P(u)$ and $\mc a_i\in\C$ for $1\lle i\lle l$. We have $\mc a_i\ne 1/2$.

Set $k=\lfloor\frac{l+1}{2} \rfloor$. Introduce $k$ pair of complex numbers $(\mc a_i,\mc b_i)$, where $\mc a_i$ are defined as above while $\mc b_i$ are defined as follows,
\begin{itemize}
    \item if $l$ is even, then $\mc b_i=\mc a_{i+k}-1$ for $1\lle i\lle k$;
    \item if $l$ is odd, then $\mc b_i=\mc a_{i+k}-1$ for $1\lle i< k$ and $\mc b_k=-\frac{1}{2}$.
\end{itemize}
Then we have
\beq\label{P-factor}
\frac{P(u)}{P(-u-s_1)}=(-1)^{\deg P(u)}\prod_{i=1}^k\frac{(u+s_1\mc a_i)(u+s_1+s_1\mc b_i)}{(u-s_1\mc b_i)(u+s_1-s_1\mc a_i)}.
\eeq
The only possible cancellation is when $l$ is odd, then
\[
u+s_1+s_1\mc b_k=u-s_1\mc b_k.
\]
Note that we have
\beq\label{notequalto}
\mc a_i+\mc b_i\ne 0,\quad \mc a_i+\mc b_j\ne 0, \quad \mc a_i+\mc a_j\ne 1,\quad \mc b_i+\mc b_j\ne -1
\eeq
for all $1\lle i\ne j\lle k$.

We consider the following tensor product of evaluation $\YMN$-modules,
\[
L(\bm{\mc a,\mc b})=L(\mc a_1,\mc b_1)\otimes \cdots\otimes L(\mc a_k,\mc b_k).
\]
For each $1\lle i\lle k$, $L(\mc a_i,\mc b_i)$ is two dimensional. We set $v_i^+$ to be a nonzero singular vector and $v_i^-=e_{21}v_i^+$. Moreover, we assume that $v_i^+$ are even. We also set $v^+=v_1^+\otimes\cdots\otimes v_k^+$. We regard $L(\bm{\mc a,\mc b})$ as a $\BMN$-module by restriction.
\begin{prop}\label{prop=}
If $\ve_1=\ve_2$, then the  $\BMN$-module $L(\bm{\mc a,\mc b})$ is irreducible. Moreover, the finite-dimensional irreducible $\BMN$-module $V(\bm\mu(u))$ is almost isomorphic to $L(\bm{\mc a,\mc b})$.
\end{prop}
\begin{proof}
It follows from the proof of Theorem \ref{thmnontrivial} that the vector $v^+$ is a highest $\ell_{\s,\bm\ve}$-weight vector. Suppose the corresponding $\ell_{\s,\bm\ve}$-weight is $\bm\nu(u)=(\nu_1(u),\nu_2(u))$. Then it follows from Proposition \ref{prop:highest-weight-inver}, Lemma \ref{lem:b-in-t}, and Lemma \ref{lemn2}, , that
\[
\tl \nu_1(u)=2\ve_1u\prod_{i=1}^k\frac{(u+s_1\mc a_i)(u+s_1+s_1\mc b_i)}{(u+s_1-s_1\mc a_i)(u+s_1\mc b_i)},\quad \tl \nu_2(u) = 2\ve_2u\prod_{i=1}^k\frac{u-s_1\mc b_i}{u+s_1\mc b_i}.
\]
Therefore, by \eqref{P-def} and \eqref{P-factor}, we have
\[
\frac{\tl \nu_1(u)}{\tl \nu_2(u)}=\prod_{i=1}^k\frac{(u+s_1\mc a_i)(u+s_1+s_1\mc b_i)}{(u-s_1\mc b_i)(u+s_1-s_1\mc a_i)}=(-1)^{\deg P(u)}\frac{P(u)}{P(-u-s_1)}=\frac{\tl \mu_1(u)}{\tl \mu_2(u)}.
\]
Thus, it suffices to prove that the  $\BMN$-module $L(\bm{\mc a,\mc b})$ is irreducible.

We claim that any vector $\eta\in L(\bm{\mc a,\mc b})$ satisfying $b_{12}(u)\eta=0$ is proportional to $v^+$. We prove the claim by induction on $k$. The case $k=1$ is obvious by Lemma \ref{lemn2}. Then we assume that $k\gge 2$. We write any such nonzero vector $\eta$ in the following form,
\[
\eta=\sum_{r=0}^1 (e_{21})^rv_1^+\otimes\eta_r=v_1^+\otimes \eta_0+v_1^-\otimes \eta_1,
\]
where $\eta_0,\eta_1\in L(\mc a_2,\mc b_2)\otimes \cdots\otimes L(\mc a_k,\mc b_k)$. We first prove that $\eta_1=0$. Suppose $\eta_1\ne 0$. Then it follows from Proposition \ref{prop:coproduct} that
\begin{align*}
    \Delta(b_{12}(u))=&\, t_{11}(u)t'_{12}(-u)\otimes b_{11}(u)+t_{11}(u)t'_{22}(-u)\otimes b_{12}(u)\\-&\,
    t_{12}(u)t'_{12}(-u)\otimes b_{21}(u)+t_{12}(u)t'_{22}(-u)\otimes b_{22}(u).
\end{align*}
Applying $b_{12}(u)$ to $\eta$ using this coproduct, it follows from $b_{12}(u)\eta=0$ that
\[
\big(t_{11}(u)t'_{22}(-u)\otimes b_{12}(u)\big)(v_1^-\otimes \eta_1)=0
\]
by taking the coefficient of $v_1^-$. Therefore, we have
\[
-\frac{(u-s_1+s_1\mc a_1)(u-s_1\mc a_1)}{(u+s_1-s_1\mc a_1)(u+s_1\mc b_1)}v_1^-\otimes b_{12}(u)\eta_1=0,
\]
which implies that $b_{12}(u)\eta_1=0$. By induction hypothesis, the vector $\eta_1$ must be proportional to $v_2^+\otimes\cdots\otimes v_k^+$. Then taking the coefficient of $v^+$ in $b_{12}(u)\eta=0$, we have
\begin{align*}
0=\frac{u+s_1\mc a_1}{u+s_1\mc b_1}v_1^+\otimes b_{12}(u)\eta_0+t_{12}(u)t'_{22}(-u)v_1^-\otimes b_{22}(u)\eta_1+t_{11}(u)t'_{12}(-u)v_1^-\otimes b_{11}(u)\eta_1.
\end{align*}
Note that by \eqref{eq:def-tl-b} we have $\tl b_{11}(u)=(2u+s_1)b_{11}(u)-s_1b_{22}(u)$ and $\tl b_{22}(u)=2ub_{22}(u)$. We deduce from the above equation and Lemma \ref{lem:b-in-t} that
\begin{align*}
0&=\frac{u+s_1\mc a_1}{u+s_1\mc b_1}v_1^+\otimes b_{12}(u)\eta_0+\frac{2\ve_2s_1u(\mc a_1+\mc b_1)}{(u+s_1\mc b_1)(2u+s_1)}\prod_{i=2}^k\frac{u-s_1\mc b_i}{u+s_1\mc b_i} v_1^+\otimes \eta_1 \\
&+\frac{2\ve_1s_1u(\mc a_1+\mc b_1)(u+s_1\mc a_1)}{(u+s_1-s_1\mc a_1)(u+s_1\mc b_1)(2u+s_1)}\prod_{i=2}^k\frac{(u+s_1\mc a_i)(u+s_1+s_1\mc b_i)}{(u+s_1-s_1\mc a_i)(u+s_1\mc b_i)} v_1^+\otimes \eta_1.
\end{align*}
Multiplying both sides by $(2u+s_1)\prod_{i=1}^k\big((u+s_1-s_1\mc a_i)(u+s_1\mc b_i)\big)$, we have
\begin{align*}
0&=(2u+s_1)(u+s_1\mc a_1)(u+s_1-s_1\mc a_1)v_1^+\otimes \prod_{i=2}^k\big((u+s_1-s_1\mc a_i)(u+s_1\mc b_i)\big)b_{12}(u)\eta_0\\
&+2\ve_2s_1u(\mc a_1+\mc b_1)(u+s_1-s_1\mc a_1)\prod_{i=2}^k\big((u+s_1-s_1\mc a_i)(u-s_1\mc b_i)\big) v_1^+\otimes \eta_1 \\
&+2\ve_1s_1u(\mc a_1+\mc b_1)(u+s_1\mc a_1)\prod_{i=2}^k\big((u+s_1\mc a_i)(u+s_1+s_1\mc b_i)\big) v_1^+\otimes \eta_1.
\end{align*}
Due to Lemma \ref{lemn2}, the operator $\prod_{i=2}^k\big((u+s_1-s_1\mc a_i)(u+s_1\mc b_i)\big)b_{12}(u)$ acts on $\eta_0$ polynomially in $u$.

(1) If $\mc a_1\ne 0$, then setting $u=-s_1\mc a_1$, we obtain
\[
2s_1\mc a_1(\mc a_1+\mc b_1)(2\mc a_1-1)\prod_{i=2}^k(\mc a_1+\mc a_i-1)(\mc a_1+\mc b_i)v_1^+\otimes \eta_1=0.
\]
Thus, by \eqref{notequalto}, we conclude that $\eta_1=0$.

(2) If $\mc a_1=0$, then setting $u=s_1\mc a_1-s_1$, we get
\[
2s_1(\mc a_1-1)(\mc a_1+\mc b_1)(2\mc a_1-1)\prod_{i=2}^k(\mc a_1+\mc a_i-1)(\mc a_1+\mc b_i)v_1^+\otimes \eta_1=0.
\]
Again by \eqref{notequalto}, we conclude that $\eta_1=0$.

Therefore, we must have $b_{12}(u)\eta_0=0$ which again by induction hypothesis that $\eta_0$ is proportional to $v_2^+\otimes\cdots\otimes v_k^+$. Thus the claim is proved.

Suppose now that $M$ is a submodule of $L(\bm{\mc a,\mc b})$. Then $M$ must contain a nonzero vector $\eta$ such that $b_{12}(u)\eta=0$, see Lemma \ref{lem:nontrivial}. The above argument thus shows that $M$ contains the vector $v^+$. It remains to prove the cyclic span $K=\BMN v^+$ coincides with $L(\bm{\mc a,\mc b})$. By Lemma \ref{lem:dual}, the dual $\BMN$-module $L(\bm{\mc a,\mc b})^*$ is almost isomorphic to the restriction of the $\YMN$-module
\[
L(\mc b_1+1,\mc a_1-1)\otimes \cdots\otimes L(\mc b_k+1,\mc a_k-1).
\]
Moreover, the highest $\ell_{\s,\bm\ve}$ vector $\zeta_i^*$ of the module $L(\mc b_i+1,\mc a_i-1)\simeq L(\mc a_i,\mc b_i)^*$ can be identified with the elements of $L(\mc a_i,\mc b_i)^*$ such that $\zeta_i^*(v_i^+)=1$ and $\zeta_i^*(v_i^-)=0$. Now, if the submodule $K$ of $L(\bm{\mc a,\mc b})$ is proper, then its annihilator
\[
\mathrm{Ann}\,K:=\{\omega\in L(\bm{\mc a,\mc b})^*~|~\omega(\eta)=0 \quad\text{for all}\quad \eta\in K \}
\]
is a nonzero submodule of $L(\bm{\mc a,\mc b})^*$ which does not contain the vector $\zeta_1^*\otimes\cdots\otimes \zeta_k^*$. However, this contradicts the claim proved in the first part of the proof because the strategy still works for the module $L(\mc b_1+1,\mc a_1-1)\otimes \cdots\otimes L(\mc b_k+1,\mc a_k-1)$ with the previous assumptions on the complex numbers $\mc a_i,\mc b_i$. In this case, instead of using $2\mc a_1-1\ne 0$, we need the condition $2\mc b_1+1\ne 0$. This is true if $k\gge 2$ as the only possibility for $\mc b_i=-\frac12$ is when $i=k$. As for the initial case $k=1$, it can be checked by a direct computation.
\end{proof}

\begin{cor}
Suppose $\s=(s_1,s_2)$ and $\bm \ve=(\ve_1,\ve_2)$ are such that $s_1\ne s_2$ and $\ve_1=\ve_2$.
\begin{enumerate}
    \item If $\bm\mu(u)$ satisfies \eqref{P-def}, where $P(u)$ and $P(-u-s_1)$ are relatively prime, then $\dim V(\bm\mu(u))=2^k$, where $k=\big\lfloor\frac{\deg P(u)+1}{2}\big\rfloor$.
    \item Given $k\in\bZ_{>0}$, let $\mc a_i,\mc b_i$, $1\lle i\lle k$, be arbitrary complex numbers such that $\mc a_i+\mc b_i\ne 0$ and set $$P(u)=\prod_{i=1}^k\big((u+s_1\mc a_i)(u+s_1+s_1\mc b_i)\big).$$ Then the $\BMN$-module obtained by the restriction of the $\YMN$-module
    \[
    L(\bm{\mc a,\mc b})=L(\mc a_1,\mc b_1)\otimes \cdots\otimes L(\mc a_k,\mc b_k)
    \]
    is irreducible if and only if the greatest common divisor of $P(u)$ and $P(-u-s_1)$ (over $\bC$) is of degree at most 1 (If the greatest common divisor is nontrivial, then it has to be $u+\frac{s_1}{2}$).
\end{enumerate}
\end{cor}

\subsection{The case $\ve_1\ne \ve_2$}
Now we assume that $\ve_1\ne \ve_2$. In this case, it is slightly different from the previous case $\ve_1=\ve_2$ as there is a nontrivial one-dimensional module. We shall give detail for this part as well.

Suppose $V(\bm\mu(u))$ is finite-dimensional, then by Proposition \ref{prop:rank1} we have
\beq\label{P-def2}
\frac{\tl \mu_1(u)}{\tl \mu_2(u)}=(-1)^{\deg P(u)+1}\frac{P(u)}{P(-u-s_1)}.
\eeq
We assume further that $P(u)$ and $P(-u-s_1)$ are relatively prime. Otherwise, we may cancel common factors and obtain a polynomial of smaller degree. Then $P(-\frac{s_1}{2})\ne 0$. Suppose
\[
P(u)=(u+s_1\mc a_1)(u+s_1\mc a_2)\cdots (u+s_1\mc a_{l})
\]
where $l=\deg P(u)$ and $\mc a_i\in\C$ for $1\lle i\lle l$. We have $\mc a_i\ne \frac{1}{2}$.

Set $k=\lfloor\frac{l}{2} \rfloor$. Introduce $k$ pair of complex numbers $(\mc a_i,\mc b_i)$, where $\mc a_i$ are defined as above while $\mc b_i$ are defined as follows,
\begin{itemize}
    \item if $l$ is odd, then $\mc b_i=\mc a_{i+k}-1$ for $1\lle i\lle k$;
    \item if $l$ is even, then $\mc b_i=\mc a_{i+k}-1$ for $1\lle i< k$ and $\mc b_k=-\frac{1}{2}$.
\end{itemize}
We also set $\gamma=\ve_1s_1(a_{l}-1)$. Then we have
\beq\label{P-factor2}
\frac{P(u)}{P(-u-s_1)}=(-1)^{\deg P(u)+1}\frac{\ve_1u+s_1\ve_1+\gamma}{\ve_2u+\gamma}\prod_{i=1}^k\frac{(u+s_1\mc a_i)(u+s_1+s_1\mc b_i)}{(u-s_1\mc b_i)(u+s_1-s_1\mc a_i)}.
\eeq
The only possible cancellation is when $l$ is even, then
\[
u+s_1+s_1\mc b_k=u-s_1\mc b_k.
\]
Note that we have
\beq\label{notequalto2}
\mc a_i+\mc b_i\ne 0,\quad \mc a_i+\mc b_j\ne 0, \quad \mc a_i+\mc a_j\ne 1,\quad \mc b_i+\mc b_j\ne -1
\eeq
for all $1\lle i\ne j\lle k$.

We consider the tensor product of evaluation $\YMN$-modules,
\[
L(\bm{\mc a,\mc b})=L(\mc a_1,\mc b_1)\otimes \cdots\otimes L(\mc a_k,\mc b_k)
\]
and the tensor product
\[
V_{\gamma}(\bm{\mc a,\mc b})=L(\bm{\mc a,\mc b})\otimes \bC_{\gamma}
\]
Then $V_{\gamma}(\bm{\mc a,\mc b})$ is a $\BMN$-module.

For each $1\lle i\lle k$, $L(\mc a_i,\mc b_i)$ is two dimensional. We set $v_i^+$ to be a nonzero singular vector and $v_i^-=e_{21}v_i^+$. Moreover, we assume that $v_i^+$ are even. Suppose $\C_\gamma$ is spanned by $v_0$. We also set  $v^+=v_1^+\otimes\cdots\otimes v_k^+\otimes v_0^{}$.
\begin{prop}
If $\ve_1\ne \ve_2$, then the  $\BMN$-module $V_{\gamma}(\bm{\mc a,\mc b})$ is irreducible. Moreover, the finite-dimensional irreducible $\BMN$-module $V(\bm\mu(u))$ is almost isomorphic to $V_{\gamma}(\bm{\mc a,\mc b})$.
\end{prop}
\begin{proof}
It follows from the proof of Theorem \ref{thmnontrivial} that the vector $v^+$ is a highest $\ell_{\s,\bm\ve}$-weight vector. Suppose the corresponding $\ell_{\s,\bm\ve}$-weight is $\bm\nu(u)=(\nu_1(u),\nu_2(u))$. Then it follows from Proposition \ref{prop:highest-weight-inver}, Proposition \ref{prop:tensor-product}, and Lemma \ref{lemn2}, that
\begin{align*}
\tl \nu_1(u)=2(\ve_1u+s_1\ve_1+\gamma)\prod_{i=1}^k\frac{(u+s_1\mc a_i)(u+s_1+s_1\mc b_i)}{(u+s_1-s_1\mc a_i)(u+s_1\mc b_i)},\quad \tl \nu_2(u) = 2(\ve_2u+\gamma) \prod_{i=1}^k\frac{u-s_1\mc b_i}{u+s_1\mc b_i}.
\end{align*}
Therefore, by \eqref{P-def} and \eqref{P-factor}, we have
\begin{align*}
 \frac{\tl \nu_1(u)}{\tl \nu_2(u)}=\frac{\ve_1u+s_1\ve_1+\gamma}{\ve_2u+\gamma}&\prod_{i=1}^k\frac{(u+s_1\mc a_i)(u+s_1+s_1\mc b_i)}{(u-s_1\mc b_i)(u+s_1-s_1\mc a_i)}\\&=(-1)^{\deg P(u)+1}\frac{P(u)}{P(-u-s_1)}=\frac{\tl \mu_1(u)}{\tl \mu_2(u)}.
\end{align*}
Thus, it suffices to prove that the  $\BMN$-module $V_{\gamma}(\bm{\mc a,\mc b})$ is irreducible.

We claim that any vector $\eta\in V_{\gamma}(\bm{\mc a,\mc b})$ satisfying $b_{12}(u)\eta=0$ is proportional to $v^+$. We prove the claim by induction on $k$. The case $k=1$ is proved by a direct computation using Lemma \ref{lemn2}. Then we assume that $k\gge 2$. We write any such nonzero vector $\eta$ in the following form,
\[
\eta=\sum_{r=0}^1 (e_{21})^rv_1^+\otimes\eta_r=v_1^+\otimes \eta_0+v_1^-\otimes \eta_1,
\]
where $\eta_0,\eta_1\in L(\mc a_2,\mc b_2)\otimes \cdots\otimes L(\mc a_k,\mc b_k)\otimes\bC_\gamma$. Similar to the proof of Proposition \ref{prop=}, the vector $\eta_1$ must be proportional to $v_2^+\otimes\cdots\otimes v_k^+$. Then taking the coefficient of $v^+$ in $b_{12}(u)\eta=0$, we have
\begin{align*}
0=\frac{u+s_1\mc a_1}{u+s_1\mc b_1}v_1^+\otimes b_{12}(u)\eta_0+t_{12}(u)t'_{22}(-u)v_1^-\otimes b_{22}(u)\eta_1+t_{11}(u)t'_{12}(-u)v_1^-\otimes b_{11}(u)\eta_1.
\end{align*}
Note that by \eqref{eq:def-tl-b} we have $\tl b_{11}(u)=(2u+s_1)b_{11}(u)-s_1b_{22}(u)$ and $\tl b_{22}(u)=2ub_{22}(u)$. We deduce from the above equation and Lemma \ref{lem:b-in-t} that
\begin{align*}
0&=\frac{u+s_1\mc a_1}{u+s_1\mc b_1}v_1^+\otimes b_{12}(u)\eta_0+\frac{2s_1u(\mc a_1+\mc b_1)(\ve_2u+\gamma)}{(u+s_1\mc b_1)(2u+s_1)(u-\gamma)}\prod_{i=2}^k\frac{u-s_1\mc b_i}{u+s_1\mc b_i} v_1^+\otimes \eta_1 \\
&+\frac{2s_1u(\mc a_1+\mc b_1)(u+s_1\mc a_1)(\ve_1u+s_1\ve_1+\gamma)}{(u+s_1-s_1\mc a_1)(u+s_1\mc b_1)(2u+s_1)(u-\gamma)}\prod_{i=2}^k\frac{(u+s_1\mc a_i)(u+s_1+s_1\mc b_i)}{(u+s_1-s_1\mc a_i)(u+s_1\mc b_i)} v_1^+\otimes \eta_1.
\end{align*}
Multiplying both sides by $(2u+s_1)(u-\gamma)\prod_{i=1}^k\big((u+s_1-s_1\mc a_i)(u+s_1\mc b_i)\big)$, we have
\begin{align*}
0&=(u-\gamma)(2u+s_1)(u+s_1\mc a_1)(u+s_1-s_1\mc a_1)v_1^+\otimes \prod_{i=2}^k\big((u+s_1-s_1\mc a_i)(u+s_1\mc b_i)\big)b_{12}(u)\eta_0\\
&+2s_1u(\ve_2u+\gamma)(\mc a_1+\mc b_1)(u+s_1-s_1\mc a_1)\prod_{i=2}^k\big((u+s_1-s_1\mc a_i)(u-s_1\mc b_i)\big) v_1^+\otimes \eta_1 \\
&+2s_1u(\ve_1u+s_1\ve_1+\gamma)(\mc a_1+\mc b_1)(u+s_1\mc a_1)\prod_{i=2}^k\big((u+s_1\mc a_i)(u+s_1+s_1\mc b_i)\big) v_1^+\otimes \eta_1.
\end{align*}
The rest of the proof is parallel to that of Proposition \ref{prop=}. Again, we need the condition that the only possible cancellation in the right hand side of \eqref{P-factor2} is when $l$ is even, then $u+s_1+s_1\mc b_k=u-s_1\mc b_k$.
\end{proof}

\begin{cor}
Suppose $\s=(s_1,s_2)$ and $\bm \ve=(\ve_1,\ve_2)$ are such that $s_1\ne s_2$ and $\ve_1\ne \ve_2$.
\begin{enumerate}
    \item If $\bm\mu(u)$ satisfies \eqref{P-def2}, where $P(u)$ and $P(-u-s_1)$ are relatively prime, then $\dim V(\bm\mu(u))=2^k$, where $k=\big\lfloor\frac{\deg P(u)}{2}\big\rfloor$.
    \item Given $k\in\bZ_{>0}$, let $\gamma,\mc a_i,\mc b_i$, $1\lle i\lle k$, be arbitrary complex numbers such that $\mc a_i+\mc b_i\ne 0$ and set $$P(u)=(u+s_1+\ve_1\gamma)\prod_{i=1}^k\big((u+s_1\mc a_i)(u+s_1+s_1\mc b_i)\big).$$ Then the $\BMN$-module
    \[
    V_{\gamma}(\bm{\mc a,\mc b})=L(\mc a_1,\mc b_1)\otimes \cdots\otimes L(\mc a_k,\mc b_k)\otimes \bC_{\gamma}
    \]
    is irreducible if and only if the greatest common divisor of $P(u)$ and $P(-u-s_1)$ (over $\bC$) is of degree at most 1.
\end{enumerate}
\end{cor}

\section{Classification in higher ranks}\label{sec:classification}
\subsection{Sufficient conditions}
We have the following sufficient condition for $V(\bm\mu(u))$ to be finite-dimensional with arbitrary $\s$ and arbitrary $\bm\ve$.
\begin{thm}\label{thm:suff}
Suppose the highest $\ell_{\s,\bm\ve}$-weight $\bm\mu(u)$ satisfies
\beq\label{eq:in-proof-01}
\frac{\tl \mu_i(u)}{\tl \mu_{i+1}(u)}=\frac{(2\ve_{i}u-\ve_{i}\rho_{i+1}+\varpi_{i+1}+2\gamma)\la_i(u) \la_{i+1}(-u+\rho_{i+1})}{(2\ve_{i+1}u-\ve_{i+1}\rho_{i+2}+\varpi_{i+2}+2\gamma)\la_{i+1}(u)\la_i(-u+\rho_{i+1})},\quad 1\lle i<\ka,
\eeq
where $\gamma\in\bC$ and $\bm\la(u)=(\la_i(u))_{1\lle i\lle \ka}$ is an  $\ell_\s$-weight such that the $\YMN$-module $L(\bla(u))$ is finite-dimensional, then $V(\bm\mu(u))$ is finite-dimensional.
\end{thm}
\begin{proof}
Let $\bm\la(u)=(\la_i(u))_{1\lle i\lle \ka}$ be an  $\ell_\s$-weight such that the $\YMN$-module $L(\bla(u))$ is finite-dimensional. Suppose its highest $\ell_\s$-weight vector is $\xi$. Consider $\BMN$ as a subalgebra as in Proposition \ref{thm:embedding}. Let $\bC_{\gamma}$ be the one-dimensional $\BMN$-module spanned by $\eta_\gamma$ as in Example \ref{eg:1-dim}. Then $L(\bla(u))\otimes \bC_\gamma$ is a $\BMN$-module, see Example \ref{eg:tensor-1}.

It follows from Proposition \ref{prop:tensor-product} and Example \ref{eg:tensor-1} that $\xi$ is a highest $\ell_{\s,\bm\ve}$-weight with $\ell_{\s,\bm\ve}$-weight $\bm\zeta(u)$ such that
\beq\label{eq:in-proof-02}
\frac{\tl \zeta_i(u)}{\tl \zeta_{i+1}(u)}=\frac{(2\ve_{i}u-\ve_{i}\rho_{i+1}+\varpi_{i+1}+2\gamma)\la_i(u) \la_{i+1}(-u+\rho_{i+1})}{(2\ve_{i+1}u-\ve_{i+1}\rho_{i+2}+\varpi_{i+2}+2\gamma)\la_{i+1}(u)\la_i(-u+\rho_{i+1})},\quad 1\lle i<\ka.
\eeq
Let $\mathcal M:=\BMN (\xi\otimes\eta)$. Then $\mathcal M$ is a  highest $\ell_{\s,\bm\ve}$-weight module with highest $\ell_{\s,\bm\ve}$-weight $\bm\zeta(u)$. Moreover $\mc M$ is finite-dimensional as a subspace of $L(\bla(u))$.

Note that the series $f(u)=\mu_{\ka}(u)/\zeta_{\ka}(u)\in 1+u^{-1}\C[[u^{-1}]]$ satisfies $f(u)f(-u)=1$, see \eqref{thm:nontrivial-N}. Denote by $\mc M^{f(u)}$ the $\BMN$-module obtained by pulling back $\mc M$ through the automorphism defined by $b_{ij}(u)\mapsto f(u)b_{ij}(u)$, for $1\lle i,j,\lle \ka$. Comparing \eqref{eq:in-proof-01} and \eqref{eq:in-proof-02}, we see that $\mathcal M^{f(u)}$ is a  highest $\ell_{\s,\bm\ve}$-weight module with highest $\ell_{\s,\bm\ve}$-weight $\bm\mu(u)$. Since $\mc M$ is finite-dimensional, so is $\mathcal M^{f(u)}$. Therefore $V(\bm\mu(u))$ is finite-dimensional.
 \end{proof}

It is very nature to expect that this is also the necessary condition for $V(\bm\mu(u))$ being finite-dimensional.

\begin{conj}\label{conj:main}
If the irreducible $\BMN$-module $V(\bm \mu(u))$ is finite-dimensional, then there exist $\gamma\in \bC$ and an  $\ell_\s$-weight $\bm\la(u)=(\la_i(u))_{1\lle i\lle \ka}$ such that
\begin{enumerate}
    \item the equations \eqref{eq:in-proof-01} hold, and
    \item the $\YMN$-module $L(\bla(u))$ is finite-dimensional.
\end{enumerate}
\end{conj}

We call $\bm\ve$ {\it simple} if there exists at most one $1\lle i<\ka$ such that $\ve_i\ne \ve_{i+1}$. Conjecture \ref{conj:main} is proved in \cite{Molev2002reflection} for the case that $\bm\ve$ is simple and $n=0$ (non-super case). Our main results in this section are to show the conjecture for (1) the case when $n=0,1$ and \textit{$\bm\ve$ is arbitrary}; and (2) the case when $\s$ is the standard parity sequence and $\bm\ve$ is simple. The main obstacle for the super case is that when $\bm s$ is not the standard parity sequence, an explicit criterion for the $\YMN$-module $L(\bla(u))$ to be finite-dimensional is not available, though a recursive criterion can be deduced from \cite{Molev2022odd,Lu2022note}.

\subsection{Reduction lemmas}\label{sec:reduct-lem}
In this subsection, we prepare reduction lemmas which allows us to construct modules of twisted super Yangians of lower ranks from  modules of twisted super Yangians of higher ranks.

For given $\bm s\in S_{m|n}$ and $\bm \ve$, define
\beq
\begin{split}
&\overline{\bm s}=(s_2,s_3,\cdots,s_\ka),\qquad \ \ \ \  \overline{\bm\ve}=(\ve_2,\ve_3,\cdots,\ve_\ka),\\
&\underline{\bm s}=(s_1,s_2,\cdots,s_{\ka-1}),\qquad \underline{\bm\ve}=(\ve_1,\ve_2,\cdots,\ve_{\ka-1}).
\end{split}
\eeq
Then we have twisted super Yangians $\mathscr B_{\ovs,\ove}$ and $\mathscr B_{\uns,\une}$. To distinguish the underlying generating series, we rewrite the series ${b}_{ij}(u)$ as ${b}^\circ_{ij}(u)$ in $\mathscr B_{\ovs,\ove}$ or $\mathscr B_{\uns,\une}$.

Let $V$ be a representation of the twisted super Yangian $\mathscr B_{\bm s,\bm \ve}$. Define
\beq\label{Vover}
\overline V=\{w\in V~|~ b_{1i}(u)w=0,\ 1<i\lle \ka\}.
\eeq
\begin{lem}\label{lem:embedding-over}
The map $\mathscr B_{\ovs,\ove}\to \mathscr B_{\bm s,\bm \ve}$ defined by ${b}^\circ_{ij}(u)\to b_{i+1,j+1}(u)$, $1\lle i,j<\ka$, induces a representation of $\mathscr B_{\ovs,\ove}$ on $\overline V$.
\end{lem}
\begin{proof}
First, one shows that $\overline V$ is invariant under the action of $b_{ij}(u)$ with $2\lle i,j\lle \ka$, i.e. $b_{1i}(u)b_{jk}(v)w=0$ for $2\lle i,j,k\lle \ka$ and $w\in \overline V$. 

Suppose $2\lle i,j,k\lle \ka$ and $w\in \overline V$, then by \eqref{eq:comm-series b}, we have 
\begin{align*}
(u^2-v^2)[b_{1i}(u),b_{jk}(v)]w=(-1)^{|1||i|+|1||j|+|i||j|}\delta_{ij}\sum_{a=1}^\ka b_{1a}(u)b_{ak}(v)w.
\end{align*}
If $i\ne j$, it is clear from above that $b_{1i}(u)b_{jk}(v)w=0$. If $i=j$, then the above equation implies
\[
(u^2-v^2)s_ib_{1i}(u)b_{ik}(v)w=\sum_{a=1}^\ka b_{1a}(u)b_{ak}(v)w.
\]
Note that the RHS is independent of $i$, we find that $s_ib_{1i}(u)b_{ik}(v)w=s_ab_{1a}(u)b_{ak}(v)w$. Hence we further have
\[
(u^2-v^2)s_ib_{1i}(u)b_{ik}(v)w=(m-n)s_ib_{1a}(u)b_{ak}(v)w,
\]
which implies $b_{1i}(u)b_{ik}(v)w=0$.

Then we need to show that the relations are preserved when acting on $\overline V$. This follows immediately from \eqref{eq:comm-series b}.
\end{proof}

Similarly, define
\beq\label{Vund}
\underline{V}=\{w\in V~|~ b_{i\ka}(u)w=0,\ 1\lle i<\ka\}.
\eeq
\begin{lem}\label{lem:embedding-under}
The map $\mathscr B_{\uns,\une}\to \mathscr B_{\bm s,\bm \ve}$ defined by
$$
b_{ij}^\circ(u)\to b_{ij}\Big(u+\frac{s_\ka}{2}\Big)+\delta_{ij}\frac{s_\ka}{2u}b_{\ka\ka}\Big(u+\frac{s_\ka}{2}\Big)
$$ induces a representation of $\mathscr B_{\uns,\une}$ on $\underline{V}$.
\end{lem}
\begin{proof}
First, one proves as in Lemma \ref{Vover} that the subspace $\underline V$ is invariant under the action of $b_{ij}(u)$ with $1\lle i,j<\ka$ and $b_{\ka\ka}(u)$. The verification of the fact that the relations are preserved when acting on $\underline V$ follows from \cite[Theorem 3.1]{Belliard2009nested}.
\end{proof}

Fix $1\lle a<\ka$. Let
\beq\label{star}
\begin{split}
   \s^\star=(s_{a},s_{a+1}),\qquad \bm \s[a]=(s_1,\cdots,s_{a},s_{a+1}),\\
   \bm\ve^\star=(\ve_{a},\ve_{a+1}),\qquad \bm \ve[a]=(\ve_1,\cdots,\ve_{a},\ve_{a+1}).
\end{split}
\eeq
For a $\BMN$-module $V$, define
\beq\label{vstar}
\begin{split}
V^\star=\{\eta\in V\,|\,& b_{ij}(u)\eta=0,\, 1\lle i< a,\, i<j\lle \ka,\\
&b_{kl}(u)\eta=0,\, a+1< l \lle \ka,\, 1\lle k<l\}.
\end{split}
\eeq
\begin{lem}\label{B2reduction}
The map
\[
\mathscr B_{\s^\star,\bm\ve^\star}\to \BMN,\quad b_{ij}^\star(u)\mapsto b_{a+i-1,a+j-1}\Big(u+\frac{\rho_{a+2}}{2}\Big)+\frac{\delta_{ij}}{2u}\sum_{k=a+2}^\ka s_k b_{kk}\Big(u+\frac{\rho_{a+2}}{2}\Big)
\]
induces a representation of $\mathscr B_{\s^\star,\bm\ve^\star}$ on $V^\star$. Moreover, under this map, we have
\[
\tl b_{ii}^\star(u)\mapsto \tl b_{a+i-1,a+i-1}\Big(u+\frac{\rho_{a+2}}{2}\Big),\qquad i=1,2.
\]
\end{lem}
\begin{proof}
The first statement follows from repeatedly applying Lemma \ref{lem:embedding-over} and Lemma \ref{lem:embedding-under}. The second one is obvious.
\end{proof}
Note that if $V$ is finite-dimensional, then, by Lemma \ref{lem:nontrivial}, none of $\overline V$, $\underline V$, and $V^\star$ is trivial.

\subsection{Classification for nonsuper case}
Let $\sigma\in \mathfrak S_{\ka}$. Given $\s$ and $\bm\ve$,
define
\[
\s^{\sigma}=(s_{\sigma^{-1}(1)},\cdots,s_{\sigma^{-1}(\ka)}), \quad \bm\ve^{\sigma}=(\ve_{\sigma^{-1}(1)},\cdots,\ve_{\sigma^{-1}(\ka)}).
\]
Then we have the following natural isomorphisms, which by abuse of notation we denote by $\sigma$ again,
\beq\label{sigmaY}
\sigma:\YMN\to \mathscr Y_{\s^\sigma},\qquad  t_{ij}^\s(u)\mapsto t_{\sigma(i)\sigma(j)}^{\s^\sigma}(u),
\eeq
and
\beq\label{sigmaB}
\sigma:\BMN\to \mathscr B_{\s^\sigma,\bm \ve^\sigma},\qquad b_{ij}^\s(u)\mapsto b_{\sigma(i)\sigma(j)}^{\s^\sigma}(u).
\eeq
Note that the latter one is the same as the one obtained by the restriction of the former.

Fix $1\lle a<\ka$, $\s$, and $\bm\ve$, we shall denote
\begin{align*}
\tl{\bm s}&:=\s^{\sigma}=(s_1,\cdots,s_{a+1},s_a,\cdots,s_\ka), \\
\tl{\bm \ve}&:=\bm\ve^{\sigma}=(\ve_1,\cdots,\ve_{a+1},\ve_a,\cdots,\ve_\ka),
\end{align*}
We shall identify the superalgebra $\YMNtl$ with $\YMN$ via the isomorphism \eqref{sigmaY}, $\BMNtl$ with $\BMN$ via the isomorphism \eqref{sigmaB}. When the underlying parity sequence $\s$ and the sequence $\bm\ve$ are omitted, we implicitly assume that we pick our fixed choice of $\s$ and $\bm\ve$. Then we can rephrase the conditions for a vector being a highest $\ell_{\tl{\s}}$-weight vector $\zeta$ of  $\ell_{\tl{\s}}$-weight $\bm\nu(u)$ as follows,
\beq\label{simpleref}
\begin{split}
&t_{ii}(u)\zeta=\nu_i(u)\zeta,\quad t_{aa}(u)\zeta=\nu_{a+1}(u)\zeta,\quad t_{a+1,a+1}(u)\zeta=\nu_a(u)\zeta,\\
&t_{i,i+1}(u)\zeta=t_{a-1,a+1}(u)\zeta=t_{a+1,a}(u)\zeta=t_{a,a+2}(u)\zeta=0,\quad i\ne a-1,a,a+1.
\end{split}
\eeq

Recall that $L(\bla(u))$ is the irreducible $\YMN$-module of highest $\ell_\s$-weight $\bla(u)$. Suppose $L(\bla(u))$ is finite-dimensional and $s_a\ne s_{a+1}$, then it follows from \cite{Zhang1995reps} that $\la_a(u)/\la_{a+1}(u)$ is a series in $u^{-1}$ as a rational function expanded at $u=\infty$. Let
\beq\label{pqdef}
\frac{\la_a(u)}{\la_{a+1}(u)}=\frac{\mc p(u)}{\mc q(u)},
\eeq
where $\mc p(u)$ and $\mc q(u)$ are relatively prime monic polynomials in $u$ of the same degree. Set
\beq\label{mcqdeg}
\deg \mc q(u)= k.
\eeq
We also need the following
\begin{lem}\label{lem:ell-weight-ref}
Suppose $L(\bla(u))$ is finite-dimensional, then $L(\bla(u))$ contains a unique highest $\ell_{\tl\s}$-weight vector (up to proportionality) of $\bm\nu(u)$, where $\bm\nu(u)$ is given by the following rules,
\begin{enumerate}
    \item if $s_a=s_{a+1}$, then $\bm\nu(u)=\bla(u)$;
    \item if $s_a\ne s_{a+1}$, then  $\nu_i(u)=\la_i(u)$ for $i\ne a,a+1$ and
    \[
    \nu_a(u)=\la_{a+1}(u)\frac{\mc q(u-s_a)}{\mc q(u)},\quad \nu_{a+1}(u)=\la_{a}(u)\frac{\mc p(u-s_a)}{\mc p(u)}.
    \]
\end{enumerate}
\end{lem}
\begin{proof}
Case (1) is probably well known and case (2) follows from the odd reflections of super Yangians, see \cite{Molev2022odd,Lu2022note}.
\end{proof}


\begin{thm}\label{thm:reflection}
Suppose $n=0$. If Conjecture \ref{conj:main} holds true for the case $\bm\ve$, then it also holds true for the case $\bm \ve^\sigma$ for any $\sigma\in\fkS_\ka$.
\end{thm}
\begin{proof}
It suffices to prove it for the case $\sigma=\sigma_a=(a,a+1)$ for any fix $1\lle a<\ka$. We shall use the notations introduced above. Since there is a single choice of $\s$, we shall drop the dependence on $\s$ for the notations. Also, to distinguish $\ell_{\s}$ and $\ell_{\tl\s}$, we use the notation $\ell$ and $\tl \ell$ instead, respectively; see \eqref{simpleref}.

Let $\bm\mu(u)$ be a highest $\ell_{\tl{\bm\ve}}$-weight. Suppose the irreducible $\mathscr B_{\tl{\bm\ve}}$-module  $V(\bm\mu(u))$ is finite-dimensional. Since $\mathscr B_{\tl{\bm\ve}}$ and $\mathscr B_{\bm\ve}$ are isomorphic, the $\mathscr B_{\tl{\bm\ve}}$-module $V(\bm\mu(u))$ is also a finite-dimensional irreducible $\mathscr B_{\bm\ve}$-module. Suppose the highest $\ell_{\bm\ve}$-weight of $V(\bm\mu(u))$ is $\bm\nu(u)$, that is the $\mathscr B_{\tl{\bm\ve}}$-module $V(\bm\mu(u))$ is isomorphic to the $\mathscr B_{\bm\ve}$-module $V(\bm\nu(u))$ if we identify $\mathscr B_{\tl{\bm\ve}}$ with $\mathscr B_{\bm\ve}$. By assumption, Conjecture \ref{conj:main} is true for the case $\bm\ve$, therefore, there exists a highest $\ell$-weight $\bla(u)$ and $\gamma\in\bC$ such that \eqref{eq:in-proof-01} are satisfied and the $\mathscr Y$-module $L(\bla(u))$ is finite-dimensional too. Note that the choice of $\bla(u)$ may not be unique. We pick $\bla(u)$ such that the finite-dimensional irreducible $\mathscr{Y}(\gl_2)$-module $L(\la_a(u),\la_{a+1}(u))$ tensor with the one-dimensional module $\bC_{\gamma}$ restricts to an irreducible $\mathscr B_{\bm\ve^\star}$-module, see Proposition \ref{prop:con-even} and also \eqref{star} for the definition of $\bm\ve^\star$. 

Then the $\scrB_{\tl{\bm\ve}}$-module $V(\bm\mu(u))$ is isomorphic to a subquotient of $L(\bla(u))\otimes \bC_\gamma$. Without loss of generality, we can assume that $\scrB_{\tl{\bm\ve}}$-module $V(\bm\mu(u))$ is indeed a subquotient of $L(\bla(u))\otimes \bC_\gamma$. Let $\xi$ be the $\ell$-weight vector of the module $L(\bla(u))$ and $\tl\xi$ the $\tl\ell$-weight vector. It is well known that the $\mathscr Y(\gl_2)$-module $\mathscr Y(\gl_2)\xi$ (generated by $\xi$), where $\mathscr Y(\gl_2)$ is the Yangian generated by coefficients of $t_{ij}(u)$ with $i,j=a,a+1$, is irreducible and contains the vector $\tl\xi$. Moreover, $\mathscr Y(\gl_2)\xi$ is isomorphic to the finite-dimensional irreducible $\mathscr{Y}(\gl_2)$-module $L(\la_a(u),\la_{a+1}(u))$. Recall our choice that this $\mathscr Y(\gl_2)$-module restricts to an irreducible $\mathscr B_{\bm\ve^\star}$-module. Hence the image of the vector $\tl\xi\otimes \eta_\gamma$ is nonzero in the subquotient $V(\bm\mu(u))\cong V(\bm\nu(u))$. In addition, this image is a highest $\ell_{\tl{\bm\ve}}$-weight vector as $\tl\xi$ is a highest $\ell$-weight vector, see Proposition \ref{prop:tensor-product}. Note that by Lemma \ref{lem:ell-weight-ref} the $\tl\ell$-weight for $\tl\xi$ remains the same as the $\ell$-weight for $\xi$. Hence to compute the $\ell_{\tl{\bm\ve}}$-weight, one only needs to change  $\bm\ve$ to $\tl{\bm\ve}$, completing the proof.
\end{proof}
\subsection{Classification for certain super cases} In this subsection, we obtain classifications of finite-dimensional irreducible representations for certain cases when $n>0$.
\begin{thm}\label{thm:main-super-class}
Suppose  $\s$ is the standard parity sequence and $\bm\ve$ is simple. The $\BMN$ module $V(\bm\mu(u))$ is finite-dimensional if and only if there exists $\gamma\in\C$ and for any $1\lle i<\ka$ such that
\begin{enumerate}
    \item if $s_i=s_{i+1}$, there exists a monic polynomial $P_i(u)$ satisfying $P_i(u)=P_i(-u+\rho_i)$,
\beq\label{eq:gamma-thm}
\frac{\tl \mu_i(u)}{\tl \mu_{i+1}(u)}=\frac{(2\ve_i u-\ve_i\rho_{i+1}+\varpi_{i+1}+\gamma)P_i(u+s_i)}{(2\ve_{i+1} u-\ve_{i+1}\rho_{i+2}+\varpi_{i+2}+\gamma)P_i(u)}.
\eeq
Moreover, if $\ve_i\ne\ve_{i+1}$, then $P_i(u)$ is not divisible by $2\ve_i u-\ve_i\rho_{i+1}+\varpi_{i+1}+\gamma$.
\item if $s_i\ne s_{i+1}$, there exists a monic polynomials $P_i(u)$ satisfying
\[
\frac{\tl \mu_i(u)}{\tl \mu_{i+1}(u)}=\ve_i\ve_{i+1}(-1)^{\deg P_i}\frac{P_i(u)}{P_{i}(-u+\rho_{i+1})}.
\]
\end{enumerate}
\end{thm}
\begin{proof} 
By Theorem \ref{thm:zhang} and Theorem \ref{thm:suff}, it suffices to prove the ``only if" part. Let $V=V(\bm\mu(u))$ and assume that $V$ is finite-dimensional. We proceed by induction on $n$. For the base case $n=2$, it follows from Propositions \ref{prop:iff-even} and \ref{prop:con-even}. Then we assume that $n\gge 3$.

Recall the notation from \S \ref{sec:reduct-lem} and set $\xi$ to be the highest $\ell_{\s,\bm\ve}$-weight vector. Consider the subspace $\overline V$ defined in \eqref{Vover}, then $\overline V$ is a finite-dimensional $\mathscr B_{\ovs,\ove}$-module by Lemma \ref{lem:embedding-over}. Clearly, $\xi\in\overline V$ and $\xi$ is a highest $\ell_{\ovs,\ove}$-weight vector of the $\ell_{\ovs,\ove}$-weight $\overline{\bm\mu}(u)=(\mu_2(u),\cdots,\mu_\ka(u))$. Thus the cyclic span $\mathscr B_{\ovs,\ove}\xi$ is a finite-dimensional highest $\ell_{\ovs,\ove}$-weight with highest $\ell_{\ovs,\ove}$-weight $\overline{\bm\mu}(u)$. In particular, $V(\overline{\bm\mu}(u))$ is finite-dimensional. By induction hypothesis, we conclude that the conditions from the theorem are satisfied by the components of $\overline{\bm\mu}(u)$ (that is for $1<i\lle\ka$) for some $\gamma_1\in\bC$.

Similarly, consider the subspace $\underline{V}$ defined in \eqref{Vund}, then $\underline V$ is a finite-dimensional $\mathscr B_{\uns,\une}$-module by Lemma \ref{lem:embedding-under}. Clearly, $\xi\in\underline V$ and $\xi$ is a highest $\ell_{\uns,\une}$-weight vector with the $\ell_{\uns,\une}$-weight
\[
\underline{\bm\mu}^\circ(u)=\Big(\mu_1\Big(u+\frac{s_\ka}{2}\Big)+\frac{s_\ka}{2u}\mu_{\ka}\Big(u+\frac{s_\ka}{2}\Big),\cdots,\mu_{\ka-1}\Big(u+\frac{s_\ka}{2}\Big)+\frac{s_\ka}{2u}\mu_{\ka}\Big(u+\frac{s_\ka}{2}\Big)\Big).
\]
Let $\tl{\mu}_i^\circ(u)$ be the series associated to $\bm\mu^\circ(u)$ as defined in \eqref{eq:mu-tilde}. Then it is clear that
\[
\tl{\mu}_i^\circ(u)=\tl{\mu}_i\Big(u+\frac{s_\ka}{2}\Big).
\]
By the same argument as in the previous paragraph, we conclude that the conditions from the theorem are satisfied for $1\lle i<\ka-1$ for some $\gamma_2\in\bC$.

Now it suffices to show that we can choose $\gamma_1=\gamma_2$. Recall from \eqref{eq:good1} that if $\ve_i=\ve_{i+1}$, for $1\lle i<\ka$, then
$$
2\ve_i u-\ve_i\rho_{i+1}+\varpi_{i+1}=2\ve_{i+1} u-\ve_{i+1}\rho_{i+2}+\varpi_{i+2}.
$$
Hence the number $\gamma$ only shows up in \eqref{eq:gamma-thm} when $\ve_i\ne \ve_{i+1}$ and $s_i= s_{i+1}$. By Proposition \ref{prop:iff-even}, the pair $(P_i(u),\gamma)$ satisfying \eqref{eq:gamma-thm} is unique in this case. Since $\bm\ve$ is simple, there is at most one $i$ such that $\ve_i\ne \ve_{i+1}$. Therefore, we can always make sure that $\gamma_1=\gamma_2$, completing the proof of the theorem.
\end{proof}

\begin{cor}
Conjecture \ref{conj:main} holds when $\s$ is the standard parity sequence and $\bm\ve$ is simple.
\end{cor}
\begin{proof}
This follows immediately from Theorem \ref{thm:zhang} and Theorem \ref{thm:main-super-class}, see also equation \eqref{eq:good2}.
\end{proof}

\begin{thm}
Conjecture \ref{conj:main} holds for arbitrary $\bm\ve$ when $\s$ is the standard parity sequence and $n=1$.
\end{thm}
\begin{proof}
The proof is similar to that of Theorem \ref{thm:main-super-class} by Theorem \ref{thm:reflection} and induction. Again the point is to argue that we can choose $\gamma_1$ and $\gamma_2$ to be the same as in the proof of Theorem \ref{thm:main-super-class}. However, the condition $n=1$ and Proposition \ref{prop:rank1} show that the choice of $\gamma$ in \eqref{eq:gamma-thm} for $i=\ka-1$ is not important as it can be absorbed into the polynomial $P_{\ka-1}(u)$. 
\end{proof}

\section{Drinfeld functor and dAHA of type BC}\label{sec:schur-weyl}
\subsection{Degenerate affine Hecke algebras} Let $l$ be a positive integer. Following \cite{Chen2014twisted}, we first recall basics about degenerate affine Hecke algebras (dAHA for short) of types A$_l$ and BC$_l$.

Denote by $\fkS_l$ the symmetric group on $l$ elements. Then the wreath product $\mathscr W_l=\Z_2\wr \fkS_l$ is the Weyl group of type BC$_l$. Let $e_1,\cdots,e_l$ be the standard basis of $\bR^l$, then the non-reduced root system of type BC$_l$ consists of the following set of vectors,
\[
\{\pm e_i+e_j, ~ \pm e_i-e_j~|~ 1\lle i\ne j \lle l\} \cup \{\pm e_i,~ \pm 2e_i ~|~ 1\lle i\lle l\}.
\]
For $1\lle i\ne j\lle j$, let $\sigma_{ij}$, $\varsigma_i$ be the reflections about the root vectors $e_i-e_j$ and $e_i$, respectively. Set $\sigma_i=\sigma_{i,i+1}$ for $1\lle i<l$.
\begin{dfn}
For $\vartheta_1\in\bC^\times$ and $l\in\Z_{>0}$, the dAHA $\mathscr{H}_{\vartheta_1}^l$ of type $\mathrm A_l$ is the associative algebra generated by the group algebra $\bC[\fkS_l]$ and $y_1,\cdots,y_l$ with the relations $y_iy_j=y_jy_i$, $1\lle i,j\lle l$, and
\begin{align*}
&\sigma_iy_i-y_{i+1}\sigma_i=\vartheta_1,&1\lle i< l,\\
&\sigma_iy_j=y_j\sigma_i, &j\ne i,i+1.
\end{align*}
\end{dfn}

\begin{dfn}\label{def:dAHA-B}
For $\vartheta_1,\vartheta_2\in\bC^\times$ and $l\in\Z_{>0}$, the dAHA $\mathscr{H}_{\vartheta_1,\vartheta_2}^l$ of type $\mathrm{BC}_l$ is the associative algebra generated by the group algebra $\bC[\scrW_l]$ and $y_1,\cdots,y_l$ with the relations $y_iy_j=y_jy_i$, $1\lle i,j\lle l$, and
\[
\sigma_iy_i-y_{i+1}\sigma_i=\vartheta_1,\qquad \varsigma_ly_i=y_i\varsigma_l,\quad 1\lle i< l,
\]
\[
\varsigma_ly_l+y_l\varsigma_l=\vartheta_2,\qquad \sigma_iy_j=y_j\sigma_i,\quad j\ne i,i+1.
\]
\end{dfn}
The following lemmas are well known, see e.g. \cite[Section 2]{Chen2014twisted}.
\begin{lem}
The subalgebra of $\mathscr{H}_{\vartheta_1,\vartheta_2}^l$ generated by $y_i$, $1\lle i\lle l$, and $\C[\fkS_l]$ is isomorphic to the dAHA $\mathscr{H}_{\vartheta_1}^l$ of type $\mathrm A_l$.\qed
\end{lem}

One has the following natural embeddings,
\begin{align}
&\imath_1: \mathscr{H}_{\vartheta_1}^{l_1} \hookrightarrow\mathscr{H}_{\vartheta_1,\vartheta_2}^l,\quad y_i\mapsto y_i,\quad \sigma_j\mapsto \sigma_j, \quad &1\lle l_1\lle l,\label{i1embed}
\\
&\imath_2: \mathscr{H}_{\vartheta_1,\vartheta_2}^{l_2} \hookrightarrow\mathscr{H}_{\vartheta_1,\vartheta_2}^l,\quad y_i\mapsto y_{i+l-l_2},\quad \varsigma_i\mapsto \varsigma_{i+l-l_2},\quad \sigma_j\mapsto \sigma_{j+l-l_2}, \quad &1\lle l_2\lle l,\nonumber
\\
& \imath_1\otimes \imath_2: \mathscr{H}_{\vartheta_1}^{l_1}\otimes \mathscr{H}_{\vartheta_1,\vartheta_2}^{l_2}\hookrightarrow \mathscr{H}_{\vartheta_1,\vartheta_2}^l,\quad &l_1+l_2\lle l.\label{eq:embedding-hecke}
\end{align}
Note that due to the relation $[\varsigma_i,y_j]=\vartheta_1\sigma_{ij}(\varsigma_i-\varsigma_j)$ for $i<j$, the last embedding $\imath_1\otimes \imath_2$ does not extend to an embedding $\mathscr{H}_{\vartheta_1,\vartheta_2}^{l_1}\otimes \mathscr{H}_{\vartheta_1,\vartheta_2}^{l_2}\hookrightarrow \mathscr{H}_{\vartheta_1,\vartheta_2}^l$.
\begin{lem} [{\cite[Lemma 3.1]{Etingof2009}}]\label{lem:dAHA-B-other}
The algebra $\mathscr{H}_{\vartheta_1,\vartheta_2}^l$ is isomorphic to the algebra generated by elements $\sfy_i$, $1\lle i\lle l$, and by $\bC[\mathscr W]$ with the relations,
\[
\sigma_i\sfy_i=\sfy_{i+1}\sigma_i,\qquad \sigma_i\sfy_j=\sfy_{j}\sigma_i,\qquad j\ne i,i+1,
\]
\[
\varsigma_l\sfy_l=-\sfy_{l}\varsigma_l,\qquad
\varsigma_l\sfy_i=\sfy_i\varsigma_l,\qquad i\ne l,
\]
\begin{align*}
    [\sfy_i,\sfy_j]= \frac{\vartheta_1 \vartheta_2}{2}\sigma_{ij}(\varsigma_j-\varsigma_i)&\,+\frac{\vartheta_1^2}{4}\sum_{\substack{k=1 \\ k\ne i,j}}^l\Big((\sigma_{jk}\sigma_{ik}-\sigma_{ik}\sigma_{jk})\\
    &\, +\sigma_{ik}\sigma_{jk}(\varsigma_i\varsigma_j-\varsigma_i\varsigma_k+\varsigma_j\varsigma_k)-\sigma_{jk}\sigma_{ik}(\varsigma_i\varsigma_j+\varsigma_i\varsigma_k-\varsigma_j\varsigma_k)\Big).
\end{align*}
Moreover, this presentation is related to the one in Definition \ref{def:dAHA-B} by
\[
\sfy_i=y_i-\frac{\vartheta_2}{2}\varsigma_i+\frac{\vartheta_1}{2}\sum_{k=1}^{i-1}\sigma_{ik}-\frac{\vartheta_1}{2}\sum_{k=i+1}^l\sigma_{ik}-\frac{\vartheta_1}{2}\sum_{\substack{k=1 \\ k\ne i}}^l\sigma_{ik}\varsigma_i\varsigma_k.\qedd
\]
\end{lem}
\begin{lem}[{\cite[3.12]{Lusztig1989affine}}]
The center of the dAHA $\mathscr{H}_{\vartheta_1,\vartheta_2}^l$ (resp. $\mathscr{H}_{\vartheta_1}^l$) is generated by the $\fkS_l$-symmetric polynomials in $y_1^2,\cdots,y_l^2$ (resp. $y_1,\cdots,y_l$).\qed
\end{lem}

\subsection{Drinfeld functor for super Yangian}
The symmetric group $\mathfrak S_l$ acts naturally on $V^{\otimes l}$, where the operator $\sigma_{ij}$ for $i<j$ acts as
\beq\label{eq perm}
\mathcal P^{(i,j)}=\sum_{a,b=1}^\ka s_b E_{ab}^{(i)}E_{ba}^{(j)}\in \End(V^{\otimes l}).
\eeq
Here we use the standard notation
\[
E_{ij}^{(k)}=1^{\otimes (k-1)}\otimes E_{ij}\otimes 1^{\otimes (l -k)}\in \End(V^{\otimes l}),\qquad 1\lle k\lle l.
\]
Set
$$
\mc Q^{(k)}=\sum_{i,j=1}^\ka (-1)^{|i||j|+|i|+|j|}E_{ij}^{(k)}\otimes E_{ij}\in \End(V^{\otimes l})\otimes \End(V),\qquad 1\lle k\lle l.
$$

Let $\ve=\pm 1$. Let $M$ be any $\mathscr H_{\vartheta_1}^l$-module. Set
\[
\mc D_\s(M)=M\otimes V^{\otimes l},\quad \mc D_{\s}^{\ve}(M)=\mc D_{\s}(M)/\sum_{i=1}^{l-1}(\mathrm{Im}\, \sigma_i-\ve),
\]
where the symmetric group acts on $\mc D_\s(M)$ by the diagonal action, namely $\sigma_i$ acts on $M\otimes V^{\otimes l}$ as $\sigma_i\otimes \mc P^{(i,i+1)}$ for $1\lle i<l$.

For $\chi,c\in\bC$, define
\[
\mathcal T^{\chi}(u)=\mc T_1^\chi(u)\cdots \mc T^{\chi}_l(u)\in \mathscr H_{\vartheta_1}^l[[u^{-1}]]\otimes \End(V^{\otimes l})\otimes \End(V),
\]
where
\[
\mc T_k^\chi(u)=1+\frac{1}{u-\chi y_k+c}\otimes \mc Q^{(k)},\qquad 1\lle k \lle l.
\]
Then the map $T(u)\mapsto \mc T^{\chi}(u)$ induces an action of $\YMN$ on $\mc D_\s(M)$.

The following statement for the Yangian $\mathrm{Y}(\gl_N)$ case  is well known, see \cite[Proposition 2]{Arakawa1999} and \cite[Theorem 1]{Drinfeld198degenerate}.
\begin{lem}[{\cite[Lemma 4.2]{Lu2021jacobi}}]\label{lem:D-functor-A}
Suppose $\vartheta_1\ne 0$ and $\vartheta_1 \chi=\ve$. Let $M$ be any $\mathscr H_{\vartheta_1}^l$-module. Then the map $T(u)\mapsto \mc T^{\chi}(u)$ induces an action of $\YMN$ on $\mc D_\s^\ve(M)$.
\end{lem}

Therefore, one has a functor $\mc D_\s^\ve$ from the category of $\mathscr H_{\vartheta_1}^l$-modules to the category of $\YMN$-modules. We call the functor $\mc D_\s^\ve$ the {\it Drinfeld functor}. For Schur-Weyl type dualities for superalgebras of type A, see \cite{Sergeev1984tensor,Berele1987hook,Moon2003highest,Mitsuhashi2006schur,Flicker2020affine,Lu2021jacobi,Lu2021gelfand,Kwon2022super,Lu2023schur,Guay2024affine,Shen2025quantum} for more details.

\subsection{Drinfeld functor for twisted super Yangian}
We need the following
\begin{lem}\label{lem:embedding}
For any $\gamma\in\C$, the mapping
\beq\label{eq:emd-b-new}
\varphi: B(u)\to T(u) (G^{\bm\ve}+\gamma u^{-1})T^{-1}(-u)
\eeq
defines a superalgebra homomorphism from the twisted super Yangian $\mathscr B_{\bm s,\bm\ve}$ to the super Yangian $\YMN$.
\end{lem}
\begin{proof}
Similar to Proposition \ref{thm:embedding}, it suffices to show that the matrix $G_{\ve}+\gamma u^{-1}$ satisfies the reflection equation \eqref{eq:reflectG} which it is known in \cite{Arnaudon2004general,Ragoucy2007analytical,Belliard2009nested}.
\end{proof}

For brevity, we set
$$
G^{\bm\ve,\gamma}(u)=G^{\bm\ve}+\gamma u^{-1},\quad \jmath=\frac{m-n}{2}.
$$
Consider the following elements in $\mathscr H^l_{\vartheta_1,\vartheta_2}[[u^{-1}]]\otimes \End(V^{\otimes l})\otimes \End(V)$,
\[
\mc T^{\chi}_k(u)=1+\frac{1}{u-\jmath-\chi y_k}\otimes \mc Q^{(k)},\quad \mc S^{\chi}_k(u)=1-\frac{1}{u+\jmath-\chi y_k}\otimes \mc Q^{(k)},
\]
for $1\lle k\lle l$. By the identity $\mc Q^{(k)}\cdot \mc Q^{(k)} = 2\jmath \mc Q^{(k)}$, we have
\beq\label{eq:t-times=s}
\mc T^{\chi}_k(u)\mc S^{\chi}_k(u)=1.
\eeq

Set
\beq\label{eq:B-chi}
\mc B^\chi(u)=\mc T_1^\chi(u)\cdots \mc T^{\chi}_l(u)G^{\bm\ve,\gamma}(u)\mc S_l^\chi(-u)\cdots \mc S^{\chi}_1(-u)
\eeq
as an element in $\mathscr H_{\vartheta_1,\vartheta_2}^l[[u^{-1}]]\otimes \End(V^{\otimes l})\otimes \End(V)$. Here $G^{\bm\ve,\gamma}(u)$ stands for $1\otimes 1\otimes G^{\bm\ve,\gamma}(u)$.

Given any $\mathscr H_{\vartheta_1,\vartheta_2}^l$-module $M$, we can regard it as an $\mathscr H_{\vartheta_1}^l$-module via the embedding \eqref{i1embed} and we have the $\YMN$-module $\mc D_\s^{\ve}(M)$ if $\vartheta_1\chi =\ve$, by Lemma \ref{lem:D-functor-A}. Moreover,  the action of $\YMN$ on $\mc D_\s^{\ve}(M)$ is given by
$$
T(u)\mapsto \mc T_1^\chi(u)\cdots \mc T^{\chi}_l(u).
$$
Hence it follows from Lemma \ref{lem:embedding} and \eqref{eq:t-times=s} that $B(u)\mapsto \mc B^\chi(u)$ induces an action of $\BMN$ on $\mc D^{\ve}_\s(M)$.

The $\fkS_l$-action on $V^{\otimes l}$ can be extended to $\scrW_l$ by setting the action of $\varsigma_k$ on $V^{\otimes l}$ by multiplication on the $k$-th factor by the matrix $G^{\bm \ve}$. We also write this operator as $G^{\bm\ve}_k$. For an $\mathscr H_{\vartheta_1,\vartheta_2}^l$-module $M$, the group $\scrW_l$ acts on $M\otimes V^{\otimes l}$ by the diagonal action. We further set
\[
\mc D_{\s,\bm\ve}^{\ve}(M)=\mc D_{\s}^\ve(M)/(\mathrm{Im}\,\varsigma_l-\ve).
\]

We shall need the following lemma. Recall from \eqref{eq:def-tl-b} that $\varpi_1=\sum_{a=1}^\ka s_a\ve_a$ and $2\jmath=\sum_{a=1}^\ka s_a$. Set
\begin{align*}
&\mathcal Q_k^{\mathfrak k}=\sum_{i,j:\ve_i=\ve_j}(-1)^{|i||j|+|i|+|j|}E_{ij}^{(k)}\otimes E_{ij},\\
&\mathcal Q_k^{\mathfrak p}=\mathcal Q^{(k)}-\mathcal Q_k^{\mathfrak k}=\sum_{i,j:\ve_i\ne\ve_j}(-1)^{|i||j|+|i|+|j|}E_{ij}^{(k)}\otimes E_{ij}.
\end{align*}
\begin{lem}\label{lem:cgm-thm4.5}
We have
\begin{align*}
&\mathcal Q_k^{\mathfrak k}\mathcal Q_k^{\mathfrak p}+\mathcal Q_k^{\mathfrak p}\mathcal Q_k^{\mathfrak k}=2\jmath \mathcal Q_k^{\mathfrak p},\\
&G^{\bm\ve}(\mathcal Q_k^{\mathfrak k}\mathcal Q_k^{\mathfrak p}-\mathcal Q_k^{\mathfrak p}\mathcal Q_k^{\mathfrak k})=\varpi_1 \mathcal Q_k^{\mathfrak p}.
\end{align*}
\end{lem}
\begin{proof}
The formulas follow from a direct computation.
\end{proof}

The following are the main results of this section.

\begin{prop}\label{prop:Drinfeld-functor-BC}
Let $M$ be any $\mathscr H_{\vartheta_1,\vartheta_2}^l$-module. If $\vartheta_2=\vartheta_1(2\gamma +\varpi_1)$ and $\vartheta_1\chi=\ve$, then the map $B(u)\mapsto \mc B^\chi(u)$ defines a representation of the twisted super Yangian $\BMN$ on the space $\mc D_{\s,\bm \ve}^{\ve}(M)$.
\end{prop}
\begin{proof}
The proof is similar to that of \cite[Theorem 4.5]{Chen2014twisted} by using Lemma \ref{lem:cgm-thm4.5}.
\end{proof}

Therefore, one has a functor $\mc D_{\s,\bm\ve}^\ve$ from the category of $\mathscr H_{\vartheta_1,\vartheta_2}^l$-modules to the category of $\BMN$-modules. Again, we call the functor $\mc D_{\s,\bm\ve}^\ve$ the {\it Drinfeld functor}.

Let $l,l_1,l_2\in\bZ_{\gge0}$ be such that $l=l_1+l_2$. Let $M_1$ be an $\mathscr H_{\vartheta_1}^{l_1}$-module, $M_2$ an $\mathscr H_{\vartheta_1,\vartheta_2}^{l_2}$-module. Set
\[
M_1\odot M_2 =\mathscr H_{\vartheta_1,\vartheta_2}^{l}\otimes_{\mathscr H_{\vartheta_1}^{l_1} \otimes \mathscr H_{\vartheta_1,\vartheta_2}^{l_2} }(M_1\otimes M_2),
\]
see \eqref{eq:embedding-hecke}. Then $M_1\odot M_2$ is an $\mathscr H_{\vartheta_1,\vartheta_2}^{l}$-module and hence $\mc D_{\s,\bm \ve}^\ve(M_1\odot M_2)$ is $\BMN$-module.

Note that $\mc D^\ve_\s(M_1)$ is a $\YMN$-module and $\mc D^\ve_{\s,\bm\ve}(M_2)$ is a $\BMN$-module, thus $\mc D^\ve_\s(M_1)\otimes\mc D^\ve_{\s,\bm\ve}(M_2)$ is a $\BMN$-module induced by the coproduct in Proposition \ref{prop:coproduct}.

\begin{prop}\label{prop:DF-coproduct}
As $\BMN$-modules, we have $\mc D_{\s,\bm \ve}^\ve(M_1\odot M_2)\cong \mc D^\ve_\s(M_1)\otimes\mc D^\ve_{\s,\bm\ve}(M_2)$.
\end{prop}
\begin{proof}
The proof is parallel to that of \cite[Porposition 4.6]{Chen2014twisted}.
\end{proof}

We say that a $\BMN$-module is \textit{of level $l$} if it decomposes as direct sums of submodules over $\mathfrak k$ of $V^{\otimes l}$ as a $\mathfrak k$-module.

\begin{thm}\label{thm:Drinfeld-B}
Let $\vartheta_1,\vartheta_2,\chi,\gamma$ be as in Proposition \ref{prop:Drinfeld-functor-BC} and $p=\# \{i\,|\, \ve_i=1, 1\lle i\lle \ka\}$. If $\max\{p,m+n-p\}< l$, then the Drinfeld functor $\mc D_{\s,\bm\ve}^\ve$ provides an equivalence between the category of finite-dimensional $\mathscr H_{\vartheta_1,\vartheta_2}^l$-modules and the category of finite-dimensional $\BMN$-modules of level $l$.
\end{thm}
We prove the theorem in Section \ref{app:B}.

\begin{thm}\label{thm:Drinfeld-simple}
Let $\vartheta_1,\vartheta_2,\chi,\gamma$ be as in Proposition \ref{prop:Drinfeld-functor-BC}. Let $M$ be an irreducible $\mathscr H_{\vartheta_1,\vartheta_2}^l$-module. Then $\mc D_{\s,\bm \ve}^\ve(M)$ is either 0 or an irreducible $\BMN$-module.
\end{thm}
The theorem is analogous to \cite[Theorem 11]{Arakawa1999} for Yangian $\mathrm{Y}(\gl_N)$, \cite[Theorem 5.5]{Nazarov1999queer} for super Yangian of type $Q_N$, \cite[Proposition 4.8]{Lu2021jacobi} for super Yangian $\YMN$ and \cite[Theorem 4.7]{Chen2014twisted} for twisted Yangian of type AIII. The proof is similar to that of \cite[Theorem 4.7]{Chen2014twisted} with suitable modifications for super case as presented in the proofs of \cite[Theorem 5.5]{Nazarov1999queer} and \cite[Proposition 4.8]{Lu2021jacobi}. Therefore, we shall not provide the details.

\subsection{Proof of Theorem \ref{thm:Drinfeld-B}}\label{app:B}
In this section, we give a proof of Theorem \ref{thm:Drinfeld-B}. The strategy is essentially the same as in \cite{Chen2014twisted}.

Recall that the action of $\BMN$ on $\mc D_{\s,\bm\ve}^\ve(M)$ is induced by the map $B(u)\mapsto \mc B^\chi(u)$, see \eqref{eq:B-chi}. Expanding $\mc B^\chi(u)$ as a series in $u^{-1}$ with coefficients in $\mathscr H_{\vartheta_1,\vartheta_2}^l\otimes \End(V^{\otimes l})\otimes \End(V)$, we find the first 3 coefficients are given by $G^{\bm \ve}$ (understood as $1\otimes 1\otimes G^{\bm \ve}$),
$$
\gamma+1\otimes \sum_{k=1}^l \big(\mc Q^{(k)}G^{\bm \ve}+G^{\bm \ve}\mc Q^{(k)}\big),
$$
\beq\label{eq:B-matrix-2nd}
\begin{split}
& 2\gamma\Big(1\otimes\sum_{k=1}^l \mc Q^{(k)}\Big)+\sum_{1\lle k<r\lle l}(1\otimes \mc Q^{(k)}\mc Q^{(r)})G^{\bm \ve}+\sum_{1\lle r<k\lle l}G^{\bm \ve}(1\otimes \mc Q^{(k)}\mc Q^{(r)})\\
& +\Big(1\otimes\sum_{k=1}^l \mc Q^{(k)}\Big)G^{\bm \ve}\Big(1\otimes\sum_{r=1}^l \mc Q^{(r)}\Big)+\sum_{k=1}^l\Big( G^{\bm \ve}\big((\jmath-\chi y_k)\otimes  \mc Q^{(k)}\big)+\big((\jmath+\chi y_k)\otimes  \mc Q^{(k)}\big)G^{\bm \ve}\Big).
\end{split}
\eeq
We set
$$
\mc B^{\chi}(u)=\sum_{r \gge 0 }\sum_{i,j=1}^\ka (-1)^{|i||j|+|i|+|j|}\sfb_{ij}^{(r)}u^{-r}\otimes E_{ij},
$$
where $\sfb_{ij}^{(r)}\in \mathscr H_{\vartheta_1,\vartheta_2}^l \otimes \End(V^{\otimes l})$. From above, we conclude that
\[
\sfb_{ij}^{(0)}=\ve_i\delta_{ij},\qquad \sfb_{ij}^{(1)}=\gamma\delta_{ij}+s_i(\ve_i+\ve_j)\sum_{k=1}^l 1\otimes E_{ij}^{(k)}.
\]
Before computing $\sfb_{ij}^{(2)}$, we prepare the following lemma.
\begin{lem}\label{lem:drinfeld-cal}
Suppose $\ve_i\ne \ve_j$. Then as operators on $V^{\otimes l}$, we have
\begin{align}
&s_i\sum_{k=1}^l\Big(\sum_{r=1}^{k-1}\sigma_{rk}-\sum_{r=k+1}^{l}\sigma_{rk}\Big)E_{ij}^{(k)}=-\ve_i\Big(\sum_{1\lle k<r\lle l}\mc Q^{(k)}\mc Q^{(r)}G^{\bm\ve}+\sum_{1\lle r<k\lle l}G^{\bm\ve}\mc Q^{(k)}\mc Q^{(r)}\Big)_{ij},\label{eq:lem-bij2-1}\\
&s_i\sum_{k=1}^l\Big(\sum_{\substack{r=1\\ r\ne k}}\sigma_{kr}\varsigma_r\varsigma_k+\varpi_1\varsigma_k\Big)E_{ij}^{(k)}=\ve_i\Big(\sum_{k,r=1}^l\mc Q^{(k)}G^{\bm\ve}\mc Q^{(r)}\Big)_{ij}.\label{eq:lembij2-2}
\end{align}
Here by $(\cdot)_{ij}$, we mean the $(i,j)$-th entry, namely, for $\mathsf G\in \End(V^{\otimes l})\otimes \End(V)$, $$\mathsf G=\sum_{i,j=1}^\ka (-1)^{|i||j|+|j|}(\mathsf G)_{ij}\otimes E_{ij}.$$
\end{lem}
\begin{proof}
Recall from \eqref{eq perm} that $\sigma_{rk}=\mc P^{(r,k)}=\sum_{a,b=1}^\ka s_bE_{ab}^{(r)}E_{ba}^{(k)}$, therefore the left hand of \eqref{eq:lem-bij2-1} is equal to
\[
\sum_{r<k}\sum_{a=1}^\ka s_is_aE_{ia}^{(r)}E_{aj}^{(k)}-\sum_{k<r}\sum_{a=1}^\ka s_is_aE_{ia}^{(r)}E_{aj}^{(k)}.
\]
A straightforward computation implies
\begin{align*}
\sum_{1\lle k<r\lle l}\mc Q^{(k)}\mc Q^{(r)}G^{\bm\ve}+&\sum_{1\lle r<k\lle l}G^{\bm\ve}\mc Q^{(k)}\mc Q^{(r)}\\ &\,=\sum_{i,j,a=1}^\ka \Big(\sum_{r<k}\ve_i+\sum_{k<r}\ve_j\Big)E_{ia}^{(k)}E_{aj}^{(s)}\otimes E_{ij}(-1)^{|i||j|+|i|+|j|+|a|}.
\end{align*}
After interchanging $k$ and $r$ and using $\ve_j=-\ve_i$, one obtains \eqref{eq:lem-bij2-1}.

Similarly, the left hand side of \eqref{eq:lembij2-2} is equal to
\[
\sum_{\substack{k,r=1\\ r\ne k}}^l \sum_{a=1}^\ka s_is_a\ve_a\ve_iE_{ia}^{(r)}E_{aj}^{(k)}+\varpi_1\sum_{k=1}^l s_i\ve_iE_{ij}^{(k)},
\]
while we also have
\begin{align*}
\sum_{k,r=1}^l\mc Q^{(k)}G^{\bm\ve}\mc Q^{(r)}=\sum_{\substack{k,r=1\\ r\ne k}}\sum_{i,j,a=1}^\ka \ve_aE_{ia}^{(k)}E_{aj}^{(r)}&\otimes E_{ij}(-1)^{|i||j|+|i|+|j|+|a|}\\+&\sum_{k=1}^l \sum_{i,j,a=1}^\ka \ve_a E_{ij}^{(k)}\otimes E_{ij}(-1)^{|i||j|+|i|+|j|+|a|}.
\end{align*}
Now \eqref{eq:lembij2-2} follows from  $\varpi_1=\sum_{a=1}^\ka s_a\ve_a=\sum_{a=1}^\ka (-1)^{|a|}\ve_a$.
\end{proof}
Note that
\beq\label{eq:lembij2-3}
G^{\bm\ve}\mc Q^{(k)}+\mc Q^{(k)}G^{\bm\ve}=\sum_{i,j=1}^\ka(\ve_i+\ve_j)\mc Q^{(k)}.
\eeq
It follows from \eqref{eq:B-matrix-2nd}, \eqref{eq:lembij2-3}, and Lemma \ref{lem:drinfeld-cal} that if $\ve_i\ne \ve_j$, then
\begin{align*}
s_i\sfb_{ij}^{(2)}&= 2\gamma \sum_{k=1}^l E_{ij}^{(k)}-\ve_i\ve\sum_{k=1}^l\Big(\sum_{r=1}^{k-1}\sigma_{rk}-\sum_{r=k+1}^{l}\sigma_{rk}\Big)\otimes E_{ij}^{(k)}\\
&\quad  +\ve_i\ve\sum_{k=1}\Big(\sum_{\substack{r=1\\ r\ne k}}^l\sigma_{kr}\varsigma_r\varsigma_k+\varpi_1\varsigma_k\Big)\otimes  E_{ij}^{(k)}-2\chi\ve_i\sum_{k=1}^l y_k\otimes  E_{ij}^{(k)}\\
&= -2\ve_i\sum_{k=1}^l\Big(\chi y_k+\frac{\ve}{2}\sum_{r=1}^{k-1}\sigma_{rk}-\frac{\ve}{2}\sum_{r=k+1}^{l}\sigma_{rk}-\frac{\ve}{2}\sum_{\substack{r=1\\ r\ne k}}^l\sigma_{kr}\varsigma_r\varsigma_k-\frac{\ve}{2}\big(\varpi_1+2\gamma\big)\varsigma_k\Big)\otimes  E_{ij}^{(k)}.
\end{align*}
Therefore, if we suppose further that $\vartheta_2=\vartheta_1(2\gamma+\varpi_1)$ and $\ve=\vartheta_1\chi$, we have
\begin{align*}
s_i\ve\vartheta_1\sfb_{ij}^{(2)}
=-2\ve_i \sum_{k=1}^l \sfy_k\otimes E_{ij}^{(k)},
\end{align*}
see Lemma \ref{lem:dAHA-B-other} and cf. \cite[Equation (4.1)]{Chen2014twisted}.

The rest of the proof is similar to the one that is outlined in \cite[proof of Theorem 4.3]{Chen2014twisted}. We shall omit the details. The $\widetilde J(E_{ij})$ there should be replaced by $\sfb_{ij}^{(2)}$ as the $J$-presentations of (twisted) super Yangians are not discussed here. The fact that the tensor space $V^{\otimes l}$ decomposes as a direct sum of irreducible modules over $\mathfrak k\times \mathscr W_l$ follows from the proof of \cite[Theorem 5.8]{Shen2025quantum}. A precise decomposition parallel to the one discussed in \cite[Introduction]{Ariki1995schur} (namely one only needs to change $W_{\underline{\la}}$ to the $\mathfrak k$-module associated to the tuple $\underline{\la}$) can be deduced from a standard approach as in \cite[Theorem 3.11]{Cheng2009dualities} or \cite{Ariki1995schur}, see also \cite[Chapter II]{Kerber1971reps}. This decomposition and the condition $\max\{p,\ka-p\}<l$ make sure that $\mc D^{\ve}_{\s,\bm\ve}(M)$ is nonzero if $M$ is nonzero.
\bibliographystyle{amsalpha}
\bibliography{all}

\end{document}